\documentclass[onefignum,onetabnum]{siamart190516}
\usepackage{lipsum}
\usepackage{amsfonts}
\usepackage{graphicx}
\usepackage{epstopdf}
\usepackage{algorithmic}
\ifpdf
  \DeclareGraphicsExtensions{.eps,.pdf,.png,.jpg}
\else
  \DeclareGraphicsExtensions{.eps}
\fi

\newsiamremark{remark}{Remark}
\newsiamremark{hypothesis}{Hypothesis}
\crefname{hypothesis}{Hypothesis}{Hypotheses}
\newsiamthm{claim}{Claim}

\headers{DDSM for DOT}{Jiahua Jiang, Yi Li and Ruchi Guo}

\title{
Learn an index operator by CNN for solving diffusive optical tomography: a deep direct sampling method
}

\author{Jiahua Jiang\thanks{School of Information Science and Technology, ShanghaiTech University
  (\email{jiangjh@shanghaitech.edu.cn}),}
  \and Yi Li\thanks{School of Information Science and Technology, ShanghaiTech University
  (\email{liyi3@shanghaitech.edu.cn}).}
\and Ruchi Guo\thanks{Department of Mathematics, University of California, Irvine
  (\email{ruchig@uci.edu}).}
}

\usepackage{amsopn}

\ifpdf
\hypersetup{
  pdftitle={Learn an index operator by CNN for solving diffusive optical tomography: a deep direct sampling method},
  pdfauthor={Jiahua Jiang, Yi Li and Ruchi Guo}
}
\fi
\usepackage{geometry}
\allowdisplaybreaks[4]

\usepackage{subcaption}

\usepackage{bbm,mathabx}
\usepackage{array,multirow,booktabs} % for fancier tables
\usepackage{psfrag, graphicx,enumitem,indentfirst,float,geometry,color,graphicx,subcaption}

\makeatletter
\def\BState{\State\hskip-\ALG@thistlm}
\makeatother

\newcommand{\bs}{\boldsymbol}

\begin{document}

%\title{A  sampling method for solving electrical impedance tomography with limited data}
%\title{Construct deep neural networks based on direct sampling methods for electrical impedance tomography}
%\author{Ruchi Guo\thanks{
%Department of Mathematics, University of California, Irvine, ADDRESS, USA. Emails: {\tt{\{guo.1778\}@osu.edu}}. R.~Guo {was} partially supported by XXX.
%}
%\and
%Jiahua Jiang\thanks{
%Department of Mathematics, Virginia Tech, ADDRESS, USA. Emails: {\tt{\{jiahua\}@vt.edu}}. R.~Guo {was} partially supported by XXX.
%}
%}

%\thanks{blah}
%\nodate
\date{}
\maketitle

\begin{abstract}
In this work, we investigate the diffusive optical tomography (DOT) problem in the case that limited boundary measurements are available. Motivated by the direct sampling method (DSM) proposed in \cite{chow2015direct}, we develop a deep direct sampling method (DDSM) to recover the inhomogeneous inclusions buried in a homogeneous background. In this method, we design a convolutional neural network (CNN) to approximate the index functional that mimics the underling mathematical structure. The benefits of the proposed DDSM include fast and easy implementation, capability of incorporating multiple measurements to attain high-quality reconstruction, and advanced robustness against the noise. Numerical experiments show that the reconstruction accuracy is improved without degrading the efficiency, demonstrating its potential for solving the real-world DOT problems. %Both Theoretical results as well as numerical evidence justify these findings. 
\end{abstract}

\section{Introduction}
Diffusive optical tomography (DOT) is a promising imaging technique with many clinical applications, such as screening for breast cancer and the development of cerebral images \cite{2003CulverChoeHolbokeZubkovDurduran,2003DehghaniBrianBrooksby,2004GuZahngBartlettSchutz}. As its mechanism, the input flux of near-infrared (NIR) photons is used to illuminate the body and the output flux is measured on the surface of the body. In this process, chromophores in the NIR window such as oxygenated and deoxygenated hemoglobin, water, and lipid, are abundant in the body tissue, and a weighted sum of their contributions gives different absorption coefficients \cite{2005Choe}. Then the measured pairs of the fluxes provide some information for detecting and reconstructing the optical properties inside the body by creating images of the distribution of absorption coefficients inside the body. Let the absorption medium occupy an open bounded connected domain $\Omega\subseteq\mathbb{R}^2$ with a piecewise $C^2$ boundary. The governing equation for the radiant intensity in NIR light denoted by $I(x,\upsilon)$ is described by the radiative transfer equation (RTE) for the physically measured $I$ at the mesoscale:
\begin{equation}
\label{RTE_0}
\upsilon\cdot \nabla I(x,\upsilon) + ( \mu + \mu_s )I(x,\upsilon) - \mu_s \int_{\mathbb{S}^{d-1}}f(x,\upsilon,\upsilon') I(x,\upsilon') d\upsilon' =  q(x,\upsilon),
\end{equation} 
where $\upsilon(x)$ is the outward unit normal at $x$ to $\partial\Omega$, $\mu_s$ is the scattering coefficients, and $f$ is the scattering phase function. 

%In this work, we develop a deep direct sampling method (DDSM) to solve diffusive optical tomography (DOT) problems. Let us first describe the considered forward model. 
%Now let us describe the considered forward model. 

Assuming the scattering coefficients  $\mu_s$ is known in advance, by the calculation procedure in \cite{2002HeinoSomersalo} that introduces an energy fluency of $I$ in terms of its integral on a unit sphere $\mathbb{S}^{d-1}$, the RTE in \eqref{RTE_0} induces the following elliptic diffusion-reaction equation \eqref{dot_model} which serves as our forward problem. Let $D$ be a subdomain of $\Omega$ representing the inhomogeneous inclusions. The absorption coefficient of the media in $\Omega$ is described by a nonnegative function $\mu \in L^{\infty}(\Omega)$, while $\mu_0$ is the absorption coefficient of the homogeneous background medium and the support of $\mu-\mu_0$ occupies the subdomain $D$. We consider the DOT model that the potential $u\in H^1(\Omega)$ representing the photon density field satisfies the following equation:
\begin{align}
\label{dot_model}
- \triangle u_{\omega} + \mu u_{\omega} &= 0 ~~~~ \text{in} ~~ \Omega, \\
\frac{\partial u_{\omega}}{\partial \mathbf{ n}} & =  g_{\omega} ~~~ \text{on} ~~ \partial \Omega, ~~~ \omega = 1,2,\cdots,N,
 \end{align}
where $g_{\omega}\in H^{-1/2}(\partial\Omega)$ denotes the surface flux along $\partial\Omega$,  and $f_{\omega} = u_{\omega}|_{\partial\Omega}$ is the measurement of the surface potential on the boundary. In this paper, we consider the Neumann boundary value problem for simplicity, and the proposed algorithm can be also applied to the Robin boundary value problem. The overarching goal of the DOT problem is to recover the geometry of the inhomogeneous inclusions $D$ from the $N$ pairs of potential-flux data $(g_{\omega},f_{\omega})_{\omega=1}^N$, referred as Cauchy data pairs. It is noted that the considered model can be viewed as a special case of the more general one  \cite{2004Bal} which has more complicated jump conditions. 
The inverse procedure in the DOT problem can be described by the well-known Neumann-to-Dirichlet (NtD) mapping 
 \begin{equation}
\label{n2t_map}
\Lambda_{\mu} : H^{-1/2}(\partial\Omega) \rightarrow H^{1/2}(\partial\Omega), ~~~ g \mapsto u|_{\partial \Omega}, 
\end{equation}
where $u$ is the trace of the solution of \eqref{dot_model} according to the boundary condition $g\in H^{-1/2}(\partial\Omega)$. It is known that the NtD mapping preserves all the information needed to recover $\mu$ \cite{2004Bal}; see also the discussion in Section \ref{sec:DDSM}. However, approximating the whole NtD mapping requires a large number of Cauchy data pairs, which hinders the application in many practical situations. To circumvent the application limitation, in this work, we focus on developing the reconstruction algorithm that only requires a reasonably small number of Cauchy data pairs. 
%we shall restrict ourselves to the case of using reasonably small number of Cauchy data pairs for reconstruction.

%There have been many efforts devoted to solving DOT problems over the past decades. 
Over the past decades, much effort has been rewarded with many promising developments of solving the DOT problem. One type of the widely used algorithms are iterative methods. The augmented Lagrange method \cite{2005AbdoulaevRenHielscher} and the shape optimization methods \cite{2009ZacharopoulosSchweigerKolehmainen,2006ZacharopoulosArridge} reconstruct the inhomogeneous inclusions by minimizing a functional measuring the observed and simulated data.  To alleviate the onerous computational costs caused by the necessity of using a fine mesh, the multigrid methods  \cite{2002OhMilstein,2001YeBoumanWebbMillane} have been proposed to resolve this issue. Another standard approach formulates an integral function involving the Green's function corresponding to background absorption coefficient, and use the Born type iteration \cite{2000YaoPeiWangBarbour,1999YeWebbMillaneDownar} to solve this integral equation until convergence. Nonetheless, the major issue of these approaches is that a large number of the forward PDEs need to be solved in the iterative process, which is time-consuming and computationally infeasible in many practical situations such as three-dimensional  problems. Thus, for ultrasound-modulated DOT, the works in \cite{2014AmmariBossyGarnierNguyenSeppecher,2013AmmariGarnierNguyenSeppecher,2014AmmariNguyenSeppecher} provide the hybrid method that is capable of reconstructing the parameters in known inclusions with only a few iterations.

 Another popular group of methods for solving the DOT problem are non-iterative methods. The joint sparsity methods \cite{2011LeeKimBreslerYe,2013LeeYe} have been investigated to solve the aforementioned integration equation by a non-iterative approach in which an ill-conditioned matrix needs to be inverted and some suitable preconditioners or regularization are demanded. The well-known factorization methods described in \cite{2004Bal,2001Cheney,1998Kirsch} compute the spectral information of the boundary data mapping and determine the inclusion shape by checking the convergence property of a series involving the spectral data. Based on a nonlinear Fourier transform,  the D-bar methods in \cite{2014TamminenHoopLassas,2017TamminenTarvainenSiltanen} reconstruct the absorption coefficient by solving a boundary integral equation. Nevertheless, the severe ill-posedness of the original inverse problem will manifest itself in the computation of the boundary integral equations, which brings strenuousness in computation. This is improved  in \cite{1997KlibanovLucasFrank,1997KlibanovLucasFrankIOP} by linearizing the reconstruction problem and reducing it to a well-posed boundary-value problem for a coupled system of elliptic equations such that it can be efficiently solved. Besides these classical approaches, vast development of deep neural networks-based algorithms has been undergone for solving the DOT problem, for instance, using CNN to learn the nonlinear photon scattering physics and reconstruct the 3D optical anomalies \cite{yoo2019deep, yooadeep}, and FDU-Net that is able to fast recover the irregular-shaped inclusions with accurate localization for breast DOT \cite{deng2021deep}, and some detailed surveys can be found in the recent monologue \cite{applegate2020recent,yedder2018deep}.
% the reconstruction problem is linearized and reduces to a well-posed boundary-value problem for a coupled system of elliptic equations such that it can be efficiently solved. 
 
%The proposed neural network (NN) is based on a recently developed direct sampling method (DSM) \cite{chow2015direct} for DOT. The fundamental idea of DSM is to construct certain index function that is able to indicate the shape of inclusions. In the construction, a special duality product is designed to operate on the boundary data and some probing functions regarded as the dual of the Green’s function for the original forward problem. The DSM has been proven to be highly robust, effective and efficient. See also the applications of DSM to electrical impedance tomography \cite{chow2014direct}, inverse scattering \cite{ito2012direct,ito2013direct}, moving potential reconstruction \cite{chow2018time} and inversion of Radon transform for computed tomography \cite{2020ChowHanZou}.

Recently, direct sampling methods (DSM) arise as very appealing non-iterative strategies to solve many geometric inverse problems.
%Among the non-iterative strategies for solving the DOT problem, one of the most appealing methods is 
 %The proposed neural network (NN) is based on a recently developed direct sampling method (DSM) \cite{chow2015direct} for DOT.
The critical component of DSM is to construct a certain index function indicating the shape and location of the inclusions through a duality product operating on the boundary data and some probing functions. The DSM has been proven to be highly robust, effective and efficient for many inverse problems, such as electrical impedance tomography (EIT) \cite{chow2014direct}, inverse scattering \cite{ito2012direct,ito2013direct}, moving potential reconstruction \cite{chow2018time} and the inversion of Radon transform for computed tomography \cite{2020ChowHanZou}. For the DOT problem, the original DSM \cite{chow2015direct} designs an elegant index function format for the case of a single measurement pair. However, only a single measurement results in insufficient accuracy of the reconstruction in more complicated cases, which stymies its application in practical scenarios of medical imaging. Indeed, the closed form of explicit index function for multiple Cauchy data pairs or complex-shaped domain may become very intricate and deriving it with conventional mathematical approaches will be very challenging. 

On the one hand, deep learning recently has become an alternative approach to canonical mathematical derivation. On the other hand, it is worth mentioning that the index function of DSM can be regarded as a mapping from a specially generated data manifold to the inclusion distribution. The data manifold consists of data functions built up by solving the forward problem with background coefficient, and this underling mathematical structure serves as our crucial guideline to design the neural networks (NN) such that the complicated nonlinear mapping can be learned by data. In our previous work for EIT \cite{2021GuoJiang}, %we have found and shown that this guideline is very useful for designing suitable NN to EIT problems, and the corresponding data-driven index function can successfully capture the essential structure of the true index function, smooth out the noise and outperform the conventional DSM. 
we have found and shown that the corresponding data-driven index function can successfully capture the essential structure of the true index function, smooth out the noise and outperform the conventional DSM.

%For both EIT and DOT problems \cite{chow2015direct,chow2014direct}, the original DSM yields the elegant formula of the index function designed based on mathematical intuition and a single pair of Cauchy data. However, such simple explicit index function \cite{chow2015direct,chow2014direct} is unable to incorporate adequate information of multiple data pairs and resulting in insufficient accuracy of the reconstruction in complicated cases, which stymies its application in practical scenarios of medical imaging. 
%Moreover, the form of explicit index function for multiple Cauchy data pairs or complex-shaped domain may become very intricate and deriving it with conventional mathematical approaches will be very challenging. 

%Mathematically, it would be difficulty if not impossible to efficiently express or approximate the mapping from the data function to the inclusion distribution. 
%The approach of the original DSM based on mathematical intuition yields an elegant formula of the index function for a single Cauchy data pair for both EIT and DOT \cite{chow2015direct,chow2014direct}. 
%Such a simple explicit closed form may not be achievable for multiple data pairs in complicated cases, which hinders its application in practical situation of medical imaging. 

To address this barrier from the original DSM, in this work we develop a novel deep direct sampling method (DDSM) for solving the DOT problem based a convolutional network (CNN).
%by designing the CNN to construct the index function (approximating the true index function by learning it through the data).
%consider using DNNs to construct the index function, that is, approximating the true index function by learning it through big data
%believe DNNs provide an alternative approach to approximate the true index function by learning it through big data, which is our major motivation to develop a novel deep direct sampling method (DDSM) for solving the DOT problem. 
The main benefits of the proposed DDSM contain:
\begin{itemize}
\item It is easy to implement for 2D and particularly 3D cases which challenge many conventional approaches.
\item The index function in the DDSM is able to incorporate multiple Cauchy data pairs which enhances the reconstruction quality and robustness with respect to noise. 
%breaking the accuracy bottleneck and significantly enhances the reconstruction quality and robustness with respect to noise. 
\item The DDSM inherits the high efficiency of DSM  \cite{chow2015direct}, which is benefitted from its offline-online decomposition structure. %Although the training process is more costly than the DSM due to the optimization process with large number of parameters, it only needs to be done once in the offline stage. 
For given measurements, the reconstruction is based on the fast evaluation of the CNN-based index function in the online stage, which has almost the same speed as the original DSM. 
\item Similar to DSM, constructing data functions by solving elliptic equations smoothes out the noise on the boundary such that the resulting NN is highly robust against the noise. 
%Thus, our proposed algorithm is competent to smooth out the noise and handle spatially-variant noise in the Cauchy data. 
%\item Only a few PDEs corresponding to the Cauchy data with the background constant absorption coefficient need to be solved to generate the data functions, which makes the implementation highly efficient even in three dimensional cases.
\item It is capable of incorporating very limited boundary data points to achieve satisfactory reconstruction.
%The limitation in the availability of the boundary data points doesn't hamper the reconstruction quality, which is demonstrated numerically in Section \ref{sec:num}. 
\item It can yield adequate reconstruction even if the absorbing coefficients of materials to be recovered are significantly different from those used for training, which is of much practical interest as only very rough guess is needed.
%The training process in the DDSM doesn't reply on the value of the absorbing coefficients of the training data, thus the DDSM can successfully reconstruct the inclusions that have very different values of the absorbing coefficients. 
\end{itemize}
%For solving DOT problems, to address this barrier, we believe deep neural networks (DNNs) provide an alternative approach to approximate the true index function by learning it through big data, which serves the major motivation for the proposed deep direct sampling method (DDSM). 
%In this work, the proposed DDSM inherits some advantages of DSM. For example, the computation of DDSM does not require iteratively solving the forward problems; instead only a few PDEs corresponding to the Cauchy data with the background constant absorption coefficient need to be solved to generate the data functions. It makes the implementation of highly efficient even in 3D. In addition, similar to DSM, these data functions smooth out the noise on the boundary and become very smooth inside the domain such that the resulting NN is highly robust to noise. It is also easy to be implemented based on those recently developed highly sophisticated deep learning packages. 

The rest of the paper is organized as follows. In Section \ref{sec:DSM}, we briefly review the original DSM. The development of our proposed DDSM and its theoretical justification are introduced in Section  \ref{sec:DDSM}. The numerical experiments  to validate the advantages of the DDSM are provided in Section  \ref{sec:num}. Finally, Section \ref{sec:conclusions} summarizes the research findings. 

%This article consists of 5 additional sections. In the next section, we briefly review the original DSM. In Section \ref{sec:DDSM}, we develop our DDSM and provide its theoretical justification. In the last section, we present a group of numerical experiments to validate the advantages of the DDSM.

%The crucial idea is to incorporate a data function that is constructed by solving the forward problem with background coefficient and the boundary data, and this data function is able to capture the fundamental mathematical structure of the true index function and smooth out the noise. Then a fundamental question is how to explicitly express or efficiently approximate the mapping from the data function to the inclusion distribution. The approach of the original DSM is based on certain mathematical intuition, which yields an elegant closed form of the index function for a single Cauchy data pair for both EIT and DOT \cite{chow2015direct,chow2014direct}. However, the classical derivation may become much more complicated or even unachievable for multiple data pairs. This difficulty hinders its application in practical situation of medical imaging. We believe the neural network and deep learning  provides an alternative approach to approximate this mapping by learning it through big data. 

\section{Direct Sampling Methods}
\label{sec:DSM}

In this section, we briefly review the conventional DSM proposed by Chow et al. in \cite{chow2015direct} for DOT that will serve as the framework guiding us to design suitable neural networks. The main idea is to approximate an index function satisfying 
\begin{equation}
\label{index_fun}
\mathcal{I}(x) = 
\begin{cases}
      & 1~ \text{if} ~ x\in D, \\
      & 0~ \text{if} ~ x \in \Omega\backslash\overline{D},
\end{cases}
\end{equation}
which depends on the given Cauchy data. Since the key structure of the index function in \cite{chow2015direct} assumes only a \textit{single} pair of Cauchy data, we follow this assumption in this section.
%In Section \ref{sec:DDSM}, we shall see the proposed DDSM can easily handle multiple data pairs.

We begin with introducing a family of probing functions $\eta_x(\xi)$ with $x\in\Omega$ on $\partial\Omega$ which are the fundamental ingredients for both the theory of uniqueness of DOT \cite{2004Bal} and the DSM  \cite{chow2015direct}.
%theoretically determining the shape of $D$ \cite{2004Bal} and the algorithm of DSM \cite{chow2015direct}. 
Consider the following diffusion equation with homogeneous background medium absorption coefficient $\mu_0$:
\begin{equation}
\label{bg_proj_1}
- \triangle w_x + \mu_0 w_x = \delta_x ~~~~ \text{in} ~ \Omega; ~~~ w_x = 0 ~~~ \text{on}~ \partial\Omega,
\end{equation}
where $\delta_x(\xi)$ is the delta function associated with $x\in\Omega$. For a fixed point $x \in \Omega$, the probing function $\eta_{x}$ is defined as the surface flux of $w_x$ over $ \partial\Omega$,
\begin{equation}
\label{prob}
\eta_x(\xi) := \frac{\partial w_x(\xi)}{\partial \mathbf{ n}}, ~~~ \xi\in \partial\Omega.
\end{equation}
%In addition, being one important structure of the index functions, the following duality product are introduced in \cite{chow2015direct}:
With the following duality product
\begin{equation}
\label{dual_prod}
\langle \phi, \psi \rangle_{\partial\Omega,s} : = \int_{\partial\Omega} (-\triangle_{\partial\Omega})^s\phi \, \psi ~ ds ~~~~ \text{on} ~H^{2s}(\partial\Omega)\times L^2(\partial\Omega),
\end{equation}
%with the duality product and the probing functions, the index function can be then defined as
the index function can be defined as 
\begin{equation}
\label{dsm_index_fun}
\mathcal{I}(x) = \frac{ \langle \eta_x, f - \Lambda_{\mu_0}g \rangle_{\partial\Omega,s} }{|\eta_x|_Y} 
\end{equation}
where $|\cdot|_Y$ is an algebraic function of seminorms in $H^{2s}(\partial\Omega)$, and a typical choice in \cite{chow2015direct} is $|\cdot|_Y = |\cdot|^{1/2}_{H^1(\partial\Omega)}|\cdot|^{3/4}_{L^2(\partial\Omega)}$. In the DSM for various geometric inverse problems \cite{chow2015direct,chow2014direct,ito2013direct}, this duality product structure plays an important role in index functions since it can effectively remove the errors contained in the data $f - \Lambda_{\mu_0}g$ by the high smoothness of $\eta_x$ near $\partial\Omega$.

Note that the format of the index function in \eqref{dual_prod} may not be computationally effective since probing functions $\eta_x$ change with respect to $x$. Therefore, an alternative characterization of the index function is derived in \cite{chow2015direct} which is the foundation of our neural network. Let $\varphi$ denote the solution of the following equation also with only the background absorption coefficient 
\begin{equation}
\label{bg_proj_2}
- \triangle \varphi + \mu_0 \varphi = 0 ~~~~ \text{in} ~ \Omega; ~~~ \varphi = -(-\triangle_{\Gamma})^s (f-\Lambda_{\mu_0}g)~~~ \text{on}~ \partial\Omega.
\end{equation}
Then the index function in \eqref{dsm_index_fun} can be effectively computed through $\varphi$. To show the relationship between the duality product and $\varphi$, we provide a brief derivation for $s=0$. Based on the equation \eqref{bg_proj_1}, the definition of duality product in \eqref{prob} and Green's formula, we have
%The equation \eqref{bg_proj_1}, the definition \eqref{prob} and Green's formula yield
\begin{equation}
\begin{split}
\label{dsm_index_fun_1}
\langle \eta_x, f - \Lambda_{\mu_0}g \rangle_{\partial\Omega,0} &= \int_{\partial\Omega} \eta_x( f - \Lambda_{\mu_0}g )ds  = -\int_{\partial \Omega} \frac{\partial w_x}{\partial \mathbf{ n}} \varphi ds \\
& = \int_{\partial\Omega} w_x\frac{\partial \varphi}{\partial \mathbf{ n}} - \frac{\partial w_x}{\partial \mathbf{ n}} \varphi ds  = \int_{\Omega} w_x \triangle \varphi - \varphi \triangle w_x dx \\
& = \int_{\Omega} w_x (\mu_0 \varphi) - \varphi \triangle w_x dx = \varphi(x).
\end{split}
\end{equation} 
The same result can be obtained for $s=1$ in \cite{chow2015direct}. With \eqref{dsm_index_fun} and \eqref{dsm_index_fun_1}, the index function can be equivalently rewritten as
%Thus, with this  the index function can be equivalently rewritten as
\begin{equation}
\label{dsm_index_fun_2}
\mathcal{I}(x) = \frac{ \varphi(x) }{|\eta_x|_Y}.
\end{equation}

It is highlighted that the index function depends on the Cauchy data only through the function $\varphi$ while the value of $|\eta_x|_Y$ is only based on the geometry of $\Omega$. This mathematical structure guides us to set $\varphi$ as the input of the networks designed to approximate the index function, which will be detailed in the next section. Given the importance of $\varphi$, we shall call it the Cauchy difference function in our following discussion. Note that $\varphi$ is readily solvable and only needs to done once from \eqref{bg_proj_2} for a single pair of Cauchy data $(f,g)$. Since there is only one PDE to solve in the reconstruction procedure, the cost is much lower than the optimization-based methods. Our proposed DDSM inherits this advantage, namely only $N$ PDEs are required to be solved for $N$ Cauchy data pairs.

Besides $\varphi$, the index function in \eqref{dsm_index_fun_2} is determined by the probing functions $\eta_x$ computed from \eqref{bg_proj_1} and \eqref{prob}. However, the evaluation of $\eta_x$ requires solving \eqref{bg_proj_1} for every $x$ that needs to be located, which maybe inefficient. To address this issue, the explicit formulas are provided in \cite{chow2015direct}  for some specific domains; for instance, the probing function corresponding to a unit circle is
\begin{equation}
\label{prob_cir}
\eta_x(\xi) = \frac{1}{2\pi} \sum_{n\in\mathbb{Z}} \frac{J_n(i\sqrt{\mu_0}r_x)}{J_n(i\sqrt{\mu_0})} e^{in(\theta_x-\theta_{\xi})}, ~~~ \xi\in\mathbb{S}^1,
\end{equation}
where $J_n$ are the Bessel functions, and $(r_x,\theta_x)$ is the polar coordinate for a point $x$. However, it has not escaped our notice that such explicit formulas may not be available for all type of geometries. Different from the analytical typed index function \eqref{dsm_index_fun} in the conventional DSM, the index function in the proposed DDSM is represented by a neural network and learned from data, which is not in a closed form but capable of handling more complicated and practical problems. 
%For the proposed DDSM, we shall see that it can be readily and naturally used in any shaped domain. 
%Finally, we mention that, 
%Different from the conventional methods with closed forms, the analytical structure of the index function in \eqref{dsm_index_fun} as well as the quantity $|\cdot|_Y$ do not appear explicitly in the proposed DDSM, while the index function of DDSM is represented by a neural network and learned through massive data.  

\section{Deep Direct Sampling Methods}
\label{sec:DDSM}

% add the comparison of different ways to add noise to emphasize the necessarily of including multiple pairs of cauchy data

Despite the successful application of the original DSM, we believe some aspects can be further improved with the recently developed DNN techniques, for which conventional mathematical derivation may face some difficulties. 

%This is where the recently developed deep neural network and data driven approaches can exploit their advantages.

First, the index function format above may only involve a {\em single} Cauchy data pair, which may limit its accuracy. Moreover our numerical results show that including one pair of Cauchy data is robust for the spatially-invariant noise used in \cite{2015ChowItoLiuZou}. But when the noise become highly spatially variant for boundary data points, for example independent and identically Gaussian distribution illustrated in the left plot of Figure \ref{fig: noise}, the reconstruction may not be robust as shown by the numerical results in Figure \ref{cir-4-comp_noise_N1} , which brings out the importance and necessarity of including multiple measurements. 
%In \cite{2015ChowItoLiuZou}, including one pair of Cauchy data is robust for the case that every boundary points has spatially-invariant noise. Note that this is not a practical assumption, although the noise of the boundary points may follow some certain distribution, for example Gaussian distribution illustrated in Figure \ref{fig: noise}. Our numerical results in Figure \ref{cir-4-comp_noise_N1} indicates that the index function learned from one pair of Cauchy data is not capable to handle to reconstruction from the noisy data, which brings out the importance and necessary of including multiple pairs of Cauchy data. 
However, it has been unclear so far how to theoretically develop an explicit index function with a closed form that systematically incorporate multiple Cauchy data pairs through canonical mathematical derivation, though some basic strategies can be applied such as average, maximum or product of each individual index function. Second, even for the case of a single measurement, the format of the index function $\mathcal{I}(x)$ may not be the optimal approximation to the true index function, for example the empirical choice of the tuning parameter $s$ and norm $| \cdot |_{Y}$. We believe this is where the DNN models can exploit their advantages to replace some theoretical derivation by data driven approaches such that the more optimal index function can be obtained.% based on a large number of data. 

%So it motivates us to use deep neural network (DNN) models to learn the index functions through a large number of data since DNNs are able to mimic the human’s learning process based on physical data.

\begin{figure}[h!]
\centering
\includegraphics[width = 0.32\textwidth]{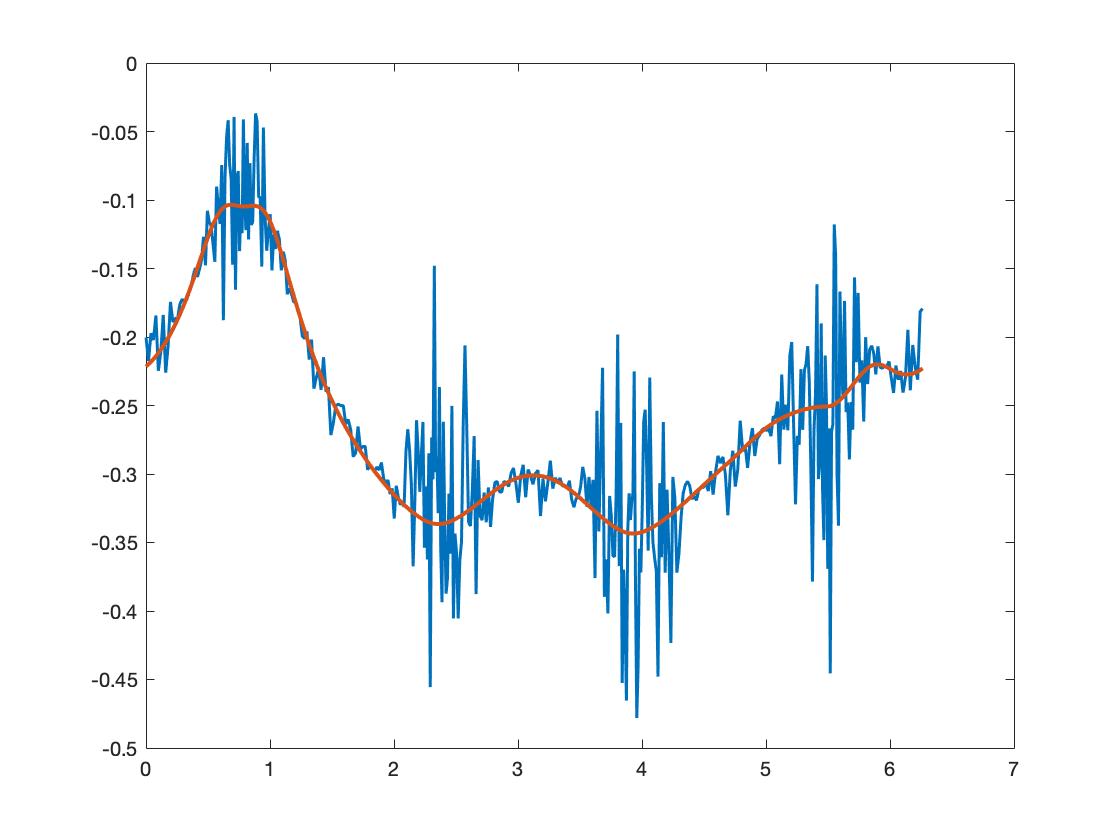}
\includegraphics[width = 0.32\textwidth]{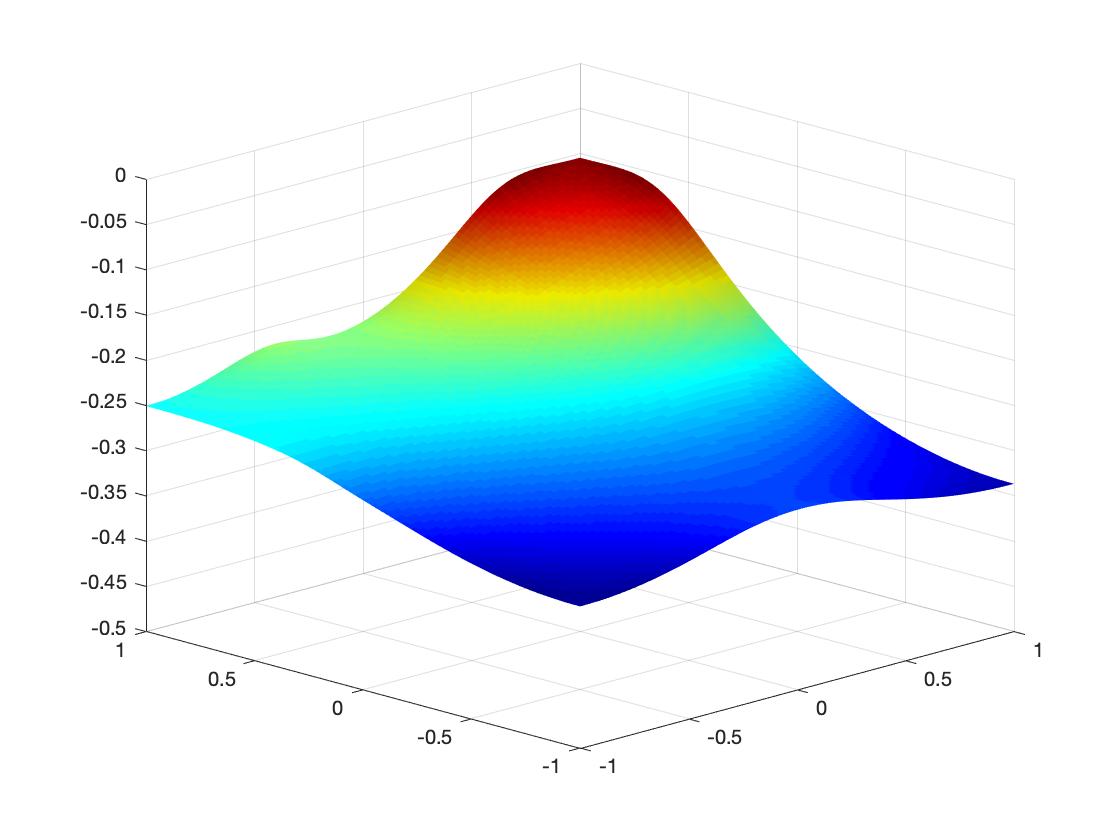}
\includegraphics[width = 0.32\textwidth]{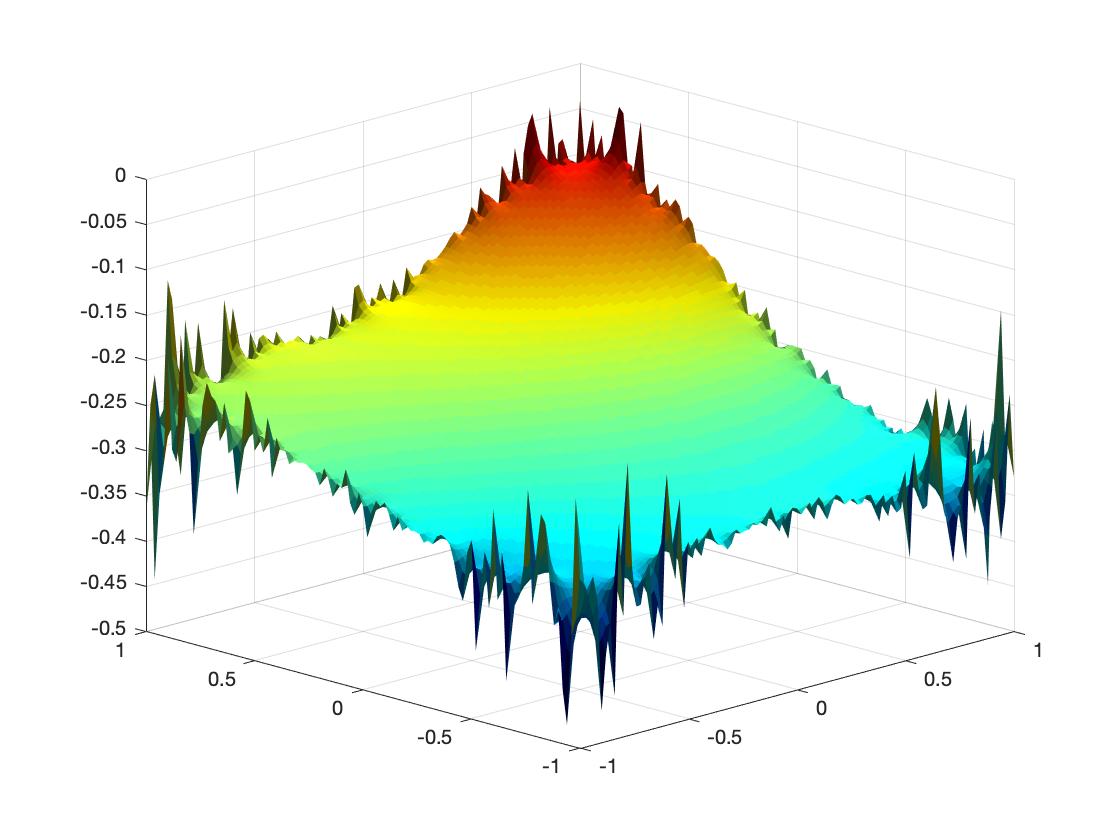}
\caption{The boundary data with independent and identically Gaussian noise (left), the generated $\varphi(x)$ without noise (middle), $\varphi(x)$ with noise (right).}
\label{fig: noise}
\end{figure}

Therefore, to enhance the performance of DSMs based the aforementioned aspects, in this work, we propose a {\em Deep Direct Sampling Method (DDSM)} by mimicking the underling mathematical structure suggested by the DSM. For simplicity,  we mainly discuss the two-dimensional case and the method can be readily and naturally extended to the three-dimensional case. We have implemented the method for both the 2D and 3D DOT problems where the numerical examples and reconstruction results will be provided in Section \ref{sec:num}. 
%associating with the experiments of two-dimensional cases to demonstrate the performance of the DDSM. 

\subsection{Neural Network Structure}

%which is theoretically verified in Section \ref{proof}.

We note that the index function (\ref{dsm_index_fun_2}) suggests the existence of a non-linear mapping or operator from $\varphi$ to the location of $x$, namely whether it is inside or outside of the inclusions. We shall see in Subsection \ref{proof} that this operator theoretically exists if the NtD mapping is given, i.e., all the Cauchy data pairs are available, but  theoretically deriving an explicit index function that is able to fully incorporate the information we have such as multiple Cauchy data pairs, is not easy. This motivates us to design a DNN to approximate this operator from {\em Cauchy difference functions $\{\varphi^{\omega}\}_{\omega=1}^{N}$} to the inclusion distribution,
%The goal of the DNN is to approximate the nonlinear functional(or operator) by training the mapping from {\em Cauchy difference functions $\{\varphi^{\omega}\}_{\omega=1}^{N}$} to the inclusion distribution, 
where $\varphi^{\omega} (1 \le \omega \le N)$ is the solution of (\ref{bg_proj_2}) with the boundary value formed by the $\omega-$th pair of Cauchy data $f_{\omega} - \Lambda_{\sigma_0}g_{\omega}$, namely
\begin{equation}
\label{eq_multiphi}
- \triangle \varphi^{\omega} + \mu_0 \varphi^{\omega} = 0 ~~~~ \text{in} ~ \Omega; ~~~ \varphi^{\omega} = -(-\triangle_{\Gamma})^s (f_{\omega} - \triangle_{\partial\Omega}g_{\omega})~~~ \text{on}~ \partial\Omega.
\end{equation}
Then, the  input of the DNN is designed as $(x, \varphi^1, \dots, \varphi^{N})$. Using these data functions as the input is one of the main difference between our proposed DDSM and and the learning-based approaches reviewed in the literature, which  brings quite some advantages. First, it extends the boundary data to the domain interior that mimic images and features, for instance the input can be treated as a $(N+2)$-channel image, and the nonlinear operator to be trained can be viewed as semantic image segmentation process \cite{chen2017deeplab, ioffe2015batch} partitioning a digital image into multiple segments (set of pixels) based on two characteristics: inside or outside the inclusions. The relationship is listed in table \ref{tab: dsm-dot} for readers from various background. This analogy provides us with some well-established tools such as the architecture of the CNN that achieves state-of-the art image segmentation in many datasets \cite{ ioffe2015batch}, for designing the DNN that is combined with the DSM for solving the DOT problem. 
%for example it naturally prompts us to hybridize CNN and DSM. 
The new index function  is assumed to take the form 
\begin{equation}
\mathcal{I} = \mathcal{F}_{\text{CNN}}(x, \varphi^1, \dots, \varphi^N)~:~~ [H^1(\Omega)]^{2N+2} \rightarrow L^2(\Omega),
\end{equation}
where $ \mathcal{F}_{\text{CNN}}$ is a function expressed by CNN with parameters to be trained. This structure agrees with the theory discussed in Subsection \ref{proof}. Note that in this framework $\mathcal{I}$ is a functional or operator from one Sobolev space to another, so in the following discussion we shall call it an index functional, which is one of the main differences from the original DSM. In addition, a remarkable feature of generating $\varphi$ as the inputs by solving the elliptic type PDEs  is that the noise can be smoothed out at the boundary, namely they are highly smoothed inside the domain. For example, in Figure \ref{fig: noise}, $\varphi$ on the right has rough behavior only near the boundary but very smooth inside the domain while $\varphi$ without noise is shown in the middle. 
%comparable with its counterpart without noise in the middle plot.

\begin{table}[h!]
\centering
\begin{tabular}{|c|c|c|}
\hline
        & DSM for DOT & Image Segmentation \\
        \hline
 $x$  & mesh point in $\Omega$ & pixel \\ 
 \hline
 $\phi^{\omega}(x)$ & Cauchy difference functions & image features (channels) \\  
 \hline
 $\mathcal{I}$ & index of inclusion distribution & dense prediction \\
 \hline
\end{tabular}
\caption{Relationship between DSM for DOT problems and image segmentation problems}
\label{tab: dsm-dot}
\end{table}

Now we proceed to describe the detailed structure of the proposed CNN. For simplicity, we first assume that $\Omega$ has rectangular shape and leave the general situation for later discussion about the implementation details. Then we suppose that $\Omega$ is discretized by a $n_1 \times n_2$ Cartesian grid where $n_1$ and $n_2$ are corresponding to the direction of $x_1$ and $x_2$ respectively. Based on the previous explanation, the input of the CNN is a $3$D matrix (a 4D tensor for 3D problems). For example, for an inclusion sample with $N$ Cauchy difference functions $\{ \varphi^{\omega}\}_{\omega=1}^{N}$ generated from $\{ g_{\omega}, f_{\omega}\}_{\omega = 1}^N$, the corresponding input denoted by $z_{\text{in}} \in \mathbb{R}^{n_x \times n_y \times (N+2)}$ is a stack of $N+2$ matrices in $\mathbb{R}^{n_1 \times n_2}$, where the first two slices are formed by spatial coordinates $x_1$ and $x_2$ respectively, and the rest $N$ slices are the numerical solutions $\varphi^1(x), \dots, \varphi^N(x)$ computed at the Cartesian grid points. The pictorial elucidation of $z_{\text{in}}$ is shown in Figure \ref{fig: CNN}.

\begin{figure}[h!]
\centering
\includegraphics[width = \textwidth]{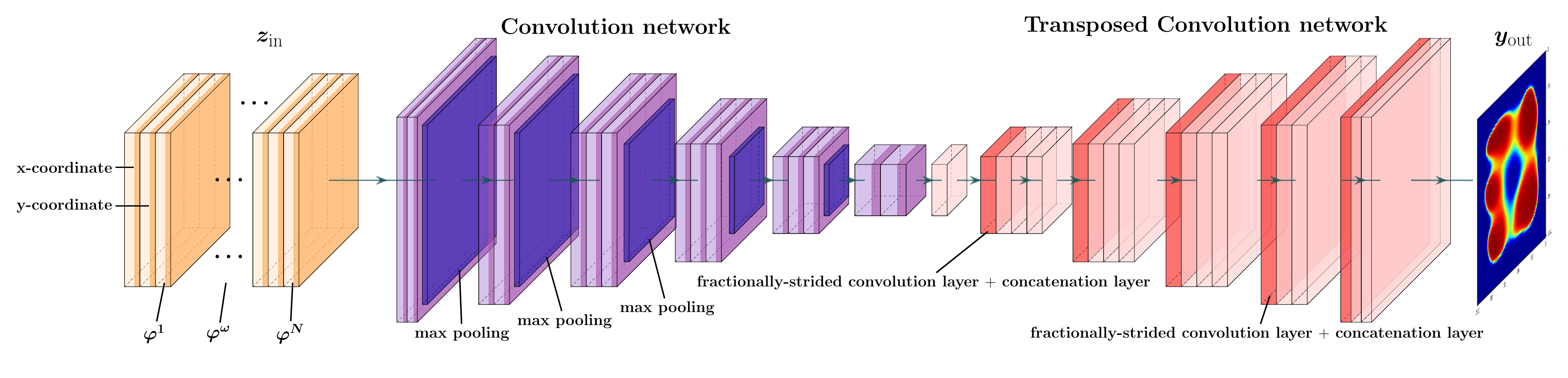}
\caption{The structure of CNN for 2D DOT problem. The input is a $3$D matrix in $ \in \mathbb{R}^{n_x \times n_y \times (N+2)}$ and the output is a matrix in $\mathbb{R}^{n_1 \times n_2}$.}
\label{fig: CNN}
\end{figure}

The overall configuration of the propose CNN consists of two parts -- convolution and transposed convolution networks as illustrated in Figure \ref{fig: CNN}. The convolution network corresponds to feature extractor that transforms the input to multidimensional feature representation,
%The proposed CNN architecture is composed of convolution networks and transposed convolution networks, where the detailed configuration is illustrated in Figure \ref{fig: CNN}. 
%The convolution part consists of several blocks, 
%where each block includes convolution layers, activation layers and max-pooling layers,
%Here, the max-pooling layers mainly help in highlighting the most prominent features of the input image by selecting the maximum element from each non-overlapping subregions. 
where the $i-$th block can be represented as 
\begin{equation}
\phi^{c}_{i}(z_{\text{conv}}) = \mathcal{M}(\kappa(\beta(W_{\text{conv}} * z_{i} + \bs{b}_{\text{conv}}))),
\label{eq:cnn1}
\end{equation}
in which $z_i$ is the input image, $\mathcal{M}$ is the max-pooling layer that mainly help in highlighting the most prominent features of the input image by selecting the maximum element from each non-overlapping subregions, $W_{\text{conv}}$ is the convolution filter for the 2$D$ convolutional layer,  $\kappa$ indicates the activation function, $\beta$ is batch normalization layer \cite{ioffe2015batch}, $*$ denotes the convolution operation, and $b_{\text{conv}}$ refers the bias. The transposed convolution network is more like a  backwards-strided convolution that produces the output by dilating the extracted feature with the learned filter, where the $i-$th block can be expressed as
%The transposed convolution part contains several blocks as well, where each one includes a transposed convolution to extrapolate the output of the convolution part to a large size image (higher resolution). The $i-$th block can be expressed as 

\begin{equation}
\phi^{t}_{i}(z_{\text{trans}}) = \mathcal{C}(\kappa(\beta(\mathcal{T}(z_{i} , W_{\text{trans}}, \bs{b}_{\text{trans}})))),
\label{eq:cnn2}
\end{equation}
where $\mathcal{T}$ refers a fractionally-strided convolution layer, $ W_{\text{trans}}$ and $\bs{b}_{\text{trans}}$ are the corresponding transposed convolutional filter and the bias, $\mathcal{C}$ is a concatenation layer, and other notations are the same as (\ref{eq:cnn1}). $\Theta$ is defined as all the unknown parameters to be learned in training, which includes the convolution and  transposed convolutional filters, as well as the biases. In this work, we consider two activation functions: the ReLu and sigmoid
\begin{equation}
\kappa(z) = \max\{0, z\} ~~~ \text{and} ~~~~ \kappa(z) = \frac{1}{1+ e^{-z}}.
\end{equation}
It is known that choosing the ReLu can significantly enhance the sharpness of reconstruction. But our experience suggests that only using ReLu may not yield satisfactory reconstructions sometimes for the situation that inclusions are far away from the training data set. Nevertheless, the numerical results show that using the sigmoid activation in the last layer may improve the reconstruction for those situations.
%So we use the sigmoid function in last layer of the neural network to reduce the sharpness but increase the extrapability, and the numerical results suggest the reconstruction is indeed improved.
%Our experience suggests that choosing the ReLu as the activation function except the last block, 
%\begin{equation}
%\kappa(z) = \max\{0, z\}.
%\end{equation}
%For the last block, the sigmoid function 
%\begin{equation}
%\kappa(z) = \frac{1}{1+ e^{-z}},
%\end{equation}
%is chosen as the activation function. 

Then the entire CNN model can be represented as 
\begin{equation}
\bs{y}_{\text{out}} = \phi^{t}_{N_t} \circ  \dots \circ  \phi^{t}_{1} \circ \phi^{c}_{Nc} \circ \dots \circ  \phi^{c}_{1}(z_{\text{in}}),
\label{eq:cnn3}
\end{equation}
where the output $\bs{y}_{\text{out}}$ is a $n_1 \times n_2$ matrix which is supposed to approximate an inclusion distribution, i.e., the entire index function values on the domain $\Omega$. Here the complicated DNN function can approximate the index functional or operator, namely
$$
\phi^{t}_{N_t} \circ  \dots \circ  \phi^{t}_{1} \circ \phi^{c}_{Nc} \circ \dots \circ  \phi^{c}_{1} \approx \mathcal{I},
$$
where, again, the input is not the data at individual points but on the entire domain. As mentioned before, different from the conventional DSM, there is no closed form for the DNN-based index functional, but those parameters to be trained may better approximate the true index operator for the current available data.

Let $S$ be the total number of training samples. To measure the accuracy of the CNN model (\ref{eq:cnn3}), we employ the mean squared error (MSE) as the loss function 
\begin{equation}
\mathcal{L}_{\text{loss}}(\Theta) = \frac{1}{S}\sum_{\ell = 1}^{S}(\bs{y}_{\text{out}}(z_{\text{in}}^{\ell}) - \mathcal{I}^{\ell})^{\top}(\bs{y}_{\text{out}}(z_{\text{in}}^{\ell}) - \mathcal{I}^{\ell}),
\label{eq:lossfun}
\end{equation}
where $\mathcal{I}^{\ell}$ is the true distribution (index values) corresponding to the $\ell-$th inclusion sample, $z_{\text{in}}^{\ell}$ denotes the input image that is related to Cauchy difference functions $\{\varphi^{(\ell, \omega)}\}_{\omega=1}^{N}$ for the $\ell-$th inclusion sample. To reduce the infeasibility of gradient descent  algorithm for large datasets, stochastic gradient descent (SGD) is implemented to find the minimization of the loss function (\ref{eq:lossfun}), for which we omit the details here.

\subsection{Unique Determination}
\label{proof}

Note that the proposed DNN model implicitly assumes that inclusion distribution can be uniquely determined by those data functions $\{\varphi^{\omega}\}_{\omega=1}^N$. In this subsection, we show that this is indeed true for $N\rightarrow \infty$ and $\mu_0>0$, $\mu\ge \mu_0$ in $D$, i.e., the boundary data are all known in the Sobolev space $H^{-1/2}(\partial\Omega)\times H^{1/2}(\partial\Omega)$.
%provide some theoretical justification for the proposed DDSM from the perspective that the data functions $\phi^{\omega}$ can indeed uniquely determine the inclusion shape by assuming they are all known in the Sobolev space $H^{-1/2}(\partial\Omega)$ and $\mu_0>0$, $\mu\ge \mu_0$ in $D$. 
Let us recall some known results at first. Define the difference between the operators $\Lambda_{\mu}$ and $\Lambda_{\mu_0}$ as
\begin{equation}
\label{diff_Lam_1}
\widetilde{\Lambda}_{\mu} = \Lambda_{\mu} - \Lambda_{\mu_0}~ : ~ H^{-1/2}(\partial\Omega) \rightarrow H^{1/2}(\partial\Omega).
\end{equation}
According to \cite{2004Bal}, based on the assumption that $\mu_0>0$ and $\mu\ge \mu_0$ in $D$, it is known that $\widetilde{\Lambda}_{\mu}$ is a self-adjoint and positive operator, and thus it admits the eigenpairs $(\lambda_\omega,\nu_\omega)$, $\omega=1,2...$, with $\lambda_\omega>0$, such that
\begin{equation}
\label{diff_Lam_2}
\widetilde{\Lambda}_{\mu} \nu_\omega = \lambda_\omega \nu_\omega, ~~~ k=1,2... \,,
\end{equation}
where $\nu_\omega$ form an orthogonal basis of the space $H^{-1/2}(\partial\Omega)$. Letting $\mathcal{R}(\widetilde{\Lambda}_{\mu}^{1/2})$ be the range of the operator $\widetilde{\Lambda}_{\mu}^{1/2}$ and by the factorization theory for DOT \cite{2004Bal}, it is known that the information of $\mathcal{R}(\widetilde{\Lambda}_{\mu}^{1/2})$ together with the probing functions can be used to determine the inclusion distribution. Here we shall employ this theory to show that the data functions can also uniquely determine the inclusion distribution.

\begin{theorem}
\label{thm_dfs}
Given a collection of orthonormal basis functions $\{g_{\omega}\}_{\omega=1}^{\infty}$ in $H^{-1/2}(\partial\Omega)$, let $\{g_{\omega},\Lambda_{\mu} g_{\omega}\}_{\omega=1}^{\infty}$ be the corresponding Cauchy data pairs and let $\{\varphi^{\omega}\}_{\omega=1}^{\infty}$ be the generated data functions. Then the inclusion can be uniquely determined by the data functions $\{\varphi^{\omega}\}_{\omega=1}^{\infty}$.
\end{theorem}
\begin{proof}
According to \cite{2004Bal}, we know that $x\in D$ if and only if the corresponding probing function $\eta_x(\xi)\in\mathcal{R}(\widetilde{\Lambda}_{\mu}^{1/2})$. By the Picard's criterion, it is further equivalent to the convergence of the following sequence
\begin{equation}
\label{thm_dfs_1}
\mathcal{S}(x; \{\nu_\omega\}_{\omega=1}^{\infty} ) = \sum_{\omega=1}^{\infty} \frac{(\eta_x,\nu_{\omega})^2_{\partial\Omega}}{| \lambda_{\omega}| } < \infty.
\end{equation}
Let us first focus on $\{g_{\omega}\}_{\omega=1}^{\infty}$ being exactly the eigenfunctions $\{\nu_\omega\}_{\omega=1}^{\infty}$ in \eqref{diff_Lam_2}. Then, using 
$
\nu_{\omega} = \lambda^{-1}_{\omega} \widetilde{\Lambda}_{\mu}\nu_{\omega}
$ 
and the identity in \eqref{dsm_index_fun_1}, i.e., using the integration by parts, we have
\begin{equation}
\begin{split}
\label{thm_dfs_2}
(\eta_x, \nu_{\omega})_{\partial\Omega} &= \lambda^{-1}_{\omega} (\eta_x, \widetilde{\Lambda}_{\mu}\nu_{\omega} )_{\partial\Omega} \\
& = \lambda^{-1}_{\omega} (\eta_x, (\Lambda_{\mu} - \Lambda_{\mu_0})\nu_{\omega} )_{\partial\Omega} = \lambda^{-1}_{\omega}  \varphi^{\omega}(x)
\end{split}
\end{equation}
where $\varphi^{\omega}$ are the data functions corresponding to $g_{\omega}=\nu_{\omega}$ solved from \eqref{bg_proj_2} with $s=0$. Then, combining \eqref{thm_dfs_1} and \eqref{thm_dfs_2}, we have
\begin{equation}
\label{thm_dfs_3}
\mathcal{S}(x; \{\nu_\omega\}_{\omega=1}^{\infty} ) = \sum_{\omega=1}^{\infty} \frac{ |\varphi^{\omega}(x)|^2 }{| \lambda_{\omega}|^3 }.
\end{equation}
Since $|\lambda_{\omega}| = \|\varphi_{\omega}\|_{L^2(\partial\Omega)}/\| \nu_{\omega} \|_{L^2(\partial\Omega)}$, the inclusion distribution is determined by the convergence property of the sequence in \eqref{thm_dfs_3} which is further determined by $\varphi^{\omega}$. Next, for general orthonormal basis $\{g_{\omega}\}_{\omega=1}^{\infty}$ as the basis, following the argument in \cite[Theorem 3.8]{2013AnagnostopoulosCharalambopoulos}, we can express $\{\nu_\omega\}_{\omega=1}^{\infty}$ by the expansion of $\{g_{\omega}\}_{\omega=1}^{\infty}$ which is then used to express $\{\varphi^{\omega}\}_{\omega=1}^{\infty}$ corresponding to $\{g_{\omega}\}_{\omega=1}^{\infty}$. Plugging it in \eqref{thm_dfs_3} we have the new series to determine the inclusion shape that only depends on $\{\varphi^{\omega}\}_{\omega=1}^{\infty}$.
\end{proof}

\begin{remark}
According to the theory above, we can conclude that the whole uniqueness does not depend on the value of $\mu_0$ and $\mu$, which only requires they are distinguished from each other. Indeed, this theoretical property in some certain sense can manifest itself in the proposed DDSM since the nice reconstruction can be obtained for the different values of $\mu$ even if they are significant from those used for training.
\end{remark}

\begin{remark}
In our work for EIT \cite{2021GuoJiang}, a similar data function is generate, but the input of the fully neural network is the gradient of the Cauchy difference function. 
%but it is its gradient that is used as the input to the DNN. 
This difference is theoretically supported by the format of $\varphi$  in the generated series \eqref{thm_dfs_3}.
%has its theoretical foundation that is the format of $\varphi$ in the generated series \eqref{thm_dfs_3}.
\end{remark}

\subsection{Implementation Details}
In order to perform the convolution operator,  the discretization of $\Omega$ needs to be chosen as Cartesian grids which are natural for rectangular domain. Note that the data functions $\varphi^{\omega}$ are solved from the equations merely with  the background absorption coefficient, so there is no need to require the mesh to align with the interface, and a simple Cartesian mesh may yield satisfactory numerical solutions. As for a general shaped domain, we just need to immerse it into a rectangle such that the Cartesian grid can be generated on the whole rectangular fictitious domain. However, if the data functions $\varphi^{\omega}$ are computed on a general triangular mesh, they need to reevaluated at the generated Cartesian grid points since the values at the triangle mesh points are not able to directly support the CNN computation. To alleviate the computational burden, we can prepare these data functions before training rather than solve them during training. It should be noted that the DDSM indeed requires more memories compared with other deep learning approaches that directly input the original boundary point data, since the newly generated data functions are at least 2D.

\section{Numerical Experiments}
\label{sec:num}

%\\
In this section, we present numerical experiments to demonstrate that our newly
developed DDSMs are effective and robust for the reconstruction of inhomogeneous inclusions in the DOT. For such medical imaging problems, in general only limited data can be obtained from real clinical environments, but instead, simulation or the so-called synthetic data are cheap, easily accessible and are not subject to the objective factors. As suggested by many works in the literature \cite{2018YedderBenTaiebShokoufi,2019FanYing,2021GuoJiang,2020YooSabirHeo,2019ZhangZhang}, these features make synthetic data suitable for training DNN, and the resulting network can be further enhanced by realistic data from clinics. So we shall sample the inclusion distribution by randomly creating some basic geometric objects in $\Omega$ which are then used to generate synthetic data set for training and testing. 

\subsection{2D Problem Setting and Data Generation} 
%Let us first describe the set-up of our numerical experiments in the 2D case.
 Let the boundary (interface) of the $i$-th ($i=1,2,...,M$) basic geometric object be represented by a level-set function $\Gamma_i(x_1,x_2)$, then define the boundary of the inclusion as the zero level set of
\begin{equation}
\label{gamma_distrib}
\Gamma(x_1,x_2) = \min_{i=1,2,...,M} \{ \Gamma_i(x_1,x_2) \}.
\end{equation}
%By this formulation, these basic geometric bodies are allowed to touch and overlap with each other. 
Then the support of the inclusions is the subset $\{(x_1,x_2) ~:~ \Gamma(x_1,x_2)<0\}$. More specifically, we consider the follow two scenarios with different basic geometric objects for training data generation:

\textbf{Scenario 1}: $\Gamma_i$ are 5 random circles with the radius sampled from $\mathcal{U}(0.2,0.4)$ and the center sampled from $\mathcal{U}(-0.7,0.7)$.

\textbf{Scenario 2}: $\Gamma_i$ are 4 random ellipses with the longer axis, the eccentricity and the center points  sampled from $\mathcal{U}(0.2,0.6)$, $\mathcal{U}(0,0.9)$ and $\mathcal{U}(-0.7,0.7)$, respectively. \\
Here $\mathcal{U}(a,b)$ denotes the uniform distribution over $[a,b]$. Sampling circles or ellipses are widely used in many deep learning methods for generating synthetic data in DOT \cite{2019FanYing,2020YooSabirHeo}. However, a major difference is that the format \eqref{gamma_distrib} allows those basic inclusions to touch and overlap with each other such that the overall inclusion distribution could be much more geometrically complicated,  which makes the reconstruction more challenging. Besides, we set $\mu_0=0$ and $\mu_1=50$ for our experiments. 

For the boundary conditions, following our previous work \cite{2021GuoJiang}, we still use Fourier functions as the applied surface flux boundary data $g_{\omega}$ since they naturally form the orthogonal bases on the boundary. In particular, we pick the first $N$ modes:
\begin{equation}
\label{bc_2d}
g_{\omega}(\theta) = \cos(\omega \theta ) ~~ \omega= 1,2,...,N/2  ~~ \text{and}  ~~ g_{\omega}(\theta)=\sin( (\omega-N/2) \theta), ~~ \omega= N/2+1 ,...,N,
\end{equation}
where $\theta\in [0,2\pi)$ is the angle of $(x_1,x_2)\in\partial\Omega$, and we choose $g_{1}(\theta) = \cos( \theta )$ for the case of a single measurement, i.e., $N=1$, and $N=10,20$ for the case of multiple measurements. For each inclusion sample with every boundary condition in \eqref{bc_2d}, we apply finite difference methods on a $100\times 100$ mesh to solve the forward equation \eqref{dot_model} to obtain $u_{\omega}$ and  $f_{\omega}=u_{\omega}|_{\partial \Omega}$. Then, the data pairs $(f_{\omega},g_{\omega})_{\omega=1}^N$ are used in \eqref{eq_multiphi} to generate data functions as the input of the proposed DNN.

On the one hand, it is generally difficult to obtain the accurate knowledge of the physiological noise presented in human data. On the other hand, as discussed in Section \ref{sec:DDSM} and shown in the left plot of Figure \ref{fig: noise}, using boundary data to generate the data functions $\varphi^{\omega}$ can smooth out the noise in the sense that the interior of the data functions can be highly smooth. Both the original DSM \cite{chow2015direct,chow2014direct} and our recent work using DDSM for solving EIT \cite{2021GuoJiang} indicate that this mechanism can significantly enhance the robustness with respect to the noise. Therefore, instead of adding noise in the training set, we will add very large noise in the test data. This is intentionally to test the robustness of the proposed DNN with respect to noise that is not contained in the training set. In particular, we apply the following point-wise noise on the synthetic measured data
\begin{equation}
\label{noise_eq}
f^{\delta}_{\omega}(x) = (1 + \delta G(x)) f_{\omega},
\end{equation}
where $\delta$ is the signal-to-noise ratio chosen as $0$, $5\%$ and $10\%$, $G(x)$ are Gaussian random variables with standard norm distribution which are assume to be independently identical with respect to the points $x$. Different from the set-up in \cite{chow2015direct} that uses a spatially-invariant noise, such noise will cause very rough data on the boundary as shown in Figure \ref{fig: noise}, challenging the robustness of the reconstruction algorithms more seriously.

%which is similar to the strategy in \cite{2019LiZhouWangWangLu}

\subsection{Numerical Results for 2D}

Now we present reconstruction results for each scenario in 2D, and explore the performance of the proposed algorithm for some more challenging situtions. 

\vspace{0.1in}

\textbf{Some basic results:}
The reconstruction results for three basic cases are provided in Figures \ref{cir-4-comp} and \ref{ellipse-4-comp} for each scenario. The figures clearly show the accurate reconstruction of the proposed algorithm. In particular, for the third case of Figure \ref{ellipse-4-comp}, the true inclusion has a concave portion near the domain center away from the boundary data which is generally difficult to be captured. But the proposed algorithm can recover it quite satisfactorily.

\begin{figure}[htbp]
\begin{tabular}{ >{\centering\arraybackslash}m{0.8in}>{\centering\arraybackslash}m{0.8in} >{\centering\arraybackslash}m{0.8in}  >{\centering\arraybackslash}m{0.8in} >{\centering\arraybackslash}m{0.8in}  >{\centering\arraybackslash}m{0.8in} }
	\centering
	True coefficients &
	Single pair, $\delta=0$ &
	N = 10, $\delta=0$ &
	N = 20, $\delta=0$ &
	N = 20, $\delta=5\%$ &
	N = 20, $\delta=10\%$ \\
        \includegraphics[width=0.8in]{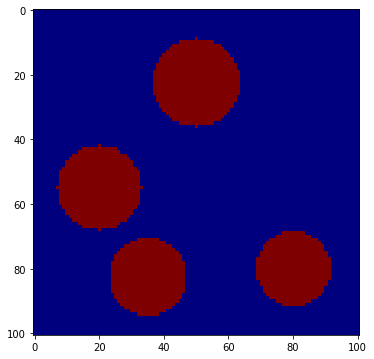}&
	\includegraphics[width=0.8in]{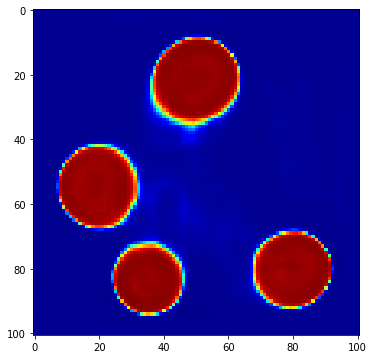}&
	\includegraphics[width=0.8in]{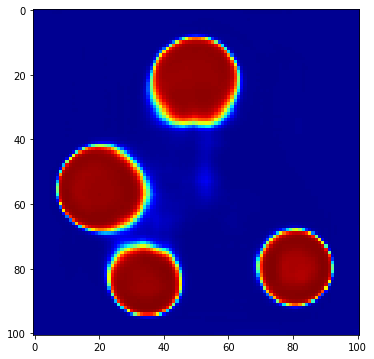}&
	\includegraphics[width=0.8in]{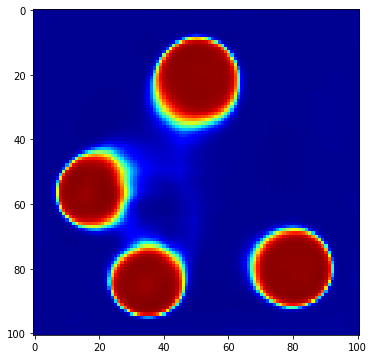}&
	\includegraphics[width=0.8in]{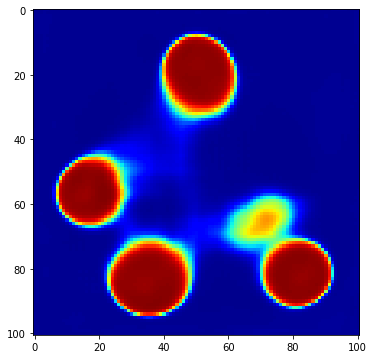}&
	\includegraphics[width=0.8in]{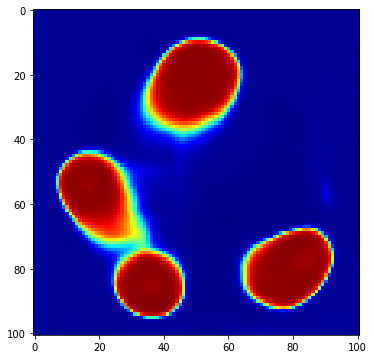}\\
	\includegraphics[width=0.8in]{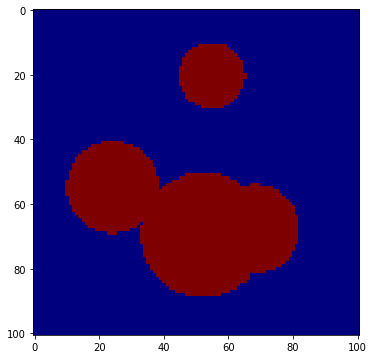}&
	\includegraphics[width=0.8in]{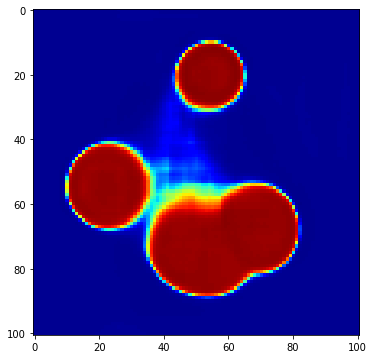}&
	\includegraphics[width=0.8in]{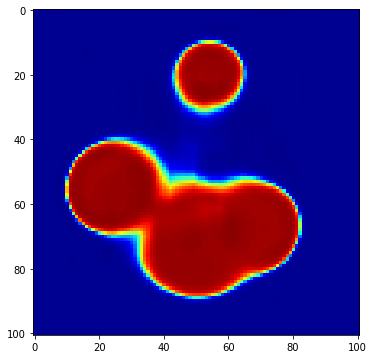}&
	\includegraphics[width=0.8in]{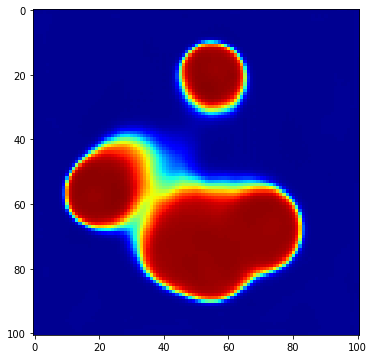}&
	\includegraphics[width=0.8in]{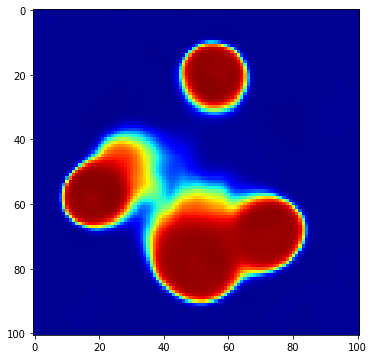}&
	\includegraphics[width=0.8in]{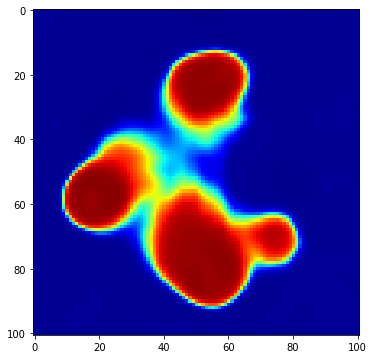}\\
	\includegraphics[width=0.8in]{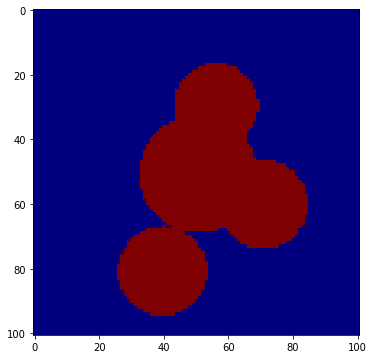}&
	\includegraphics[width=0.8in]{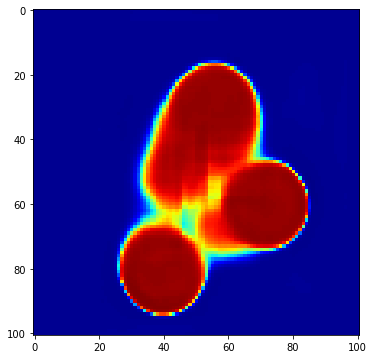}&
	\includegraphics[width=0.8in]{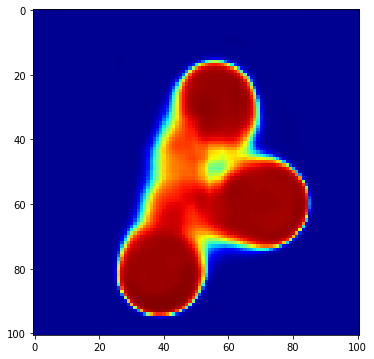}&
	\includegraphics[width=0.8in]{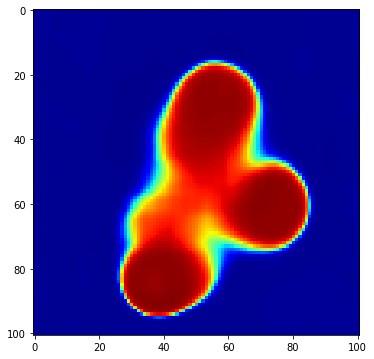}&
	\includegraphics[width=0.8in]{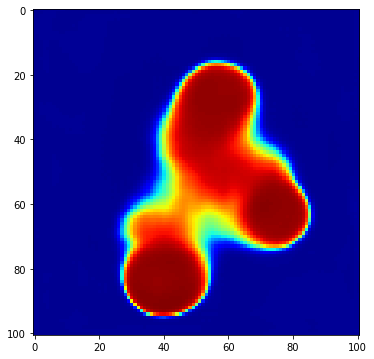}&
	\includegraphics[width=0.8in]{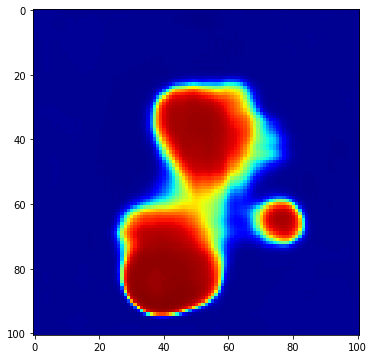}\\
\end{tabular}
\caption{Reconstruction for 3 cases in \textbf{Scenario 1} (4 circles) with different Cauchy data number and noise level: Case 1(top), Case 2(middle) and Case 3(bottom) } 
\label{cir-4-comp}
\end{figure}

\begin{figure}[ht]
\begin{tabular}{ >{\centering\arraybackslash}m{0.8in}>{\centering\arraybackslash}m{0.8in} >{\centering\arraybackslash}m{0.8in}  >{\centering\arraybackslash}m{0.8in} >{\centering\arraybackslash}m{0.8in}  >{\centering\arraybackslash}m{0.8in} }
	\centering
	True coefficients &
	Single pair, $\delta=0$ &
	N = 10, $\delta=0$ &
	N = 20, $\delta=0$ &
	N = 20, $\delta=5\%$ &
	N = 20, $\delta=10\%$ \\
	\includegraphics[width=0.8in]{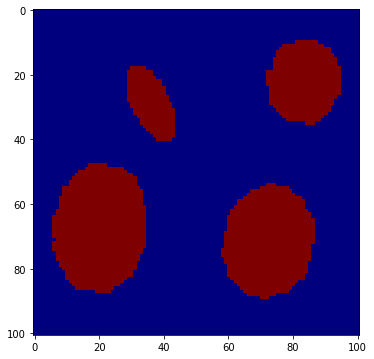}&
	\includegraphics[width=0.8in]{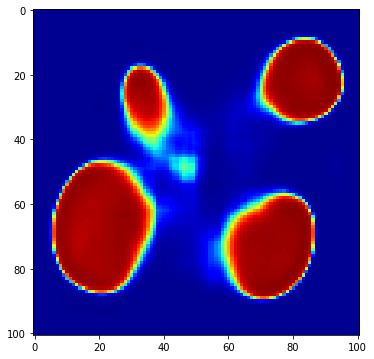}&
	\includegraphics[width=0.8in]{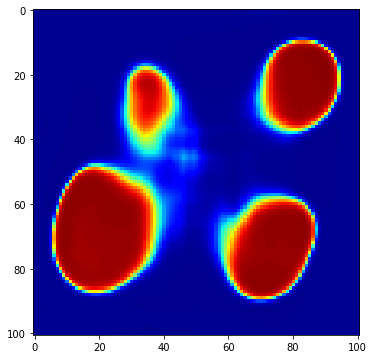}&
	\includegraphics[width=0.8in]{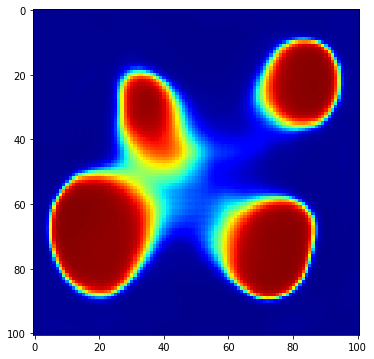}&
	\includegraphics[width=0.8in]{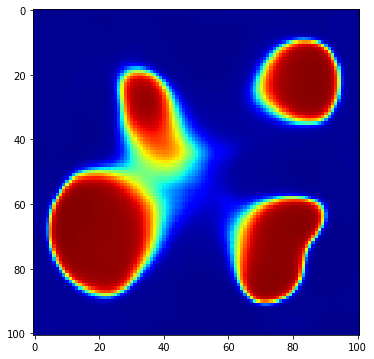}&
	\includegraphics[width=0.8in]{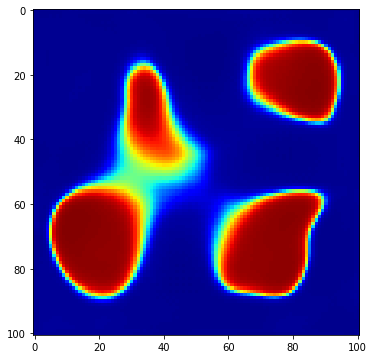}\\
	\includegraphics[width=0.8in]{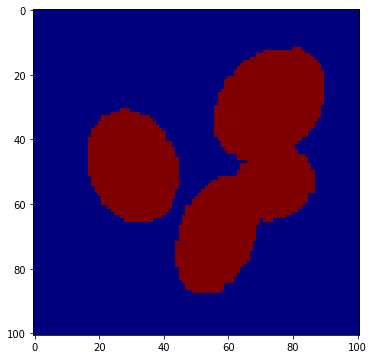}&
	\includegraphics[width=0.8in]{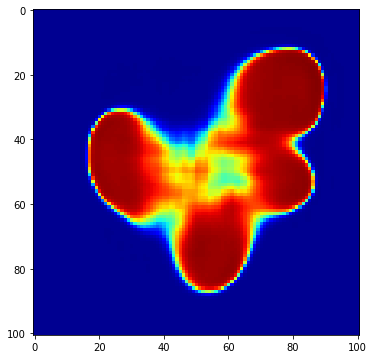}&
	\includegraphics[width=0.8in]{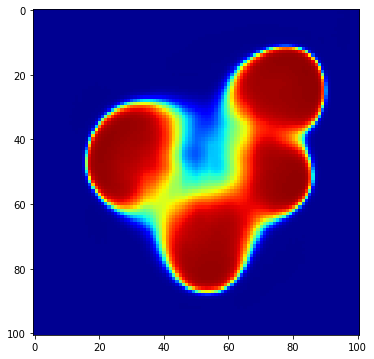}&
	\includegraphics[width=0.8in]{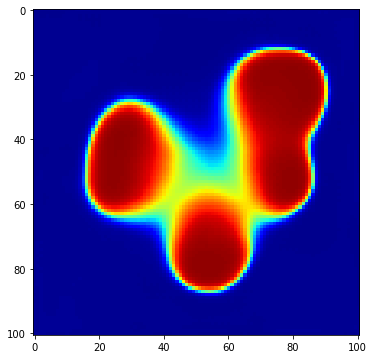}&
	\includegraphics[width=0.8in]{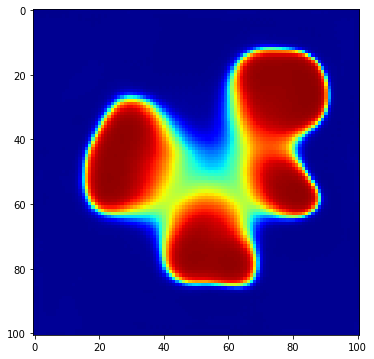}&
	\includegraphics[width=0.8in]{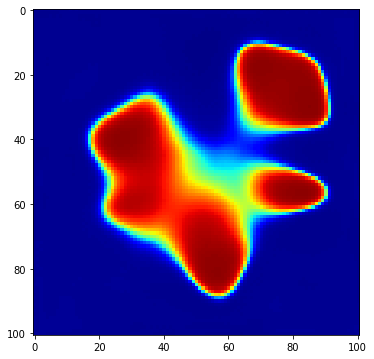}\\
	\includegraphics[width=0.8in]{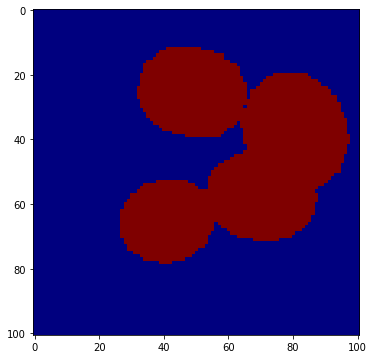}&
	\includegraphics[width=0.8in]{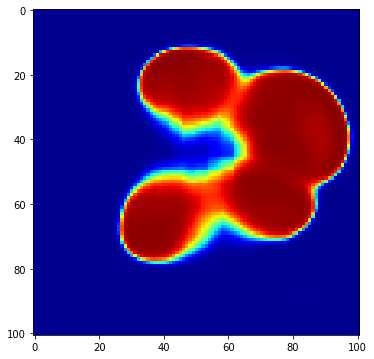}&
	\includegraphics[width=0.8in]{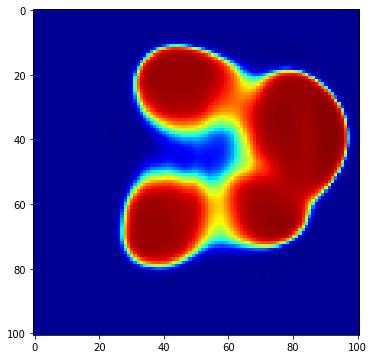}&
	\includegraphics[width=0.8in]{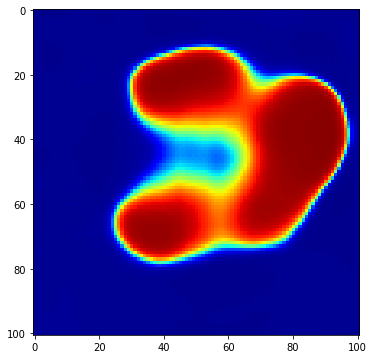}&
	\includegraphics[width=0.8in]{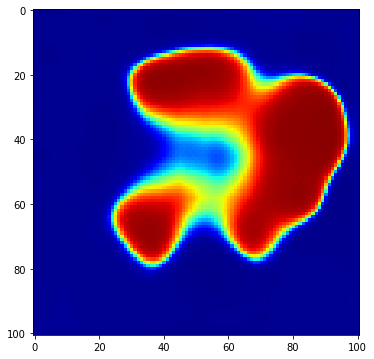}&
	\includegraphics[width=0.8in]{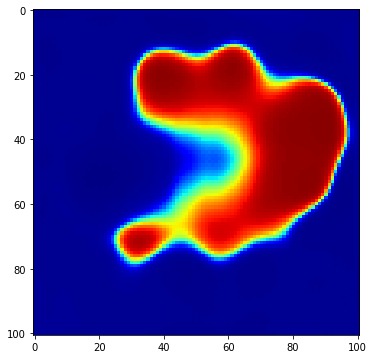}\\
\end{tabular}
\caption{Reconstruction for 3 cases in \textbf{Scenario 2} (4 ellipses) with different Cauchy data number and noise level: Case 1(top), Case 2(middle) and Case 3(bottom) } 
\label{ellipse-4-comp}
\end{figure}

We observe that the proposed algorithm using single or multiple measurements almost gives comparably good reconstruction. However, our further numerical results suggest that using multiple measurements can significantly enhance the robustness of the reconstruction with respect to noise. To demonstrate this, for the first case in Scenario 1 (the top one in Figure \ref{cir-4-comp}), we present the reconstruction of the single pair and 10 pairs with noise in Figure \ref{cir-4-comp_noise_N1}. It is observed that even $5\%$ noise can totally destroy the reconstruction for the single measurement case while the performance of 10 pairs of measurements is much better but still worse than 20 pairs. It is also worth mentioning that, even for the single measurement, the reconstruction is still highly robust with respect to the spatially-invariant noise used in \cite{2015ChowItoLiuZou}, i.e., $G$ is independent of $x$ in \eqref{noise_eq}, of which the results are omitted here.

\begin{figure}[htbp]
\centering
\begin{tabular}{>{\centering\arraybackslash}m{0.8in} >{\centering\arraybackslash}m{0.8in}  >{\centering\arraybackslash}m{0.8in} >{\centering\arraybackslash}m{0.8in}  >{\centering\arraybackslash}m{0.8in} }
	\centering
	  &
	 $\delta=0.5\%$ &
	 $\delta=1\%$ &
	 $\delta=2\%$ &
	 $\delta=5\%$ \\
	Single pair&
	 \includegraphics[width=0.8in]{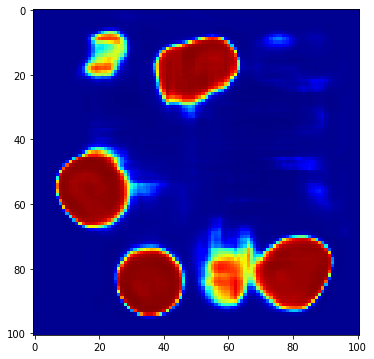}&
	\includegraphics[width=0.8in]{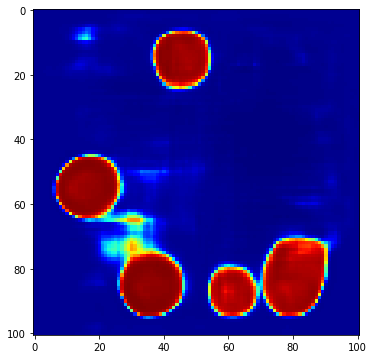}&
	\includegraphics[width=0.8in]{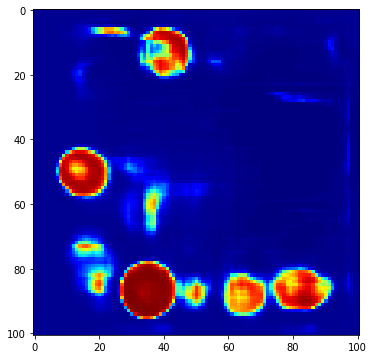}&
        \includegraphics[width=0.8in]{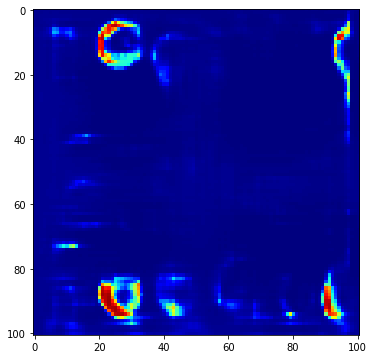}\\
	$N=10$ &
	\includegraphics[width=0.8in]{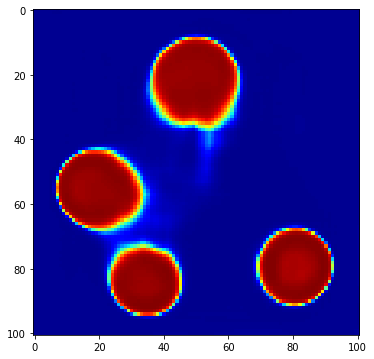}&
	\includegraphics[width=0.8in]{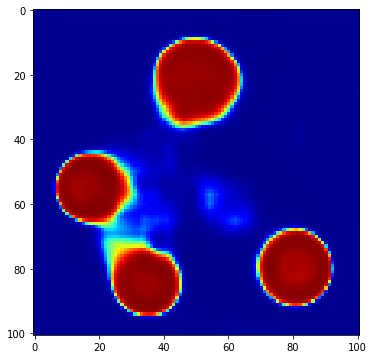}&
	\includegraphics[width=0.8in]{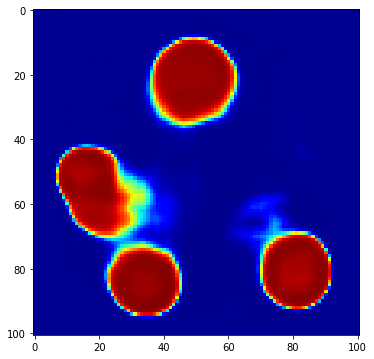}&
	\includegraphics[width=0.8in]{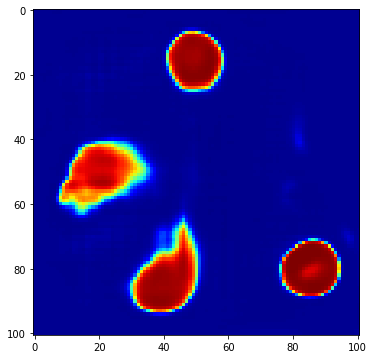}\\
\end{tabular}
\caption{Comparison of reconstruction with respect to noise for the single measurement and 10 pairs of measurements.} 
\label{cir-4-comp_noise_N1}
\end{figure}

%\begin{figure}[ht]
%	\centering
%	\includegraphics{true-cir-04.png}
%	\caption{True circle}
%	\label{fig:circle-04}
%\end{figure}
%
%\begin{figure}[htbp]
%\centering
%\subfigure[pic1.]{
%\includegraphics[width=5.5cm]{./true-cir-04.png}
%%\caption{fig1}
%}
%\quad
%\subfigure[pic2.]{
%\includegraphics[width=5.5cm]{./test-cir-1pair-noise-0perc-relu-mse-sgd-04.png}
%}
%\quad
%\subfigure[pic3.]{
%\includegraphics[width=5.5cm]{./test-cir-10pair-noise-0perc-relu-mse-sgd-04.png}
%}
%\quad
%\subfigure[pic4.]{
%\includegraphics[width=5.5cm]{./test-cir-20pair-noise-0perc-relu-mse-sgd-04.png}
%}
%\caption{ pics}
%\end{figure}

\vspace{0.1in}

\textbf{Sensitivity to Data:}
The DOT is well-known to be highly ill-posed which means the inclusions are very insensitive to the boundary data. To examine this phenomenon and test the sensitivity of the algorithm with respect to the data, we consider two different inclusion distributions in Figure \ref{fig:cir-5-comp4-4},  where the only difference is the center inclusion. We plot the corresponding $f_{\omega}$ by fixing the applied flux $g_{\omega}$, $\omega=1,2,...,10$, which are indeed quite close to each other. Namely, the center inclusion is extremely insensitive to noise, which makes its reconstruction very difficult. The results in Figure \ref{fig:cir-5-comp4-4} show that the center inclusion can still  be clearly captured. We highlight that for 20 pairs of measurements the center inclusion is captured even with $5\%$ noise which is even larger than the relative difference of their boundary data. These results demonstrate that the proposed algorithm can dig out the small difference buried in boundary data for various inclusion distribution but still keep the robustness with respect to the noise.

 %the sensitivity of the proposed algorithm with respect to data and but still  robustness to noise.

\begin{figure}[htb]
	\centering
	\includegraphics[width=.32\linewidth]{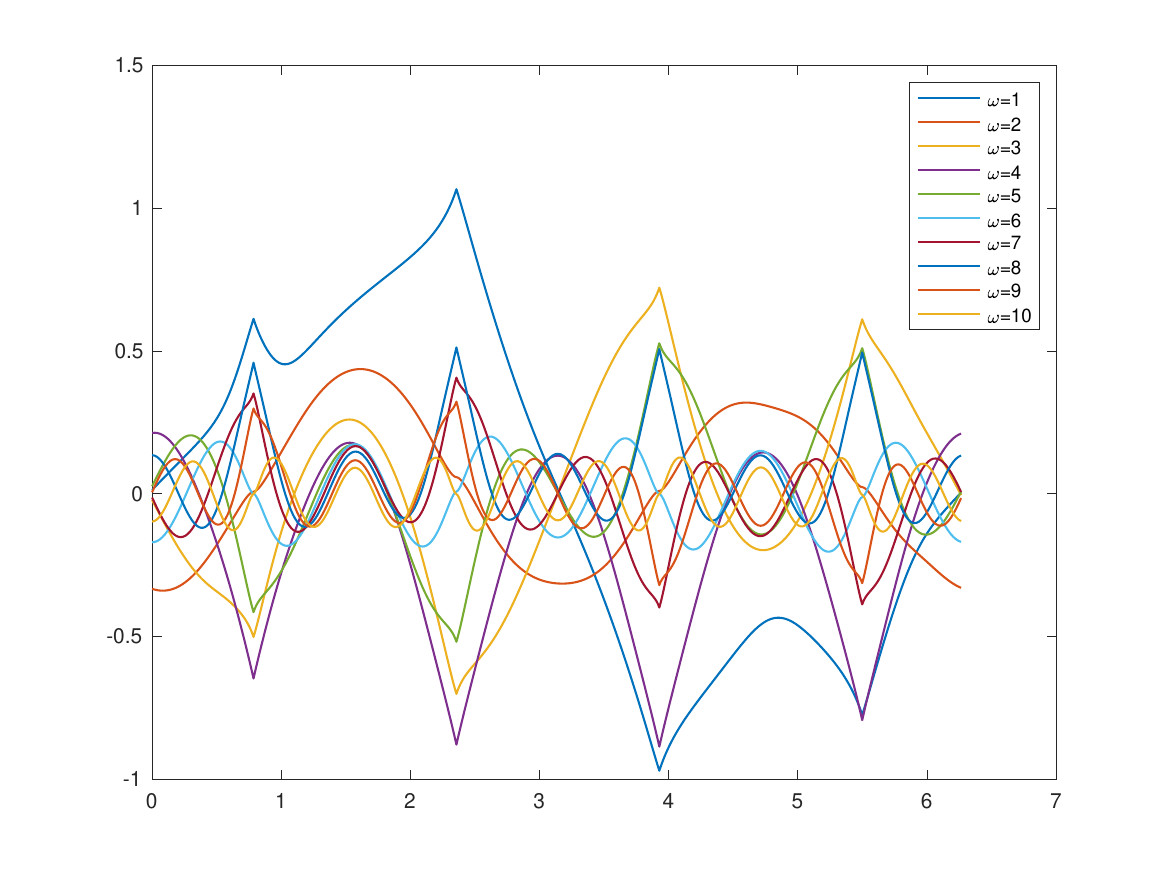}
	\includegraphics[width=.32\linewidth]{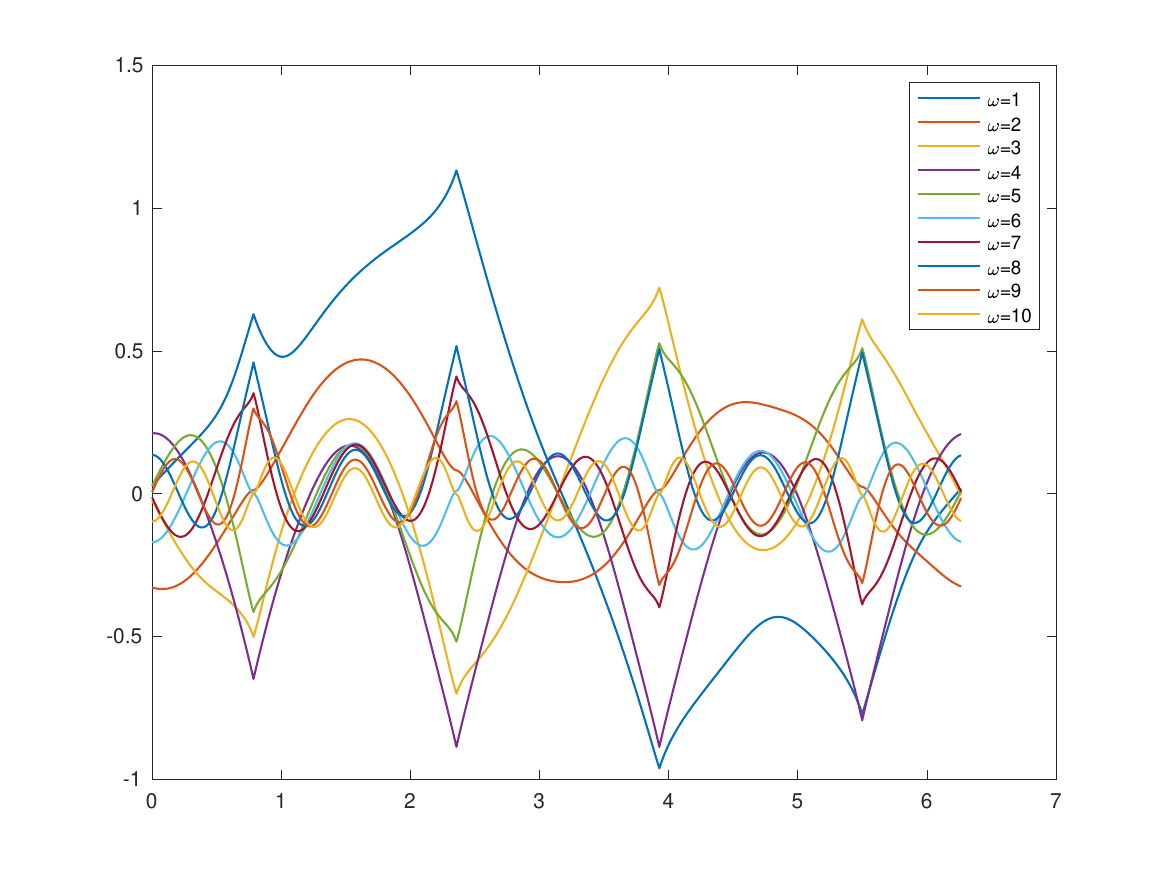}
	\includegraphics[width=.32\linewidth]{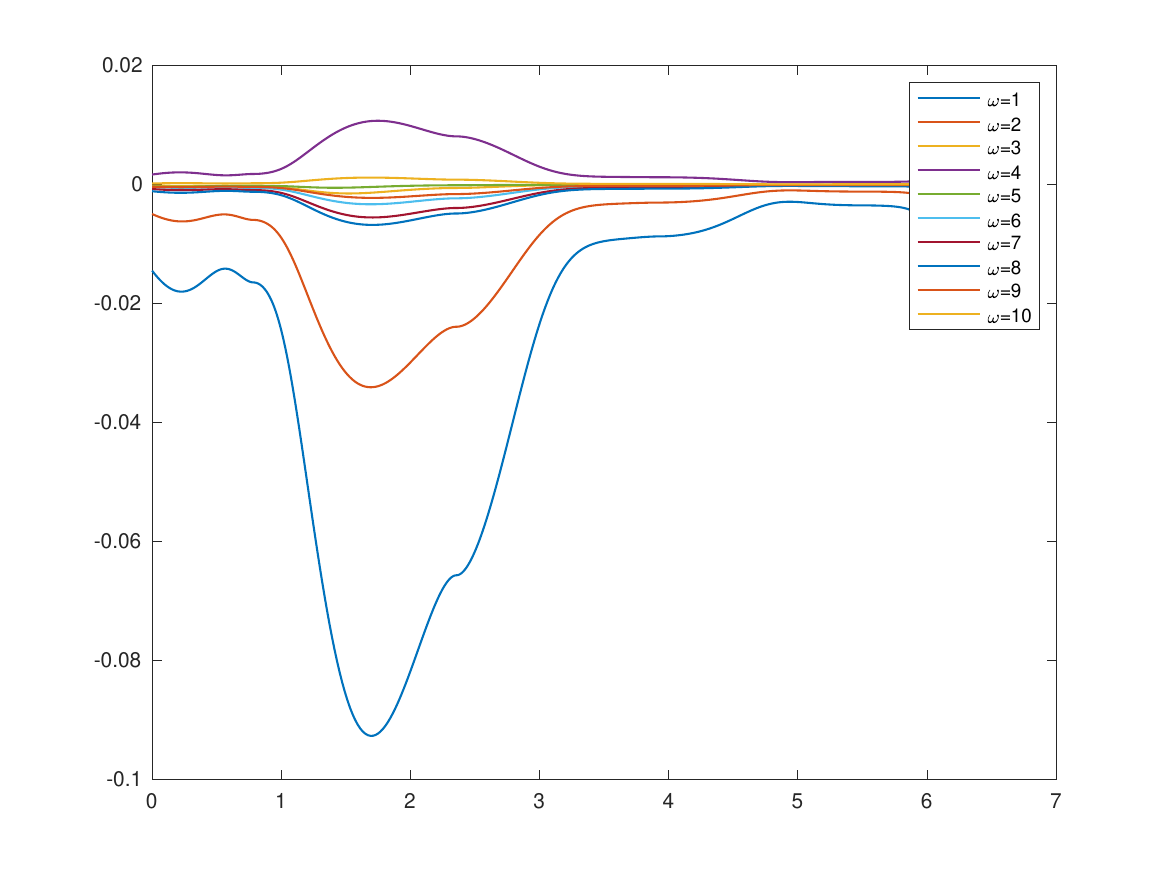}
	\caption{Plots of $g_{1,\omega}|_{\partial\Omega}, g_{2,\omega}|_{\partial\Omega}$ and $g_{1,\omega}|_{\partial\Omega} - g_{2,\omega}|_{\partial\Omega}$ versus the polar angle $\theta$ (of points on the boundary) where $g_{1,\omega}$ and $g_{2,\omega}$ correspond to inclusion distribution with and without the center inclusion respectively.}
	\label{fig:cir-5-diff-4}
\end{figure}

%\begin{figure}[htb]
%	\centering
%	\includegraphics[width=.32\linewidth]{f_mat_diff_4ellip_10pair_ex}
%	\includegraphics[width=.32\linewidth]{f_mat_diff_3ellip_10pair_ex}
%	\includegraphics[width=.32\linewidth]{f_mat_diff_4ellip_3ellip_10pair_ex}
%	\caption{Plots of $u^e_{1,\omega}|_{\partial\Omega}, u^e_{2,\omega}|_{\partial\Omega}$ and $u^e_{1,\omega}|_{\partial\Omega} - u^e_{2,\omega}|_{\partial\Omega}$ versus the polar angle $\theta$ (of points on the boundary)}
%	\label{fig:ellip-4-diff-3}
%\end{figure}

\begin{figure}[htbp]
\begin{tabular}{ >{\centering\arraybackslash}m{0.8in}>{\centering\arraybackslash}m{0.8in} >{\centering\arraybackslash}m{0.8in}  >{\centering\arraybackslash}m{0.8in} >{\centering\arraybackslash}m{0.8in}  >{\centering\arraybackslash}m{0.8in} }
	\centering
	True coefficients &
	Single pair, $\delta=0$ &
	N = 10, $\delta=0$ &
	N = 20, $\delta=0$ &
	N = 20, $\delta=5\%$ &
	N = 20, $\delta=10\%$ \\
	\includegraphics[width=0.8in]{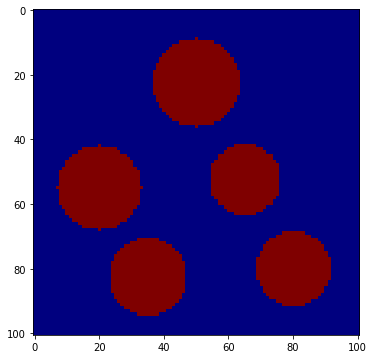}&
	\includegraphics[width=0.8in]{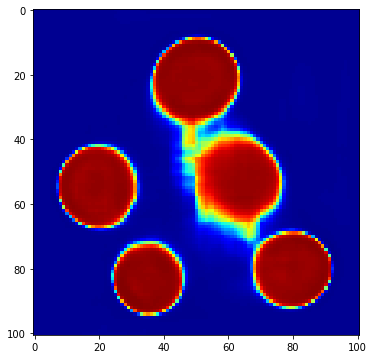}&
	\includegraphics[width=0.8in]{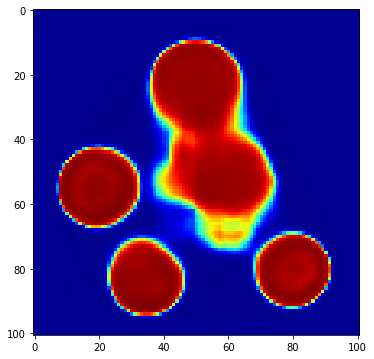}&
	\includegraphics[width=0.8in]{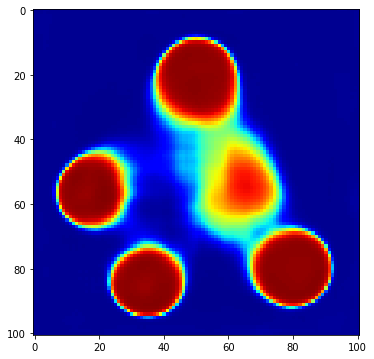}&
	\includegraphics[width=0.8in]{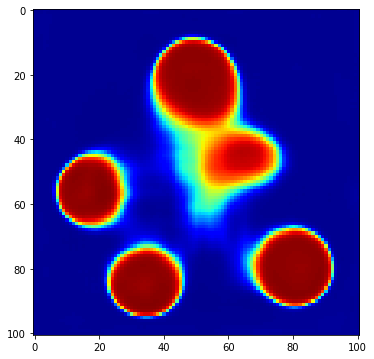}&
	\includegraphics[width=0.8in]{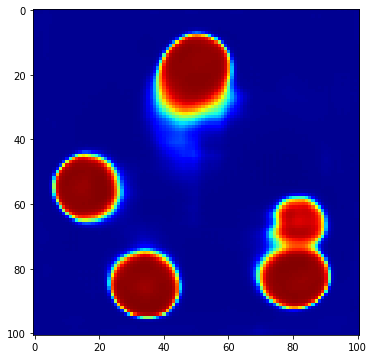}\\
	\includegraphics[width=0.8in]{true-4cir-comp4-01}&
	\includegraphics[width=0.8in]{test-4cir-1pair-noise-0perc-relu-mse-sgd-comp4-01}&
	\includegraphics[width=0.8in]{test-4cir-10pair-noise-0perc-relu-mse-sgd-comp4-01}&
	\includegraphics[width=0.8in]{test-4cir-20pair-noise-0perc-relu-mse-sgd-comp4-01}&
	\includegraphics[width=0.8in]{test-4cir-20pair-noise-5perc-relu-mse-sgd-comp4-01-real}&
	\includegraphics[width=0.8in]{test-4cir-20pair-noise-10perc-relu-mse-sgd-comp4-01-real}\\
\end{tabular}
\caption{Comparison of reconstruction for two close inclusion distribution.} 
\label{fig:cir-5-comp4-4}
\end{figure}

\vspace{0.1in}

\textbf{Out-of-scope prediction:}
We note that the basic geometric shape such as circles or ellipses may be more appropriate to be chosen to imitate the tumors for the clinical application of the DDSM. 
%may be chosen as those more close to tumor shapes to apply the DDSM for clinical environment such that better reconstruction may be obtained. 
But even so, it may not be expected that the shape to be reconstructed is always covered by or close to the training data set, which triggers us to further investigate the capability of the DDSM for the inclusion distribution that has the geometry out of the scope of the training data, that is, they can not be generated by those basic geometry objects. Here we present the reconstruction for three different shapes: a triangle, two bars and an annulus in Figure \ref{others-comp}. In particular, the annulus has a hollow center which is difficult to be captured by the lights sourced at the boundary. But we can still observe the satisfactory reconstruction of their basic geometric properties for all these inclusions. Similar to the previous results, the reconstruction by 20 parts is still robust with respect to the large noise. Note that one can add these shapes to the training data set if some more accurate reconstruction is needed.

\begin{figure}[htbp]
\begin{tabular}{ >{\centering\arraybackslash}m{0.8in}>{\centering\arraybackslash}m{0.8in} >{\centering\arraybackslash}m{0.8in}  >{\centering\arraybackslash}m{0.8in} >{\centering\arraybackslash}m{0.8in}  >{\centering\arraybackslash}m{0.8in} }
	\centering
	True coefficients &
	N = 1, $\delta=0$ &
	N = 10, $\delta=0$ &
	N = 20, $\delta=0$ &
	N = 20, $\delta=5\%$ &
	N = 20, $\delta=10\%$ \\
	\includegraphics[width=0.8in]{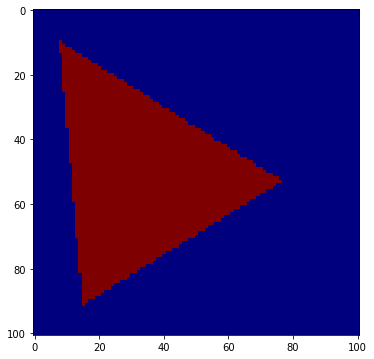}&
	\includegraphics[width=0.8in]{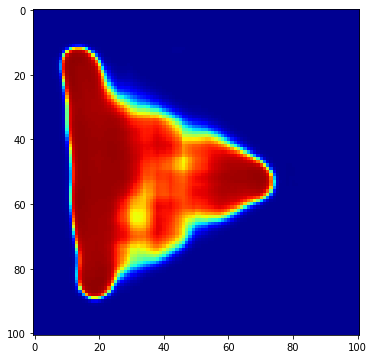}&
	\includegraphics[width=0.8in]{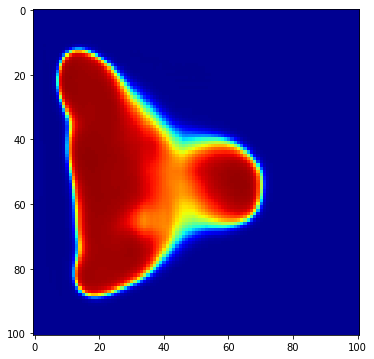}&
	\includegraphics[width=0.8in]{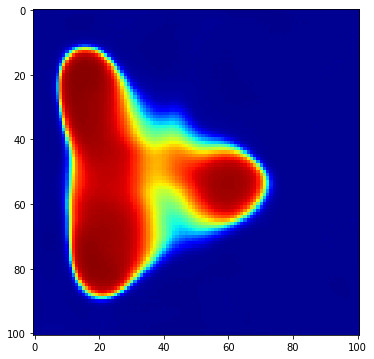}&
	\includegraphics[width=0.8in]{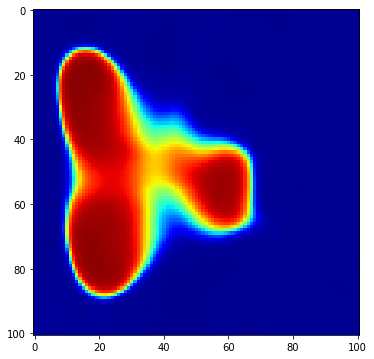}&
	\includegraphics[width=0.8in]{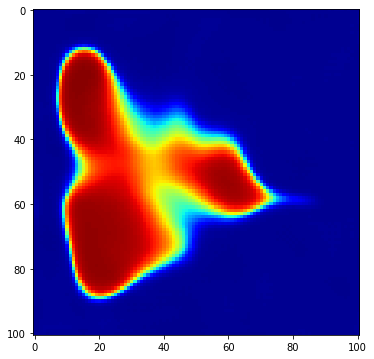}\\
	\includegraphics[width=0.8in]{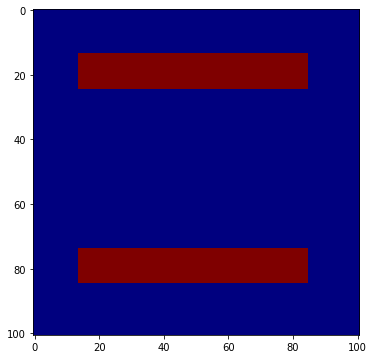}&
	\includegraphics[width=0.8in]{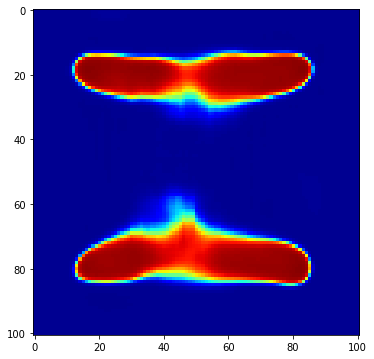}&
	\includegraphics[width=0.8in]{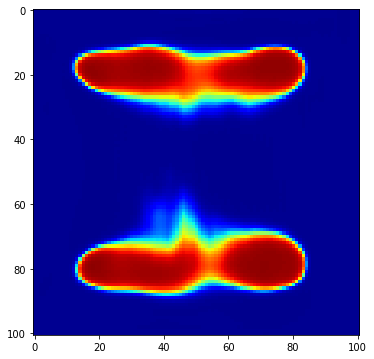}&
	\includegraphics[width=0.8in]{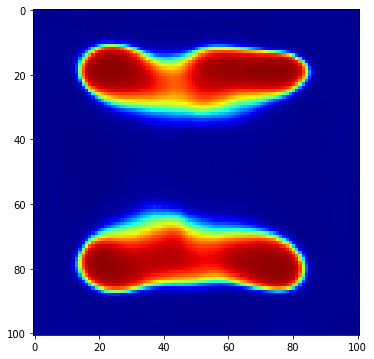}&
	\includegraphics[width=0.8in]{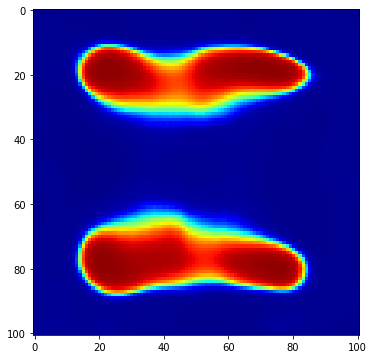}&
	\includegraphics[width=0.8in]{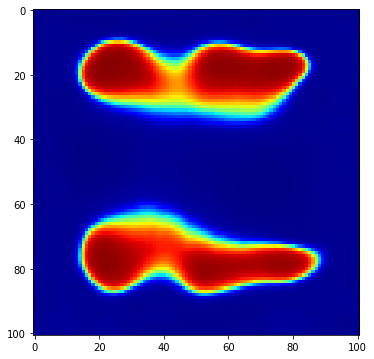}\\
	\includegraphics[width=0.8in]{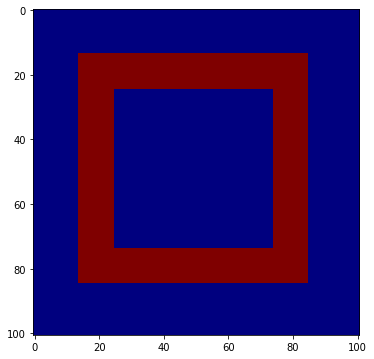}&
	\includegraphics[width=0.8in]{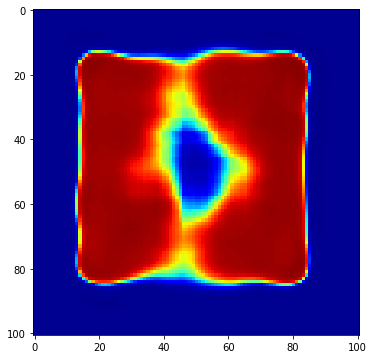}&
	\includegraphics[width=0.8in]{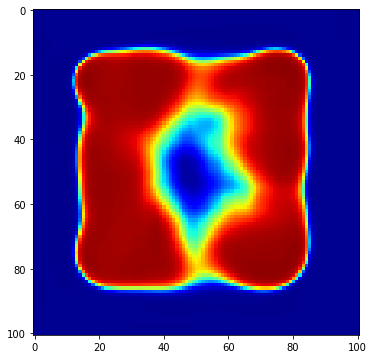}&
	\includegraphics[width=0.8in]{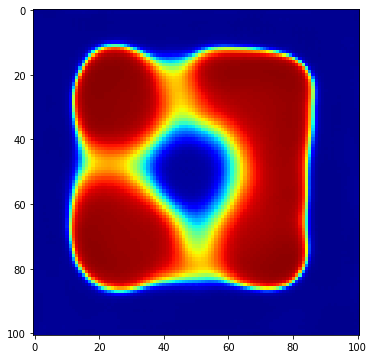}&
	\includegraphics[width=0.8in]{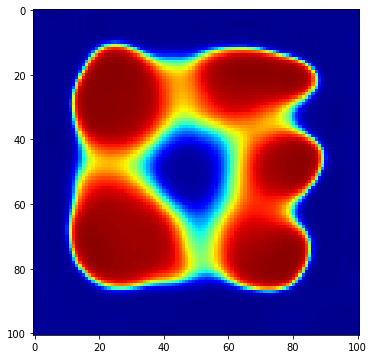}&
	\includegraphics[width=0.8in]{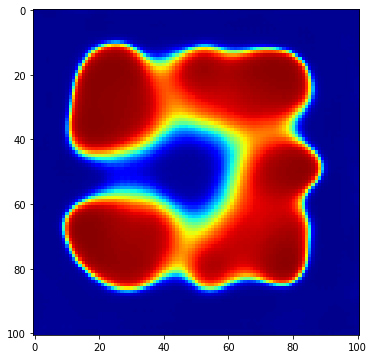}\\
\end{tabular}
\caption{CNN-DDSM reconstruction for 3 special inclusion shapes: one triangle (top), two long rectangular bars (middle) and a rectangular ring (bottom) } 
\label{others-comp}
\end{figure}

\vspace{0.1in}

\textbf{Limited data points:}
Note that all the previous results are generated by assuming that the data are available at every boundary grid point for using finite difference methods to generate data functions $\varphi^{\omega}$, which may not be practical. Indeed, according to DOT experiments, see \cite{2020RenJiangCostanzoKalyanov} for example, even if a camera may receive the light at every point, the light sources can be placed on only a few points on the boundary. So we shall explore this issue by limiting the data points where the Neumann data $g_{\omega}$ are available. In particular, we consider the case that there are $4L$, $L=8,16$, data points on the boundary with $L$ points equally distributed on each side of the domain. The data $g_{\omega}$ are assumed to be obtained at these points which are then linearly interpolated to generate functions on the boundary. These functions are then used to generate the corresponding Dirichlet data functions $f_{\omega}$ available at all the boundary mesh points. Then the interpolated Neumann data and the simulated Dirichlet date are used to generate data functions $\varphi^{\omega}$ as the input to the same DNN used above, and the reconstruction results are presented in Figure \ref{fig:4ellip-1-limited-1} for the first case in Figure \ref{ellipse-4-comp}. As we can see, the three relatively larger ellipses are reconstructed quite satisfactory even with $10\%$ noise. However, the result for the small ellipse is not as good as the others. We guess it may be due to its small size that receives too little light for passing its geometric information to the boundary. Despite its inaccurate shape, the algorithm still tells us that an inclusion exists around the upper-left corner.

\begin{figure}[htb]
\centering
%\begin{tabular}{ >{\centering\arraybackslash}m{0.2in} >{\centering\arraybackslash}m{0.8in} >{\centering\arraybackslash}m{0.8in}  >{\centering\arraybackslash}m{0.8in}  >{\centering\arraybackslash}m{0.8in}  >{\centering\arraybackslash}m{0.8in} }
%	\centering
%	& True coefficients & & \\
%	& \includegraphics[width=0.8in]{true-ellipse-60} & N = 10 & \\
%	& $ 36\;\;pints $ & $40\;\;points$ & $44\;\;points$ \\
%	\rotatebox{90}{$\delta=0\%$}&
%	\includegraphics[width=0.8in]{test-4ellip-10pair-noise-0perc-relu-mse-sgd-comp4-01-limited36}&
%	\includegraphics[width=0.8in]{test-4ellip-10pair-noise-0perc-relu-mse-sgd-comp4-01-limited40}&
%	\includegraphics[width=0.8in]{test-4ellip-10pair-noise-0perc-relu-mse-sgd-comp4-01-limited44}\\
%\end{tabular}
%%%%%%%%%%%%%
\begin{tabular}{ >{\centering\arraybackslash}m{0.8in}   >{\centering\arraybackslash}m{0.2in} >{\centering\arraybackslash}m{0.8in} >{\centering\arraybackslash}m{0.8in}  >{\centering\arraybackslash}m{0.8in}  >{\centering\arraybackslash}m{0.8in}  >{\centering\arraybackslash}m{0.8in} }
	\centering
	 True coefficients & & $ \delta=0\% $ & $\delta=5\%$ & $\delta=10\%$ \\
	 \includegraphics[width=0.8in]{true-ellipse-60} &
	\rotatebox{90}{$32\;\;pints$}&
	\includegraphics[width=0.8in]{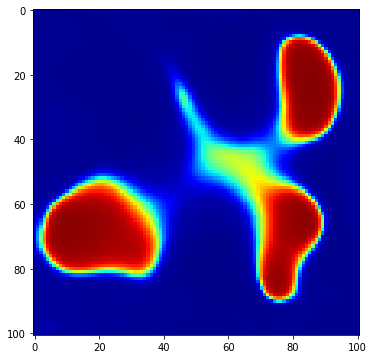}&
	\includegraphics[width=0.8in]{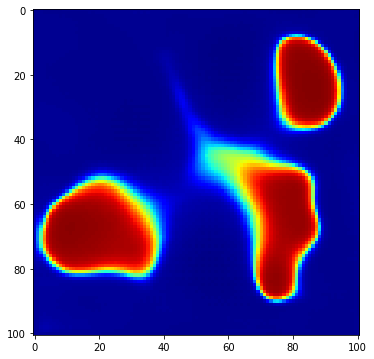}&
	\includegraphics[width=0.8in]{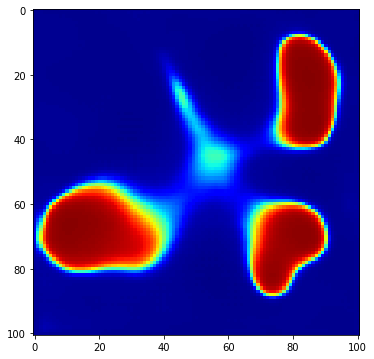}\\
	&
	\rotatebox{90}{$48\;\;pints$}&
	\includegraphics[width=0.8in]{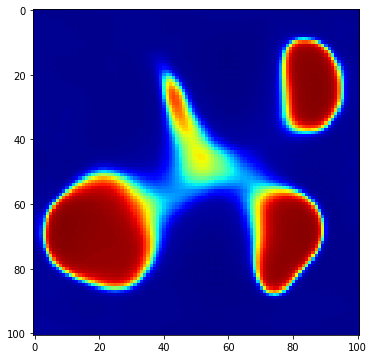}&
	\includegraphics[width=0.8in]{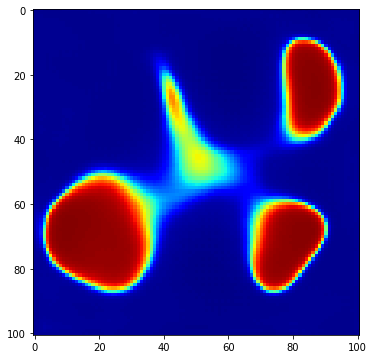}&
	\includegraphics[width=0.8in]{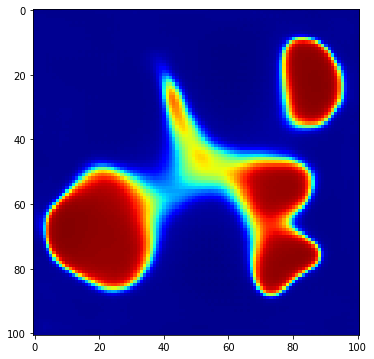}\\
%	&
%	\rotatebox{90}{$\delta=10\%$}&
%	\includegraphics[width=0.8in]{test-4ellip-20pair-noise-10perc-relu-mse-sgd-comp4-01-limited20}&
%	\includegraphics[width=0.8in]{test-4ellip-20pair-noise-10perc-relu-mse-sgd-comp4-01-limited32}&
%	\includegraphics[width=0.8in]{test-4ellip-20pair-noise-10perc-relu-mse-sgd-comp4-01-limited40}\\
	%%
\end{tabular}
\caption{DDSM reconstruction for random ellipses with limited data points and $N=20$.}
\label{fig:4ellip-1-limited-1}
\end{figure}

%CNN-DDSM (the bottom three rows) reconstruction for random ellipses: one ellipse is located closed to the center of domain and blocked from the boundary by other 3 ellipses (left) and this center ellipse is removed (right).

\vspace{0.1in}

\textbf{Reconstruction for different $\mu$:}
Now we demonstrate that the proposed algorithm can be used to obtain reasonable reconstruction even though the true absorption coefficient values are much different from those used to generate training data. Here we also use the first case in Figure \ref{ellipse-4-comp} of the scenario 2 as an example. But we generate the boundary data with two different groups of absorption coefficient values $(\mu_0,\mu_1)=(0,200)$ and $(1,100)$ where the background and inclusion coefficient values may both vary. The same DNN is used predict the inclusion geometry of which the results are presented in Figure \ref{fig:4ellip-diff-mu}. Although the reconstruction is not as good as the one shown in Figure \ref{ellipse-4-comp}, we can see that all the four ellipses are clearly captured and the reconstruction is still stable with respect to noise. We highlight that this feature is very important for the practical clinical situations as the material property of patients' tumors may vary and may not be known accurately. In spite of this, the proposed algorithm still has the potential to detect the appearance and shapes of the tumors to certain extend.

\begin{figure}[htb]
\begin{tabular}{ >{\centering\arraybackslash}m{0.1in} >{\centering\arraybackslash}m{0.77in} >{\centering\arraybackslash}m{0.77in}  >{\centering\arraybackslash}m{0.77in}  >{\centering\arraybackslash}m{0.77in}  >{\centering\arraybackslash}m{0.77in} }
	\centering
	& $\delta = 0\%$ & $\delta = 2\%$ & $\delta = 5\%$ \\
	\rotatebox{90}{$N=10$}&
	\includegraphics[width=0.77in]{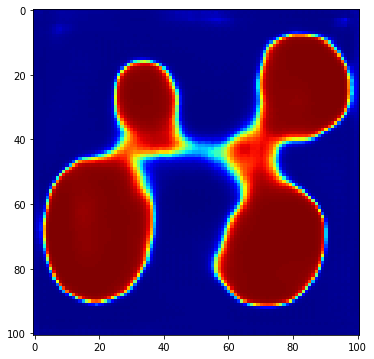}&
	\includegraphics[width=0.77in]{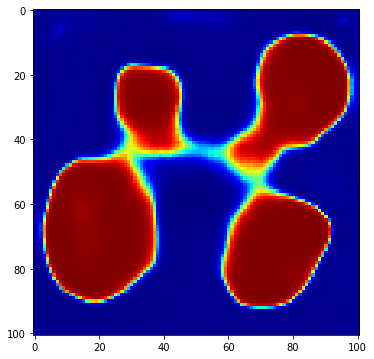}&
	\includegraphics[width=0.77in]{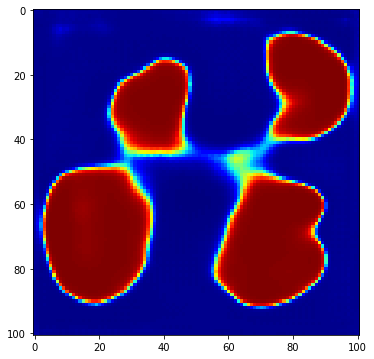}\\
	\rotatebox{90}{$N=20$}&
	\includegraphics[width=0.77in]{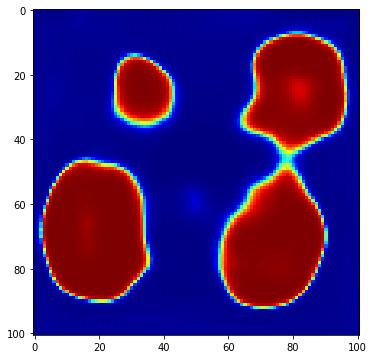}&
%	&
	\includegraphics[width=0.77in]{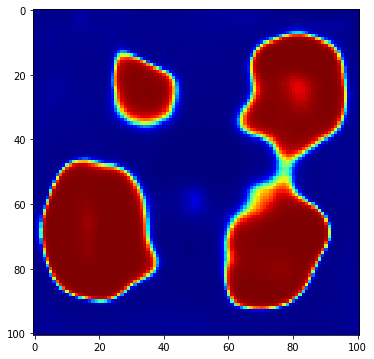}&
	\includegraphics[width=0.77in]{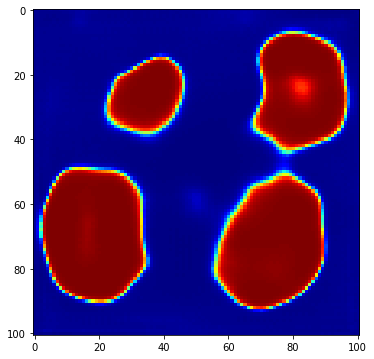}\\
%	& $\delta = 0\%$ & $\delta = 2\%$ & $\delta = 5\%$ \\
%	\rotatebox{90}{$N=10$}&
%	\includegraphics[width=0.8in]{test-4ellip-10pair-noise-0perc-mu-1-100-relu-mse-sgd-01-real-fix}&
%%	&
%	\includegraphics[width=0.8in]{test-4ellip-10pair-noise-2perc-mu-1-100-relu-mse-sgd-01-real-fix}&
%	\includegraphics[width=0.8in]{test-4ellip-10pair-noise-5perc-mu-1-100-relu-mse-sgd-01-real-fix}\\
%	\rotatebox{90}{$N=20$}&
%	\includegraphics[width=0.8in]{test-4ellip-20pair-noise-0perc-mu-1-100-relu-mse-sgd-01-real-fix}&
%%	&
%	\includegraphics[width=0.8in]{test-4ellip-20pair-noise-2perc-mu-1-100-relu-mse-sgd-01-real-fix}&
%%	&
%	\includegraphics[width=0.8in]{test-4ellip-20pair-noise-5perc-mu-1-100-relu-mse-sgd-01-real-fix}\\
\end{tabular}
%%%%%%%%%%%%%
\begin{tabular}{ >{\centering\arraybackslash}m{0.77in} >{\centering\arraybackslash}m{0.77in} >{\centering\arraybackslash}m{0.77in}  >{\centering\arraybackslash}m{0.77in}  >{\centering\arraybackslash}m{0.77in}  >{\centering\arraybackslash}m{0.77in} }
	\centering
	 $\delta = 0\%$ & $\delta = 2\%$ & $\delta = 5\%$ \\
	\includegraphics[width=0.77in]{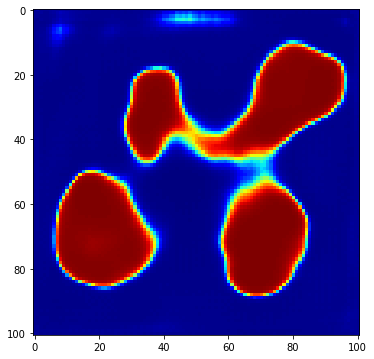}&
%	&
	\includegraphics[width=0.77in]{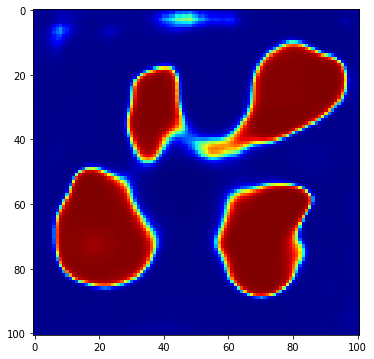}&
	\includegraphics[width=0.77in]{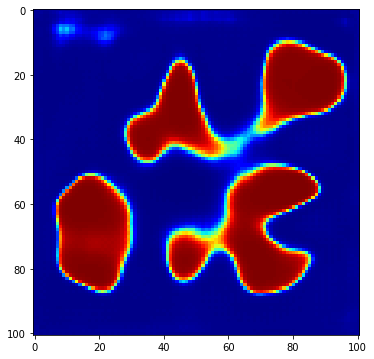}\\
	\includegraphics[width=0.77in]{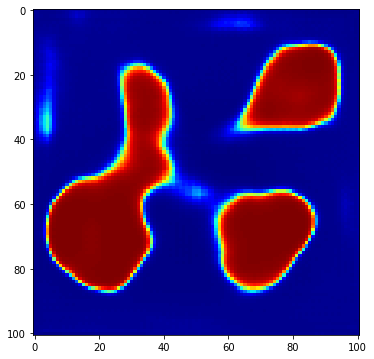}&
%	&
	\includegraphics[width=0.77in]{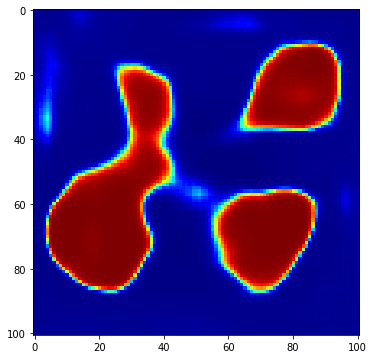}&
%	&
	\includegraphics[width=0.77in]{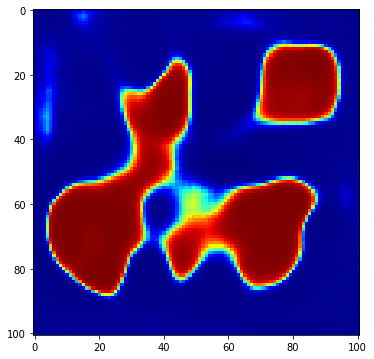}\\
%	$\delta = 0\%$ & $\delta = 2\%$ & $\delta = 5\%$ \\
%	\includegraphics[width=0.8in]{test-4ellip-10pair-noise-0perc-mu-1-100-relu-mse-sgd-03-real-fix}&
%%	&
%	\includegraphics[width=0.8in]{test-4ellip-10pair-noise-2perc-mu-1-100-relu-mse-sgd-03-real-fix}&
%	\includegraphics[width=0.8in]{test-4ellip-10pair-noise-5perc-mu-1-100-relu-mse-sgd-03-real-fix}\\
%	\includegraphics[width=0.8in]{test-4ellip-20pair-noise-0perc-mu-1-100-relu-mse-sgd-03-real-fix}&
%%	&
%	\includegraphics[width=0.8in]{test-4ellip-20pair-noise-2perc-mu-1-100-relu-mse-sgd-03-real-fix}&
%%	&
%	\includegraphics[width=0.8in]{test-4ellip-20pair-noise-5perc-mu-1-100-relu-mse-sgd-03-real-fix}\\
\end{tabular}
\caption{CNN-DDSM reconstruction for different $\mu$: $(\mu_0, \mu) = (0, 200)$ (left three columns) and $(1, 100)$ (right three columns)}
\label{fig:4ellip-diff-mu}
\end{figure}

\subsection{3D Reconstruction}

In this subsection, we apply the DDSM to 3D DOT problems. We consider a cubic domain $\Omega=(-1,1)\times (-1,1) \times (-1,1)$, and two ellipsoids with the level-set functions $\Gamma_i(x_1,x_2,x_3)$, $i=1,2$, which have the axis length, rotation angles and center points sampled from $\mathcal{U}(0.4,0.6)$, $\mathcal{U}(0, 2\pi)$ and $\left[ \mathcal{U}(-0.4,0.4) \right]^3$, respectively. Similar to \eqref{gamma_distrib}, we let the inclusions be generated by the following function involving the random variables described above
\begin{equation}
\label{gamma_distrib_3D}
\Gamma(x_1,x_2,x_3) = \min_{i=1,2} \{ \Gamma(x_1,x_2,x_3) \}.
\end{equation}
In the 3D case, we employ the first 9 spheric harmonic functions below and map them to the surface of the domain
\begin{equation}
\begin{split}
&\Re Y^m_l(\theta,\phi) = \sqrt{ \frac{2l+1}{4\pi} \frac{(l-m)!}{(l+m)!} } P^m_l(\cos(\theta)) \sin(m \phi) \\
&\Im Y^m_l(\theta,\phi) = \sqrt{ \frac{2l+1}{4\pi} \frac{(l-m)!}{(l+m)!} } P^m_l(\cos(\theta)) \cos(m \phi)
\end{split}
\end{equation}
to generate the flux boundary data, where $P^m_l(z)$ are the Legendre polynomials with $l=0,1,2$, $0\le m \le l$, of which some examples are plotted in Figure \ref{fig: harmon3D}. The other setting-ups are similar to the 2D case.

%\begin{figure}[h!]
%\centering
%\includegraphics[width = 0.2\textwidth]{./Harmon12}
%%\includegraphics[width = 0.2\textwidth]{./Harmon13}
%%\includegraphics[width = 0.2\textwidth]{./Harmon15}
%\includegraphics[width = 0.2\textwidth]{./Harmon14}
%\caption{Harmonic functions on the surface of $\Omega$.}
%\label{fig: harmon3D}
%\end{figure}

\begin{figure}[h]
\centering
\begin{minipage}{0.6\textwidth}
  \centering
\includegraphics[width = 0.4\textwidth]{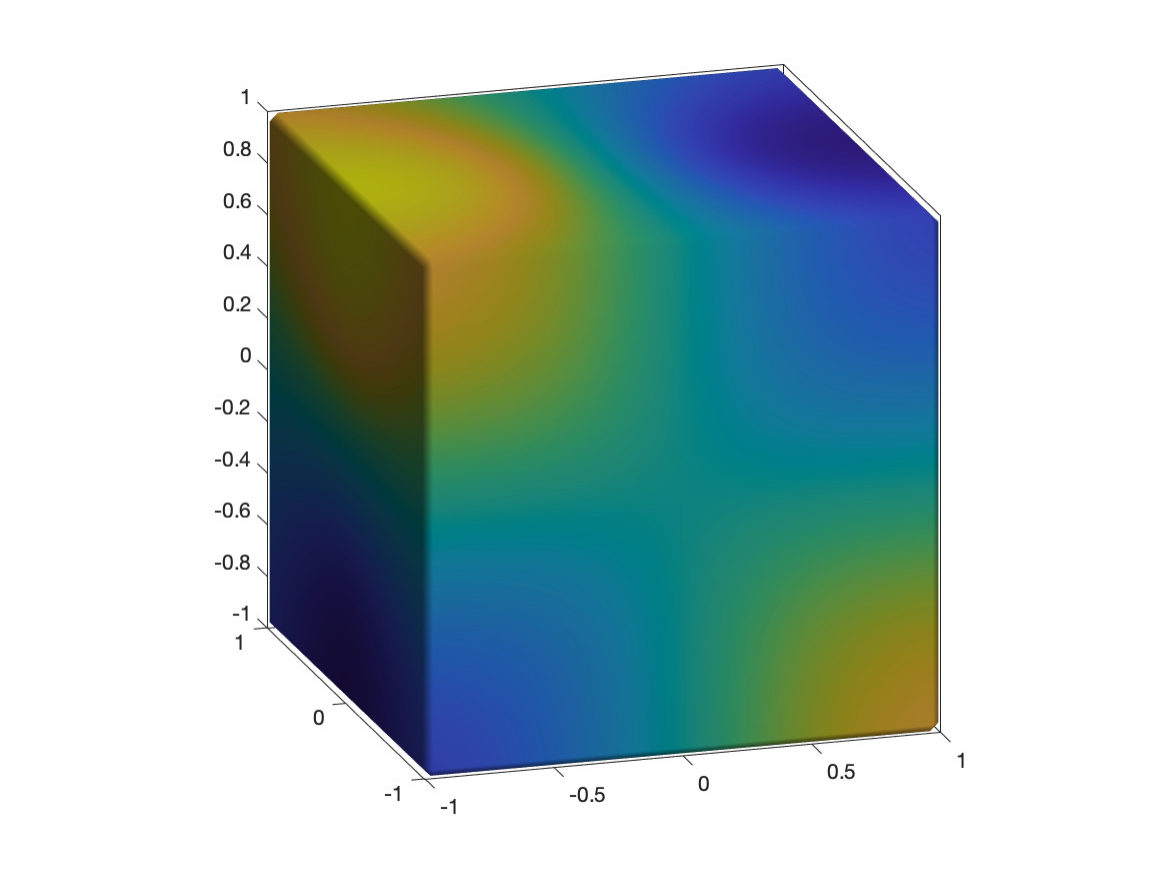}
\includegraphics[width = 0.4\textwidth]{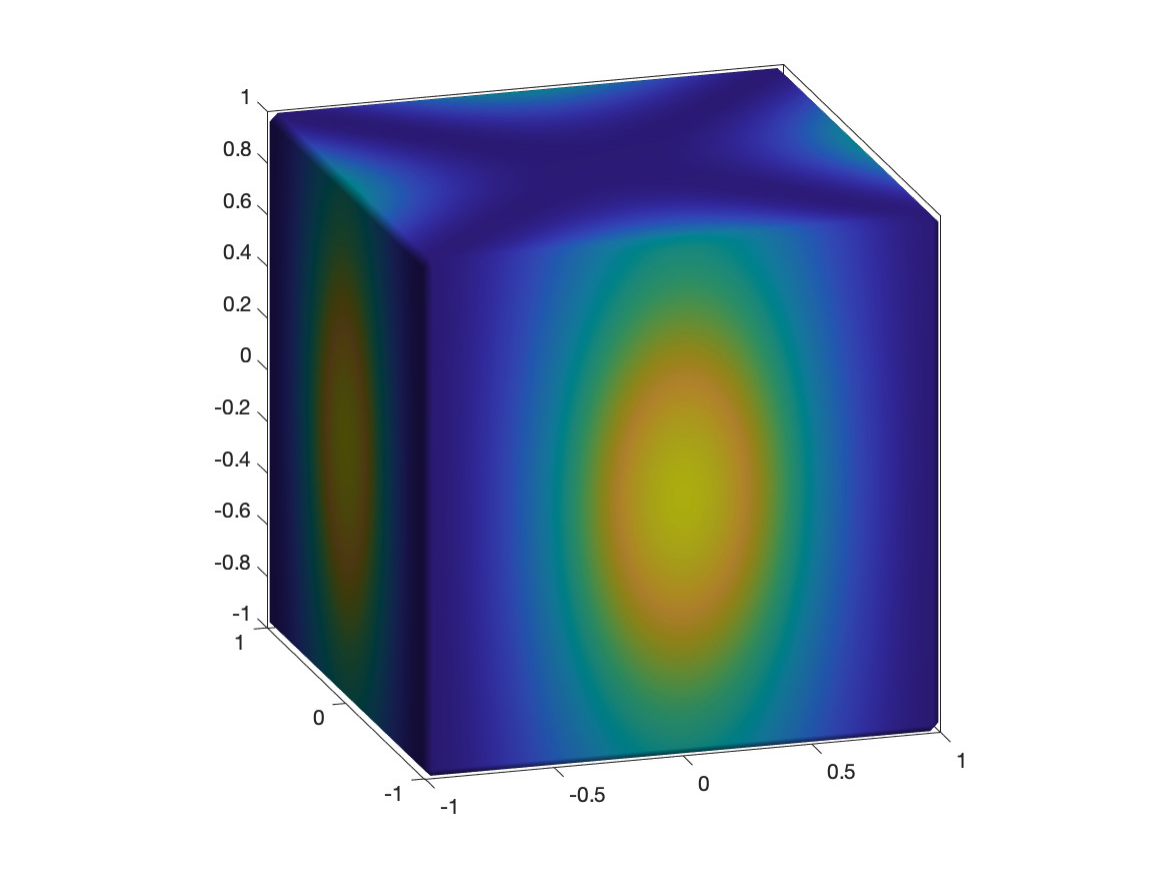}
\caption{Harmonic functions on the surface of $\Omega$.}
\label{fig: harmon3D}
\end{minipage}
\begin{minipage}{0.3\textwidth}
  \centering
   \includegraphics[width=1.5in]{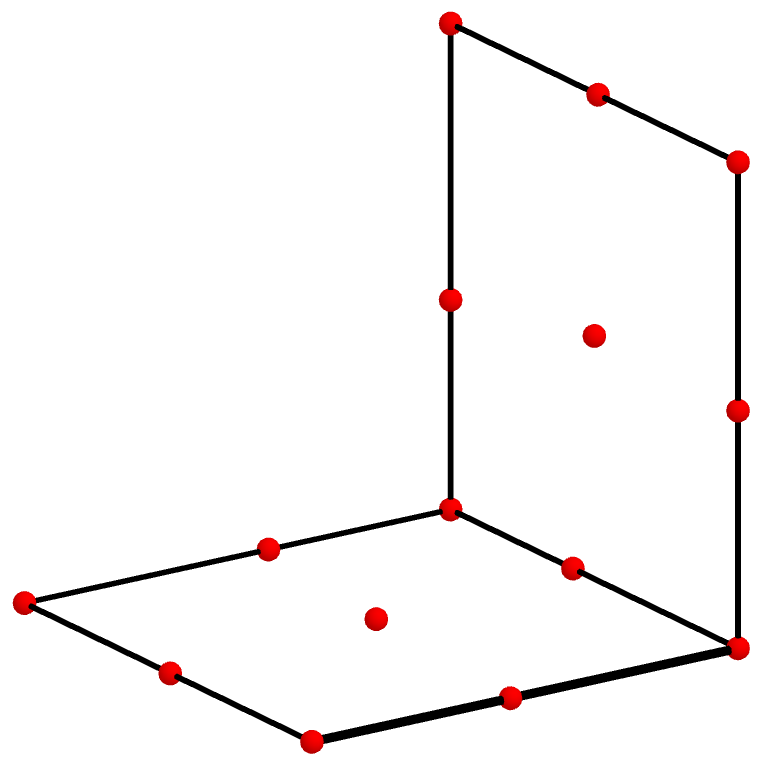}
  \caption{Data points on two faces of the cubic domain}
  \label{fig:cub_limtpts}
  \end{minipage}
\end{figure}

The reconstruction results of two typical cases are presented in Figure \ref{3D_ell}, where the two ellipsoids are glued together in the first case and separated in the second one. To show the reconstructed inclusion distribution, we employ the 3D density plots: the red, blue and mesh surface are corresponding to the isosurface with the values $0.75$, $0.06$ and $0.025$, respectively, which is the first row of each case. In addition, we also plot some cross sections which are given in the second row of each case. As shown by all these plots, the reconstruction results are quite accurate and also robust with respect to large noise. Note that for 3D DOT problems, solving 3D forward problems for iterative methods can be highly expensive. Given the efficiency, accuracy and robustness of the proposed DDSM, we believe it can be very attractive for further real-world applications. 

For the second case, we also test the performance of the DDSM for that the boundary data are only available at a few points not every grid point. Specifically, as only the low-frequency data are used for the 3D examples, we assume there are only 9 points evenly distributed on each face of the domain, see Figure \ref{fig:cub_limtpts} for an illustration, then there are only 26 points distributed on the boundary. Similar to the 2D case, linear interpolation is used to generate data functions on each face. The reconstruction results are presented in Figure \ref{3D_limtdata}, where we can see the shape can be recovered quite well even though the data is extremely limited.

\begin{figure}[htbp]
\begin{tabular}{ >{\centering\arraybackslash}m{1.3in}>{\centering\arraybackslash}m{1.3in} >{\centering\arraybackslash}m{1.3in}  >{\centering\arraybackslash}m{1.3in}   }
	\centering
	True coefficients &
          $\delta=0$ &
          $\delta=5\%$ &
          $\delta=10\%$ \\
%	  $\delta=5\%$ &
%	  $\delta=10\%$ \\
	\includegraphics[width=1.3in]{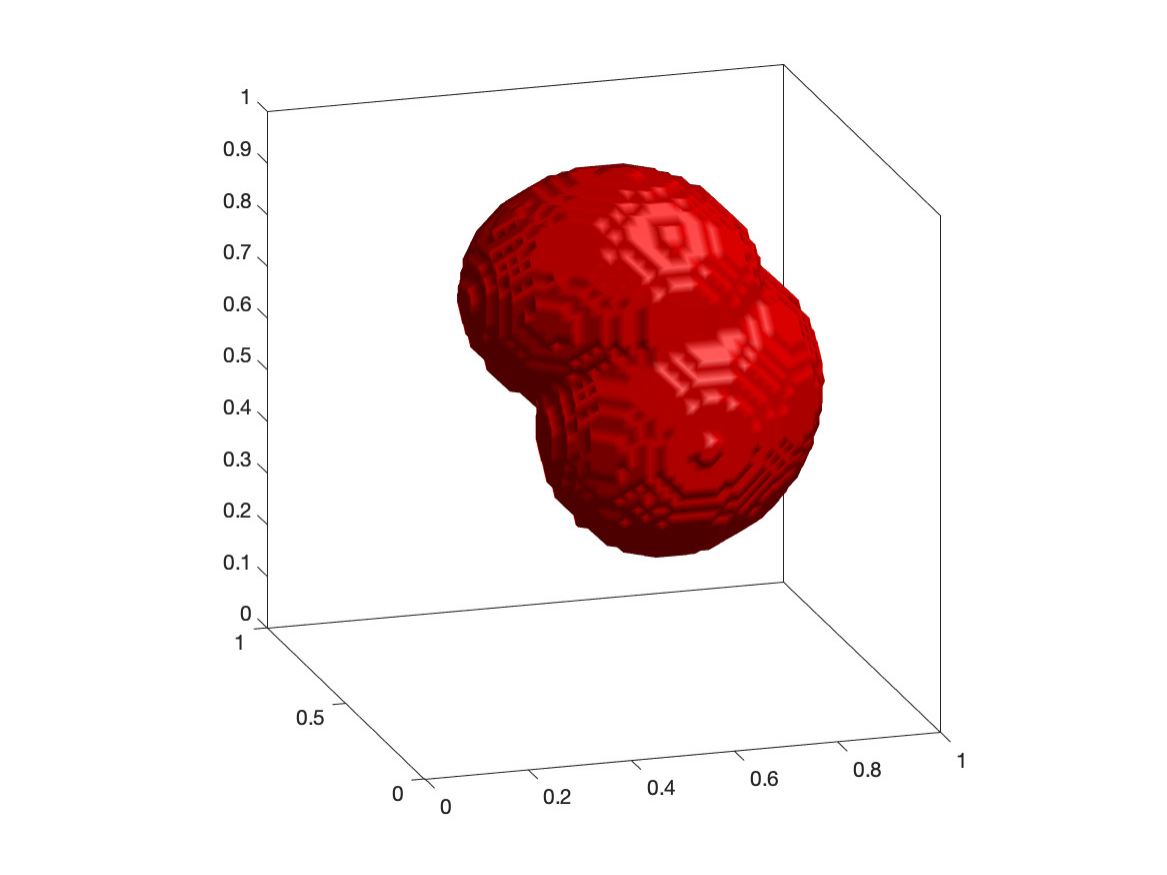}&
	\includegraphics[width=1.3in]{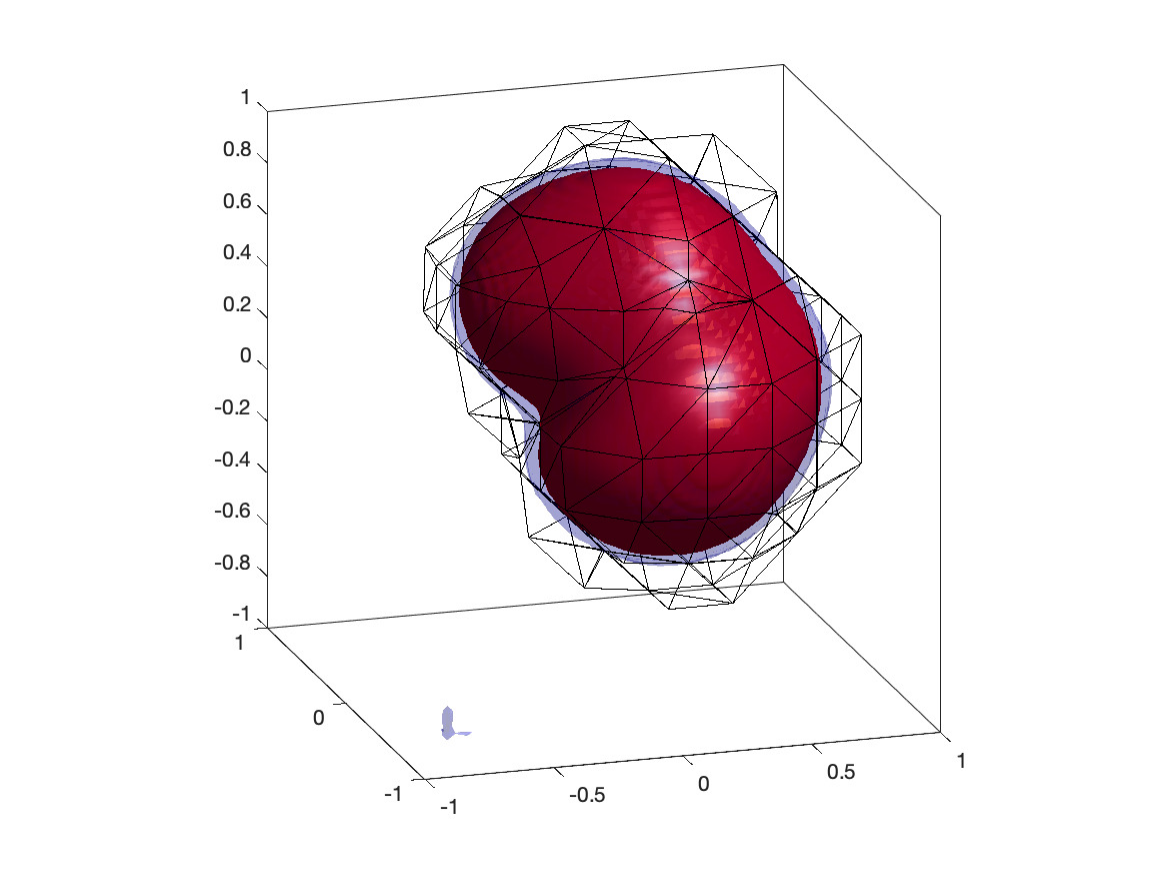}&
	\includegraphics[width=1.3in]{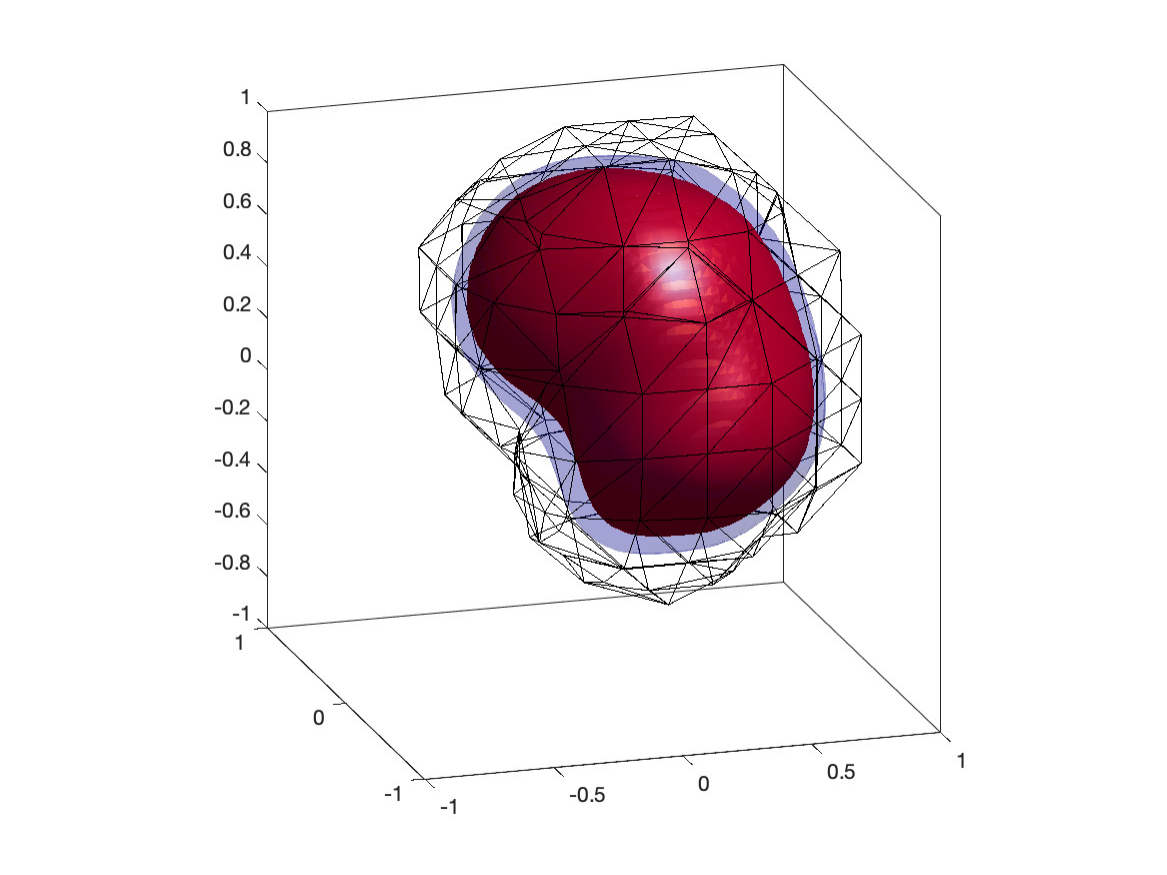}&
	\includegraphics[width=1.3in]{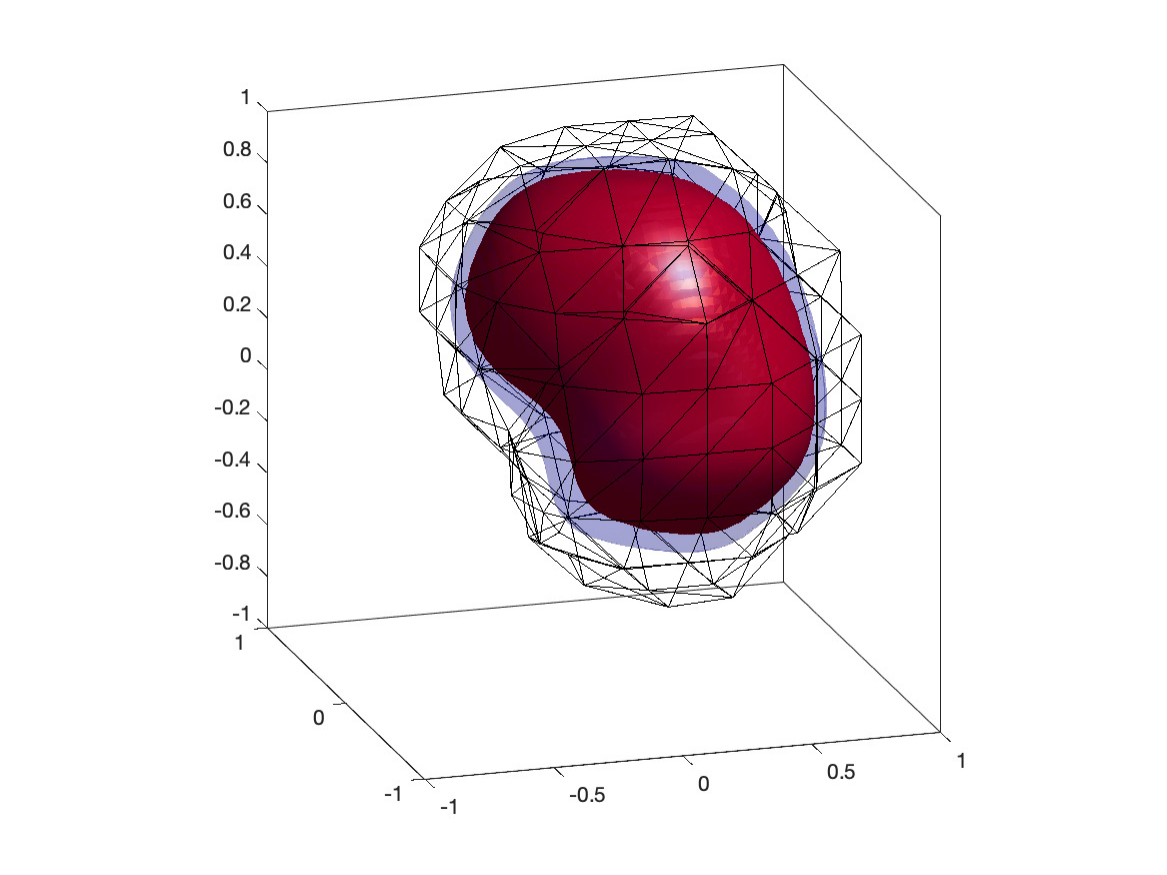}\\
         \includegraphics[width=1.3in]{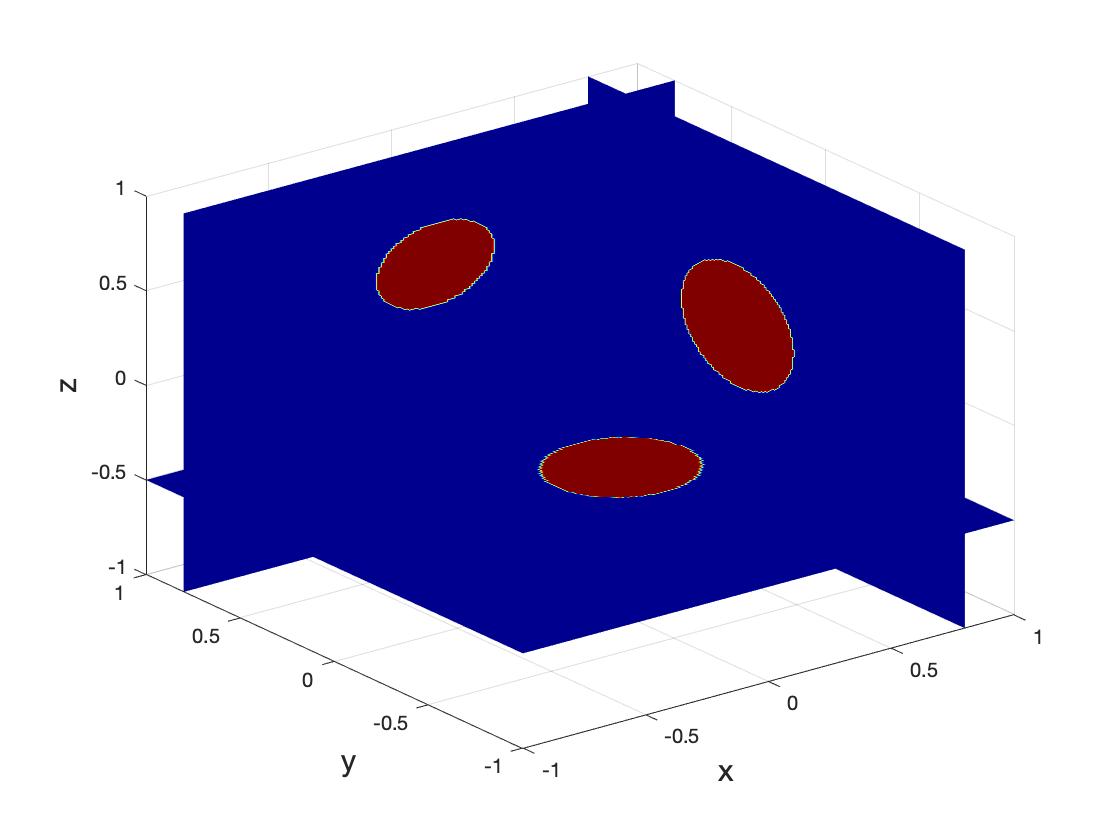}&
	\includegraphics[width=1.3in]{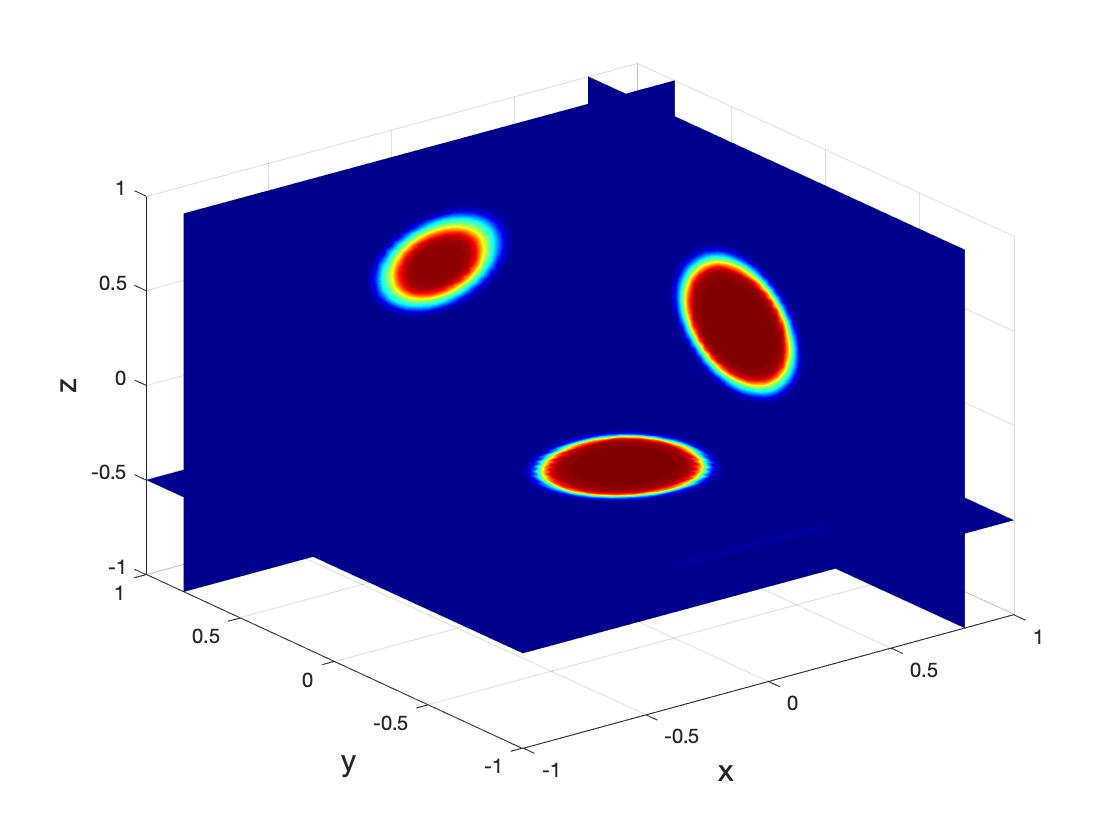}&
	\includegraphics[width=1.3in]{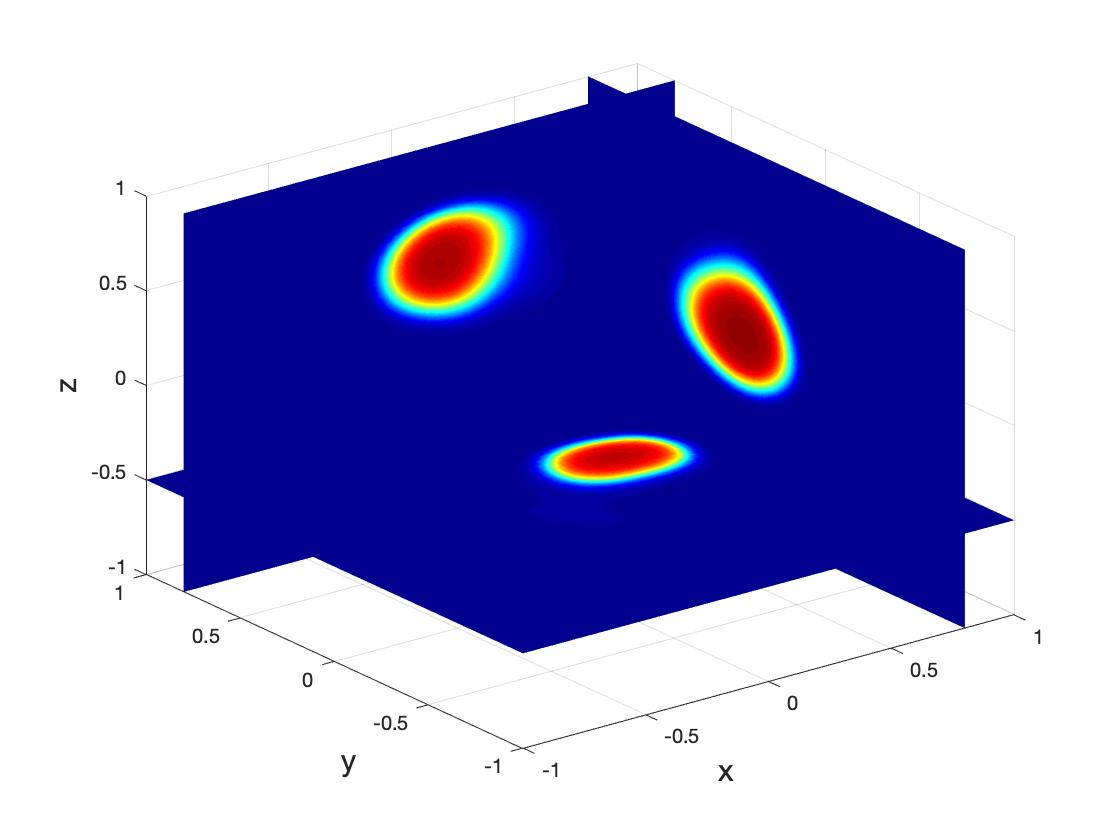}&
	\includegraphics[width=1.3in]{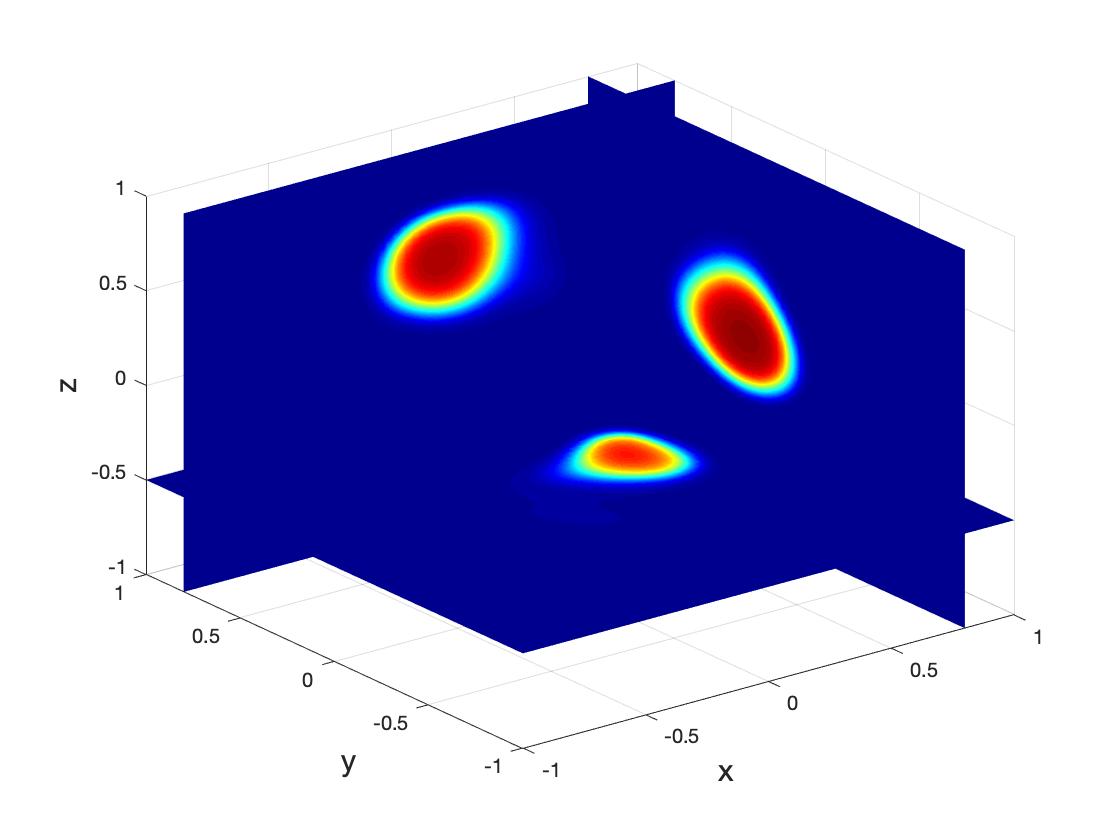}\\
	\includegraphics[width=1.3in]{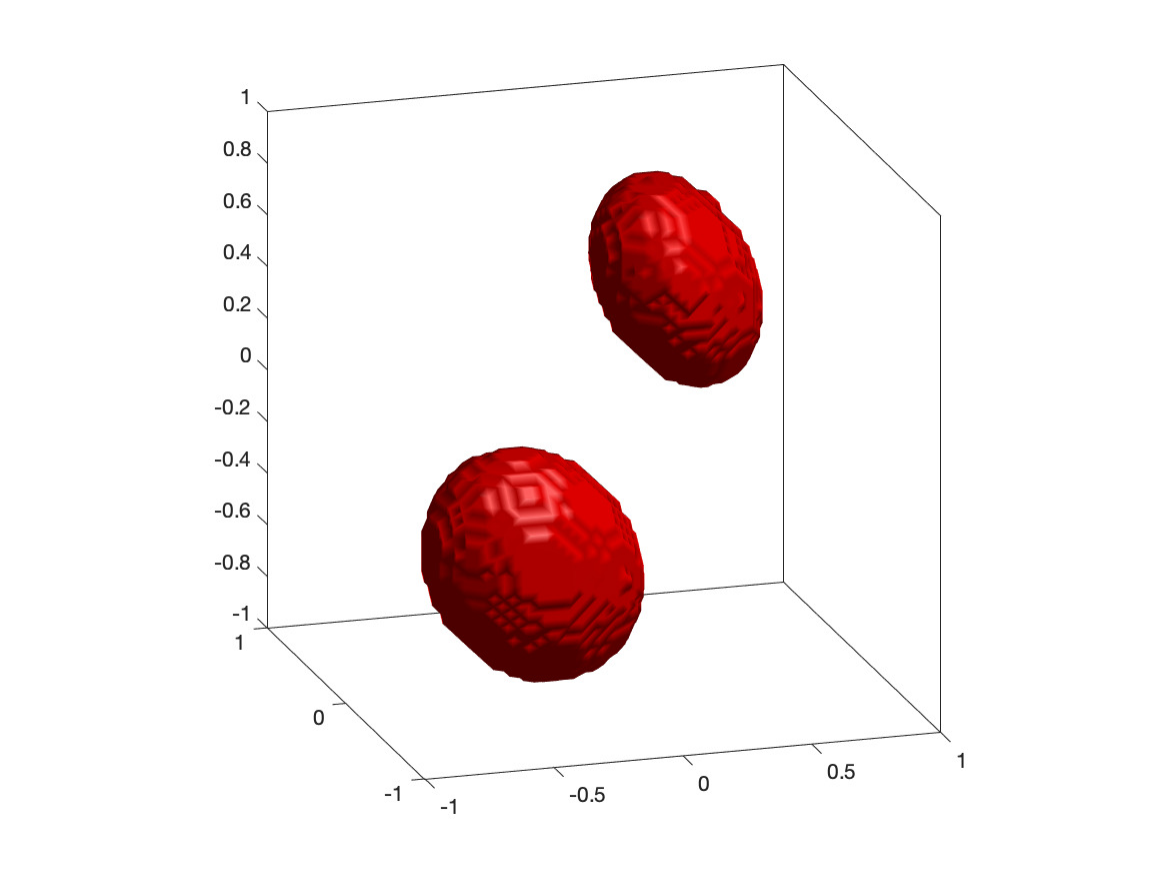}&
	\includegraphics[width=1.3in]{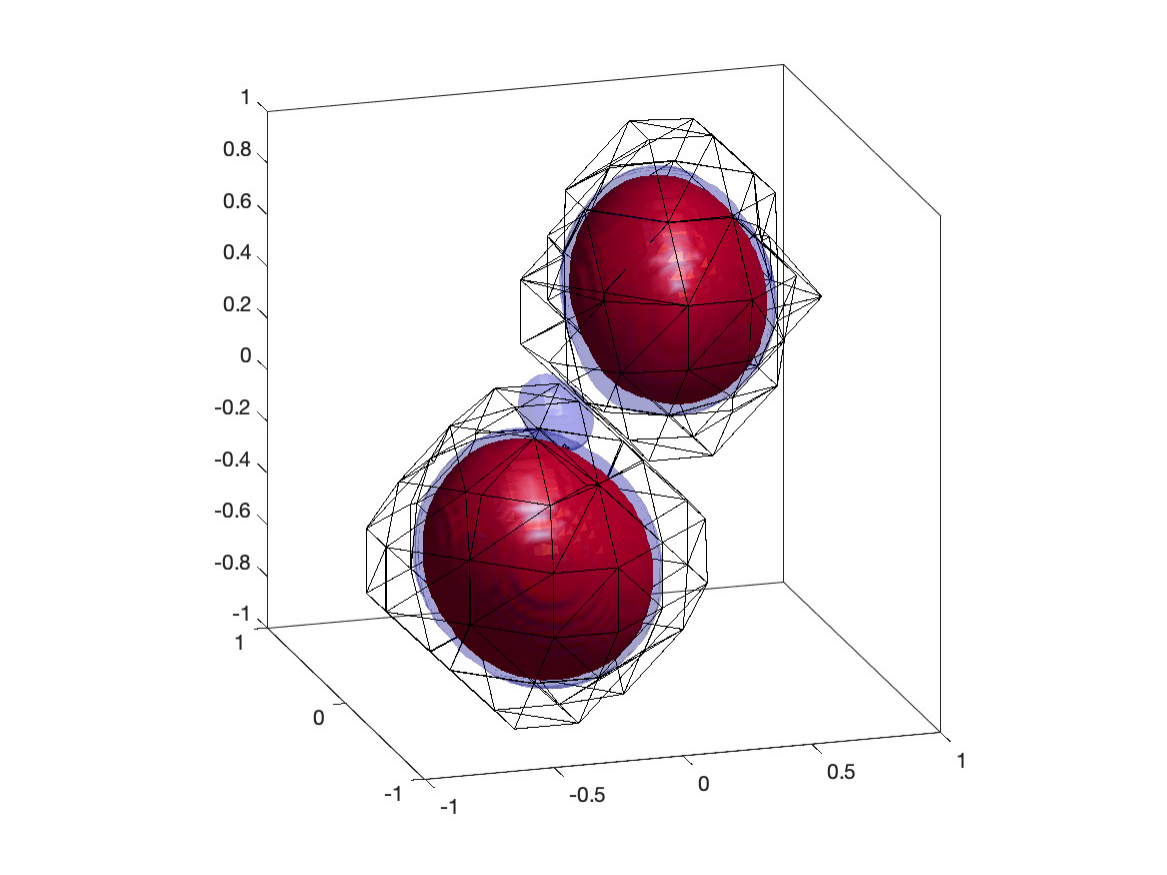}&
	\includegraphics[width=1.3in]{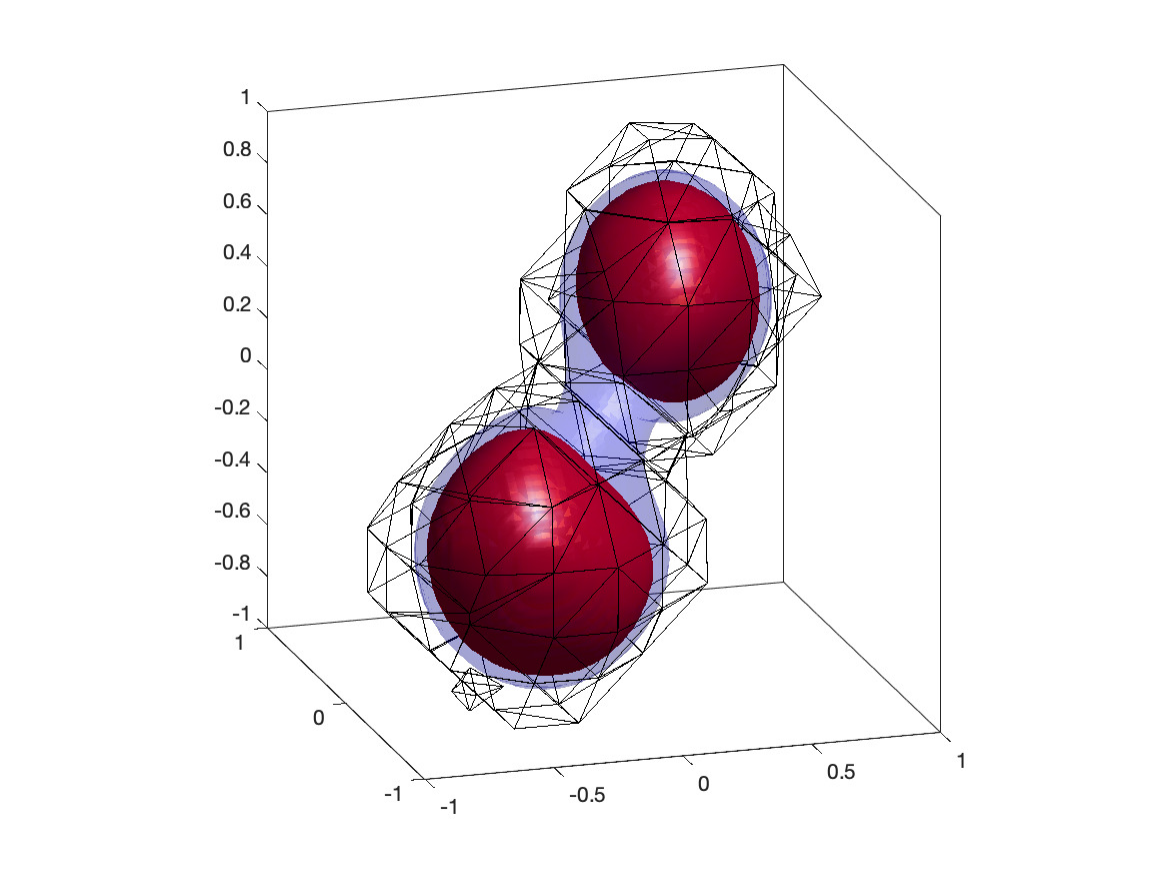}&
	\includegraphics[width=1.3in]{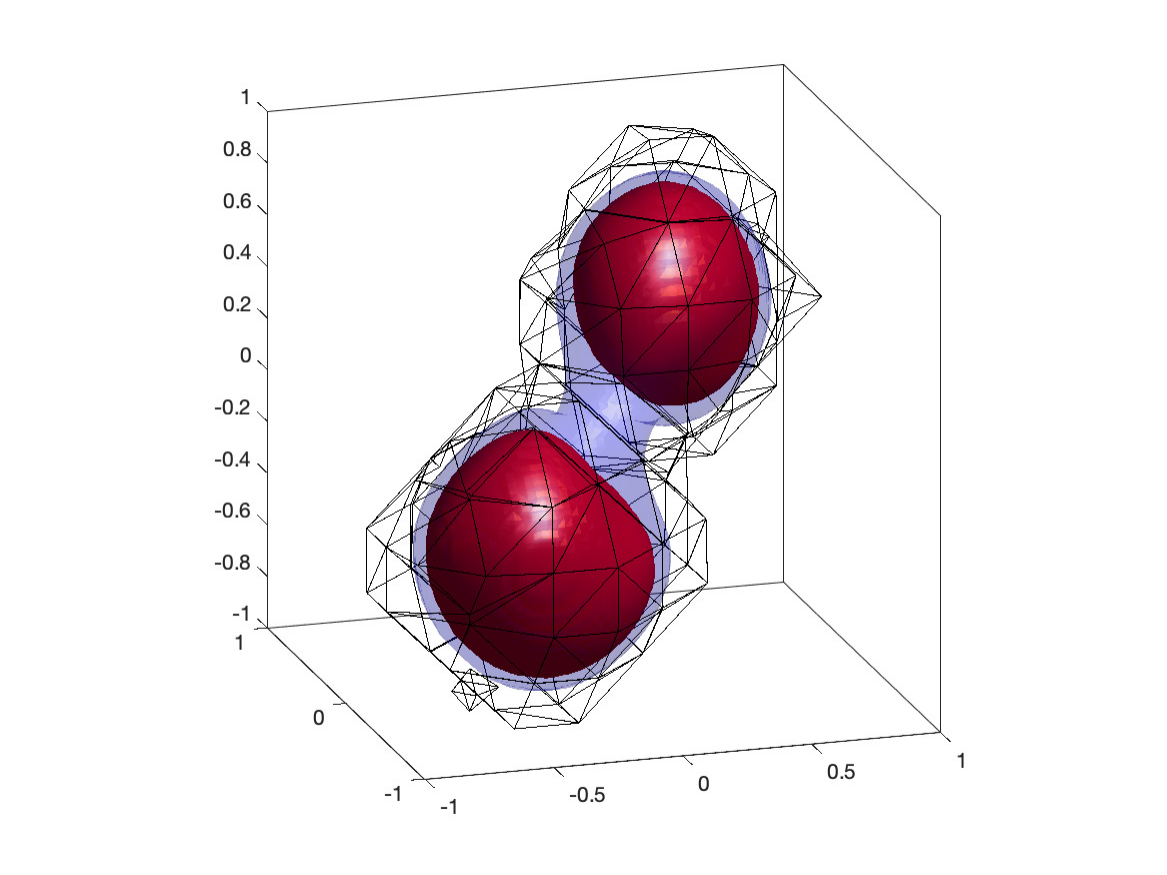}\\
	\includegraphics[width=1.3in]{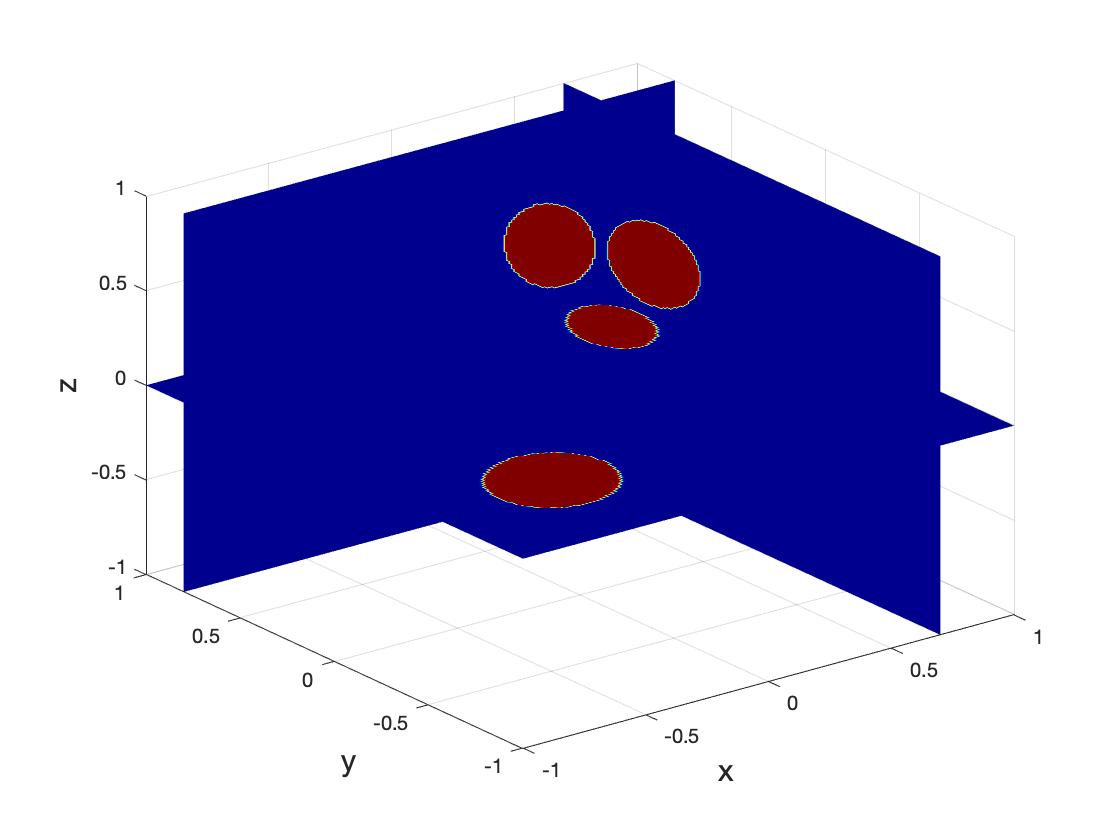}&
	\includegraphics[width=1.3in]{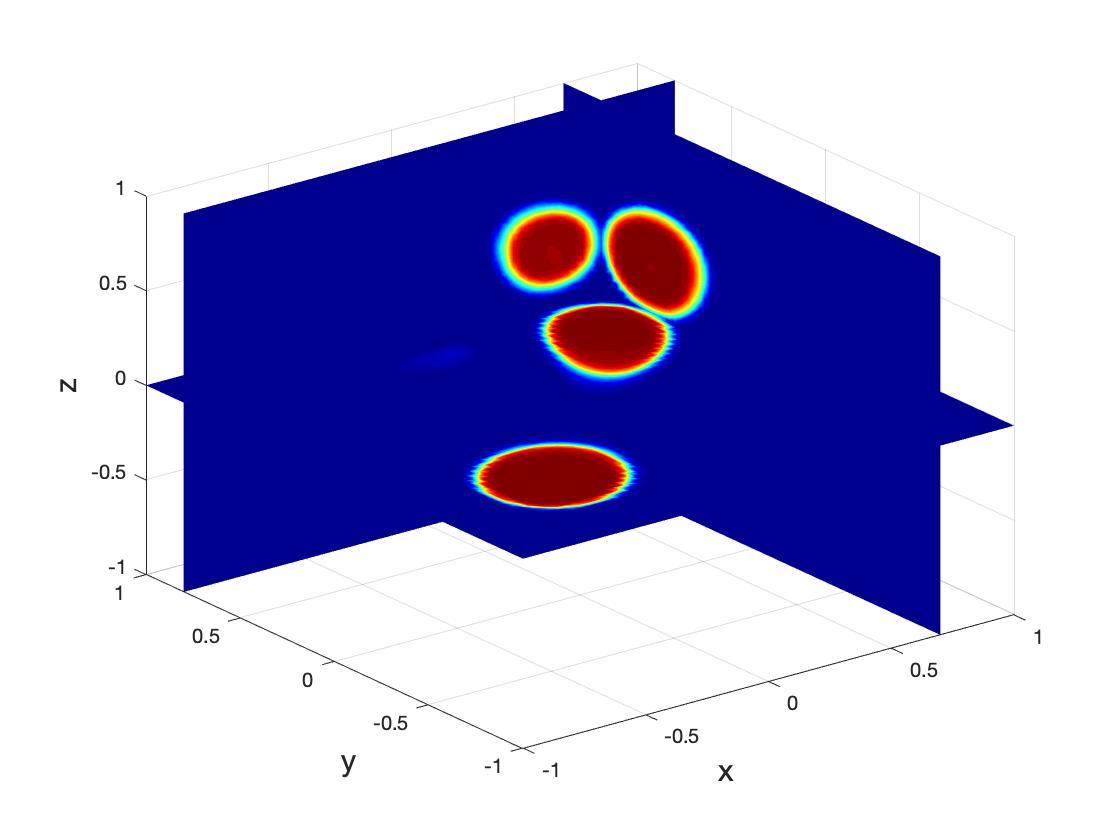}&
	\includegraphics[width=1.3in]{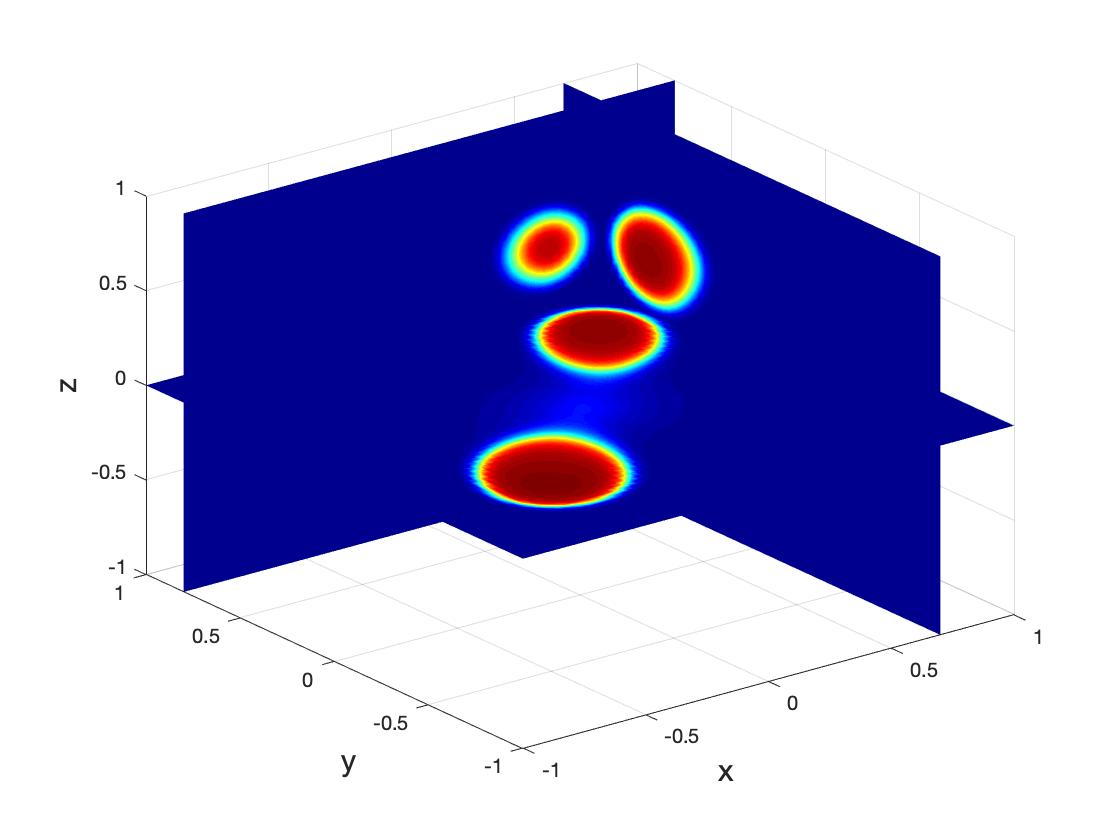}&
	\includegraphics[width=1.3in]{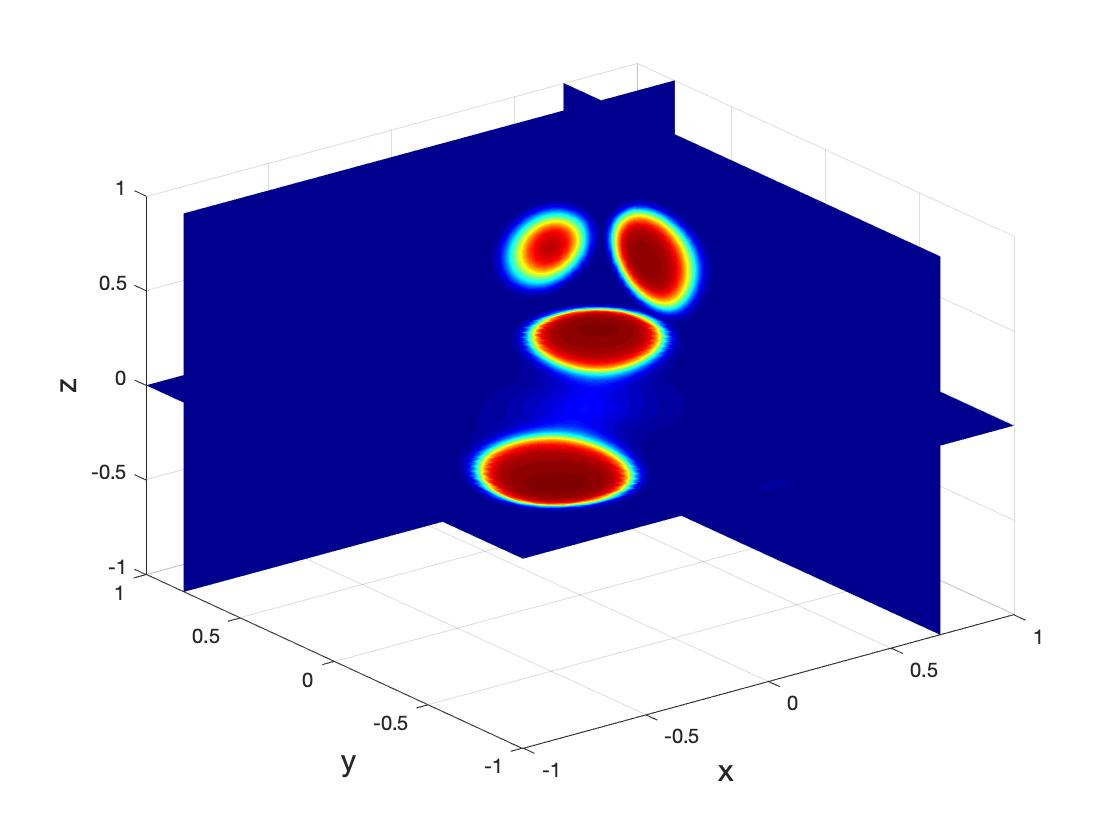}\\	
%	\includegraphics[width=0.8in]{test-cir-20pair-noise-5perc-relu-mse-sgd-04-real}&
%	\includegraphics[width=0.8in]{test-cir-20pair-noise-10perc-relu-mse-sgd-04-real}\\
%	%%
%	\includegraphics[width=0.8in]{true-cir-07}&
%	\includegraphics[width=0.8in]{test-cir-1pair-noise-0perc-relu-mse-sgd-07}&
%	\includegraphics[width=0.8in]{test-cir-10pair-noise-0perc-relu-mse-sgd-07}&
%	\includegraphics[width=0.8in]{test-cir-20pair-noise-0perc-relu-mse-sgd-07}&
%	\includegraphics[width=0.8in]{test-cir-20pair-noise-5perc-relu-mse-sgd-07-real}&
%	\includegraphics[width=0.8in]{test-cir-20pair-noise-10perc-relu-mse-sgd-07-real}\\
%	%%
%	\includegraphics[width=0.8in]{true-cir-16}&
%	\includegraphics[width=0.8in]{test-cir-1pair-noise-0perc-relu-mse-sgd-16}&
%	\includegraphics[width=0.8in]{test-cir-10pair-noise-0perc-relu-mse-sgd-16}&
%	\includegraphics[width=0.8in]{test-cir-20pair-noise-0perc-relu-mse-sgd-16}&
%	\includegraphics[width=0.8in]{test-cir-20pair-noise-5perc-relu-mse-sgd-16-real}&
%	\includegraphics[width=0.8in]{test-cir-20pair-noise-10perc-relu-mse-sgd-16-real}\\
	%%
\end{tabular}
\caption{Reconstruction for 2 cases with different Cauchy data number and noise level: Case 1(top) and Case 2(bottom) } 
\label{3D_ell}
\end{figure}

\begin{figure}[htbp]
\begin{tabular}{ >{\centering\arraybackslash}m{1.3in}>{\centering\arraybackslash}m{1.3in} >{\centering\arraybackslash}m{1.3in}  >{\centering\arraybackslash}m{1.3in}   }
	\centering
	True coefficients &
          $\delta=0$ &
          $\delta=5\%$ &
          $\delta=10\%$ \\
%	  $\delta=5\%$ &
%	  $\delta=10\%$ \\
	\includegraphics[width=1.3in]{3D_case3_true}&
	\includegraphics[width=1.3in]{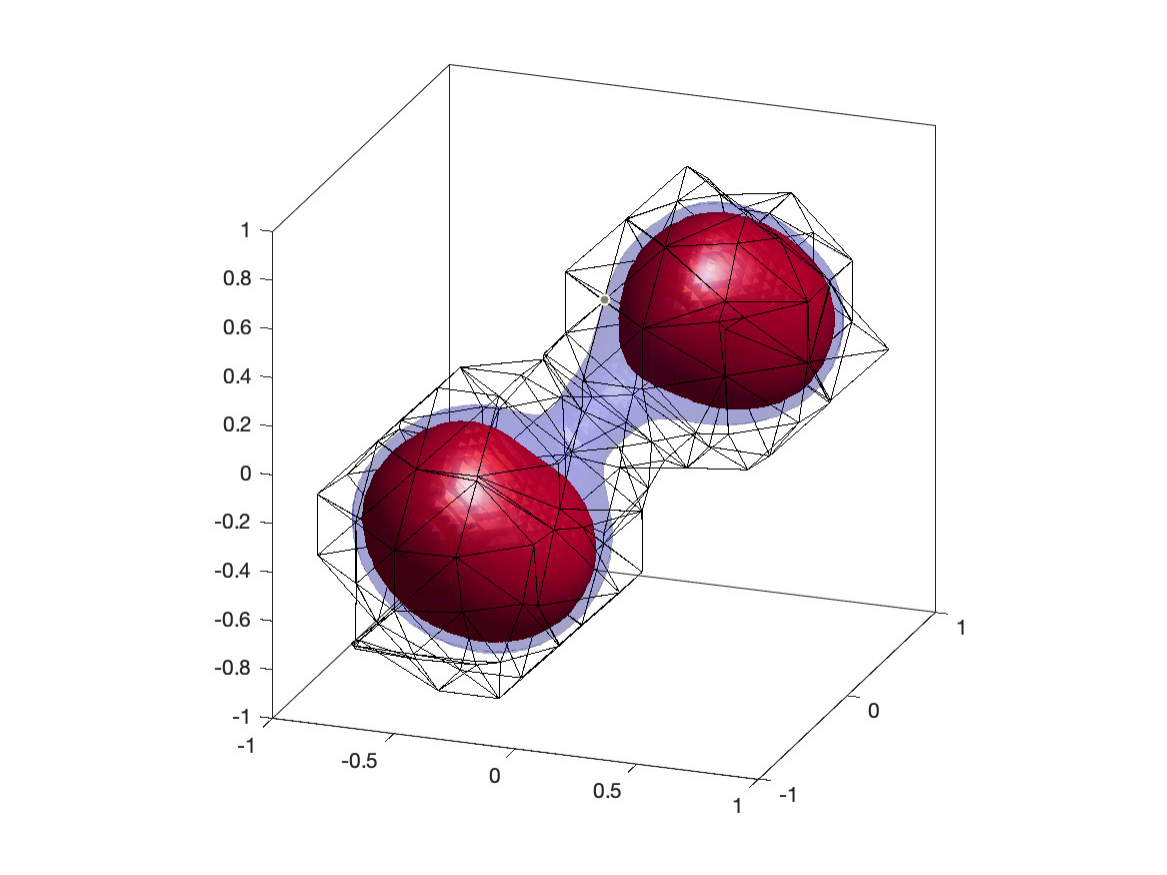}&
	\includegraphics[width=1.3in]{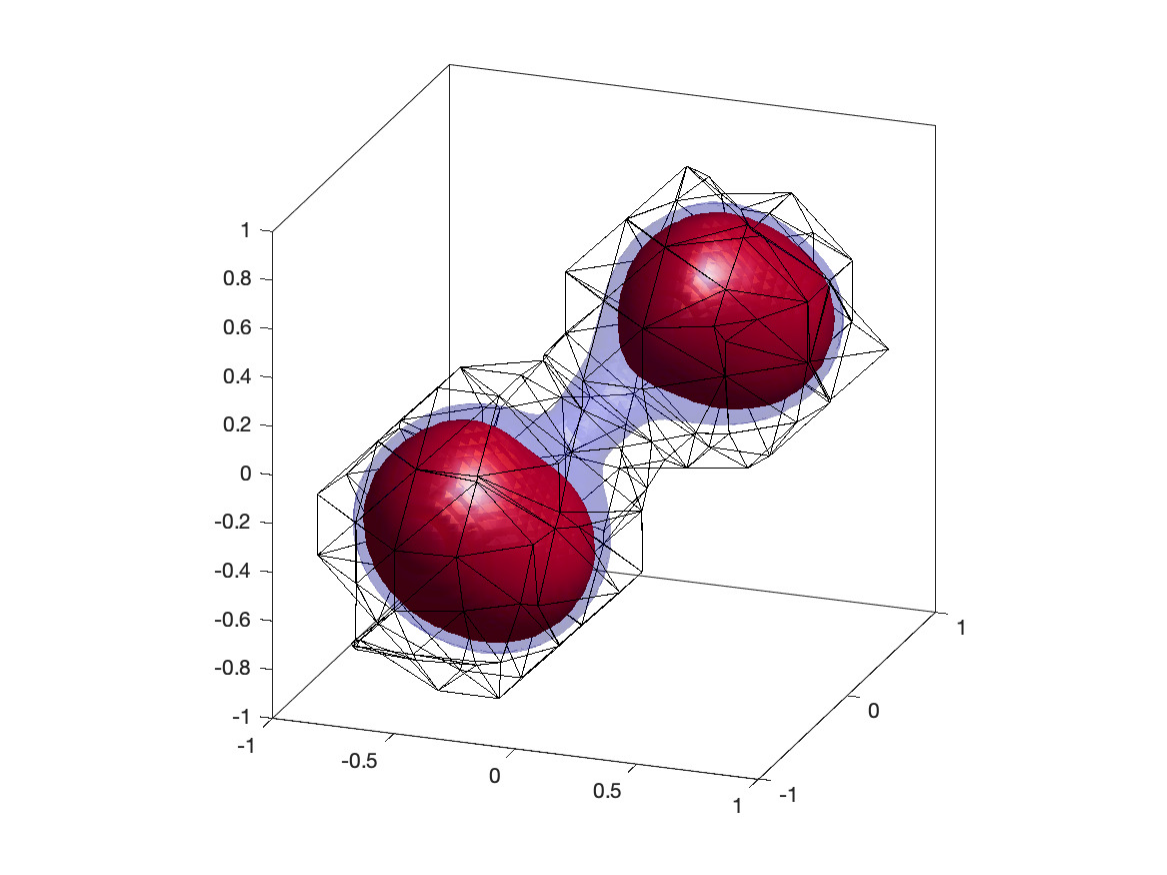}&
	\includegraphics[width=1.3in]{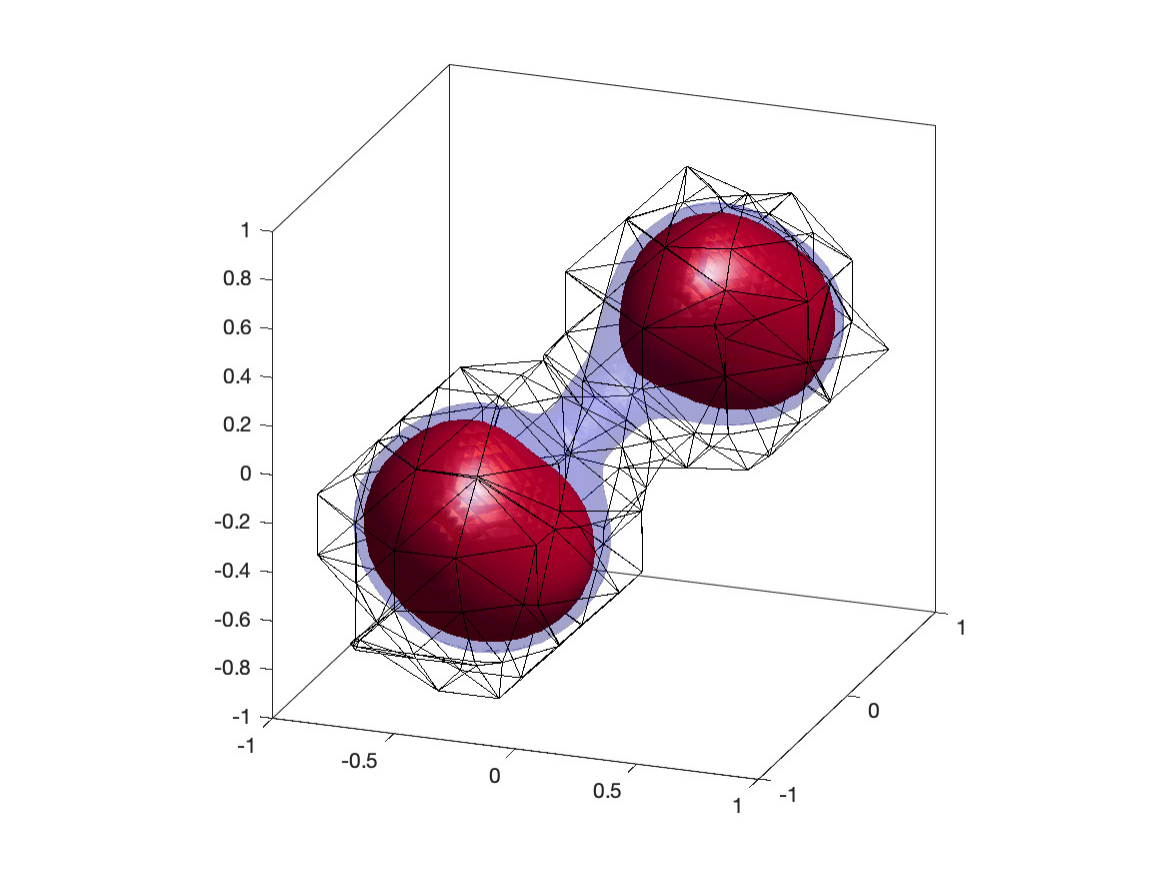}\\
         \includegraphics[width=1.3in]{I_true_1SPL_3d_2ellipsoid_divided_201_01}&
	\includegraphics[width=1.3in]{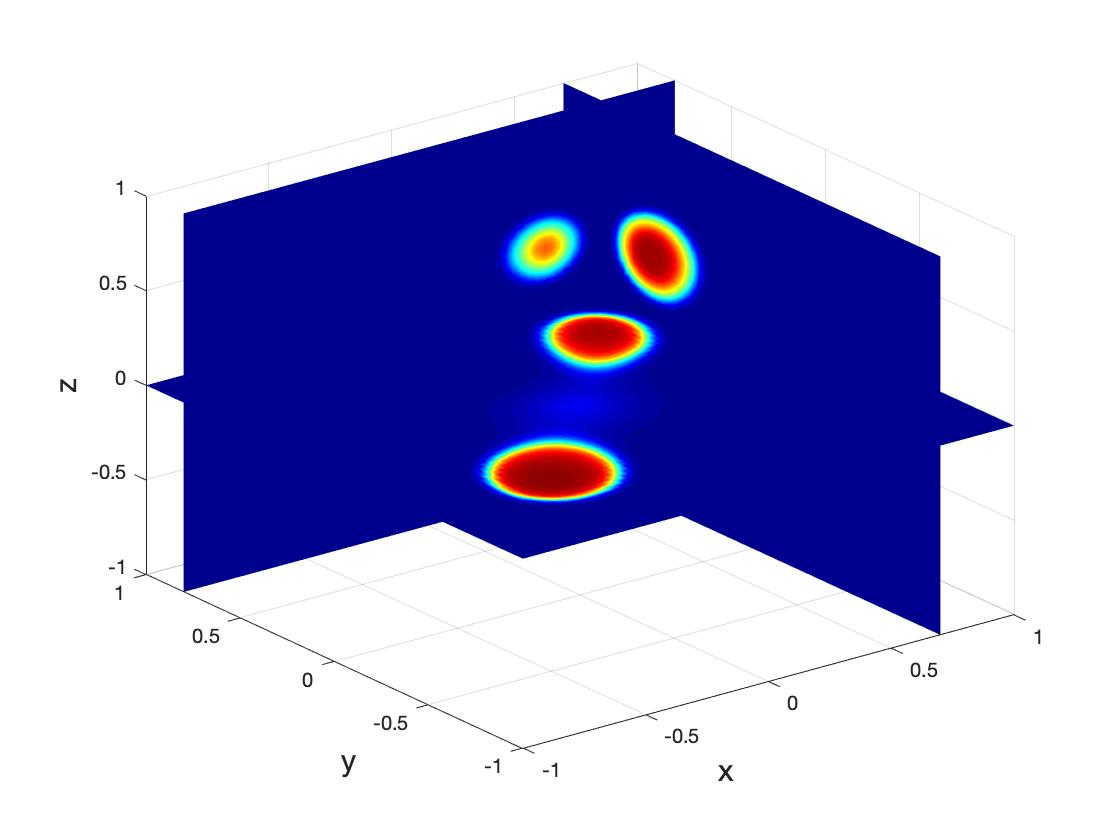}&
	\includegraphics[width=1.3in]{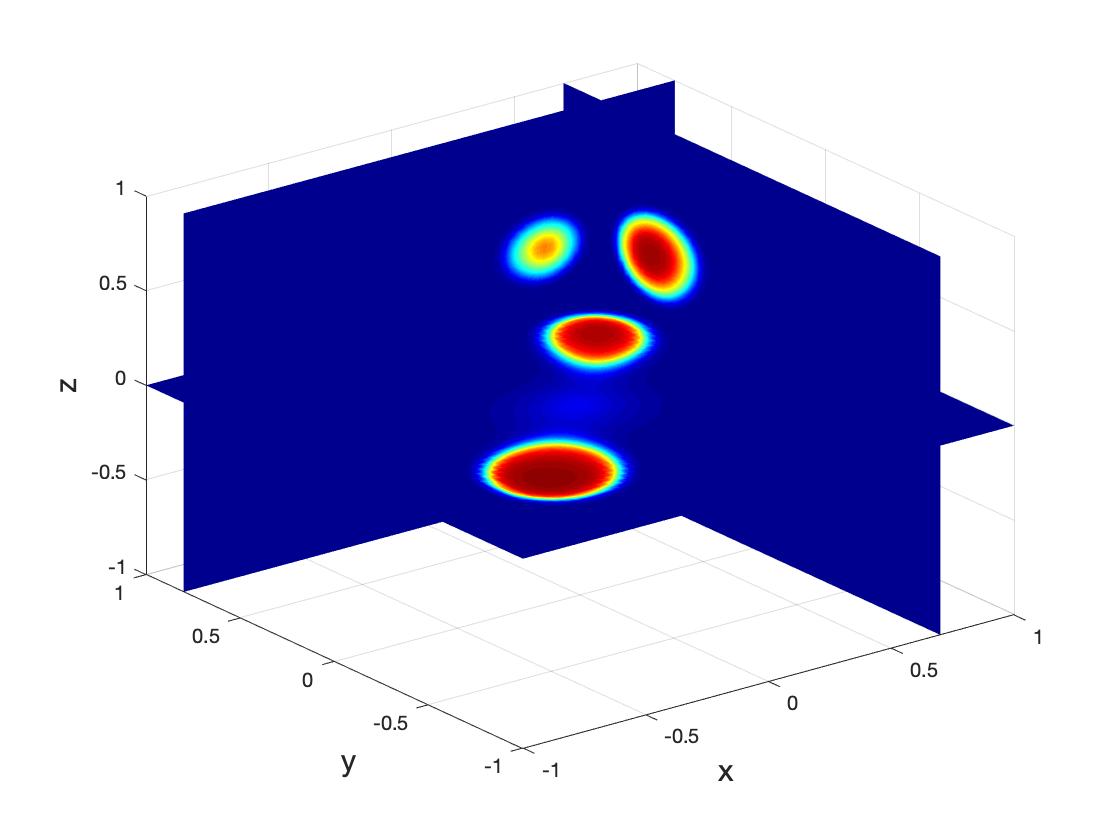}&
	\includegraphics[width=1.3in]{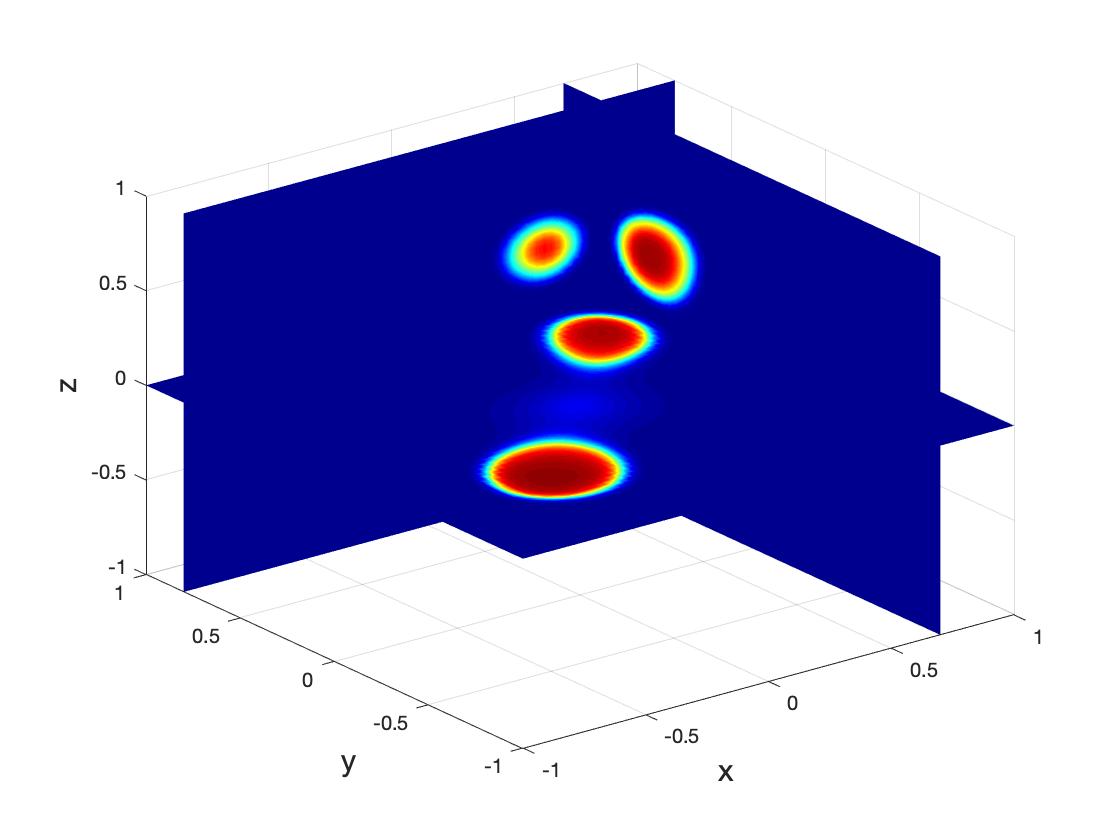}\\
\end{tabular}
\caption{Reconstruction with limited data points on the boundary. } 
\label{3D_limtdata}
\end{figure}

Moreover, we apply the DNN to an inclusion which has a duck shape that is not in the scope of the training data set, i.e., it is certainly not a union of two ellipsoids. In this case, we can still observe that the reconstruction can capture the basic duck shape which is still robust with respect to the large noise. But the geometry of the beak is lost as the training data set is only sampled from two ellipsoids. It can be expected that more detailed geometry can not be recovered accurately. However, one can still add more inclusion samples with various geometric randomness to the training data set in order to attain better accuracy.

\begin{figure}[htbp]
\begin{tabular}{ >{\centering\arraybackslash}m{1.3in}>{\centering\arraybackslash}m{1.3in} >{\centering\arraybackslash}m{1.3in}  >{\centering\arraybackslash}m{1.3in}   }
	\centering
	True coefficients &
          $\delta=0\%$ &
          $\delta=5\%$ &
          $\delta=10\%$ \\
%	  $\delta=5\%$ &
%	  $\delta=10\%$ \\
	\includegraphics[width=1.3in]{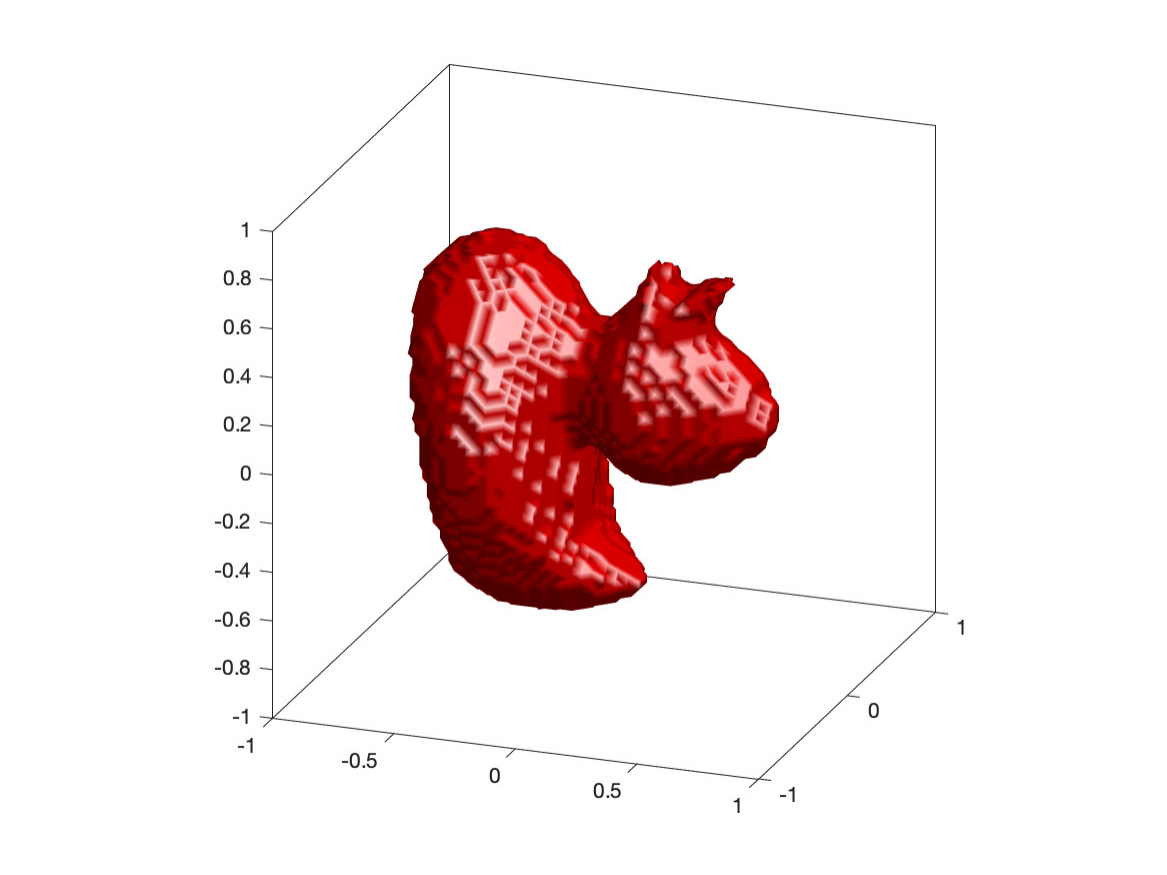}&
	\includegraphics[width=1.3in]{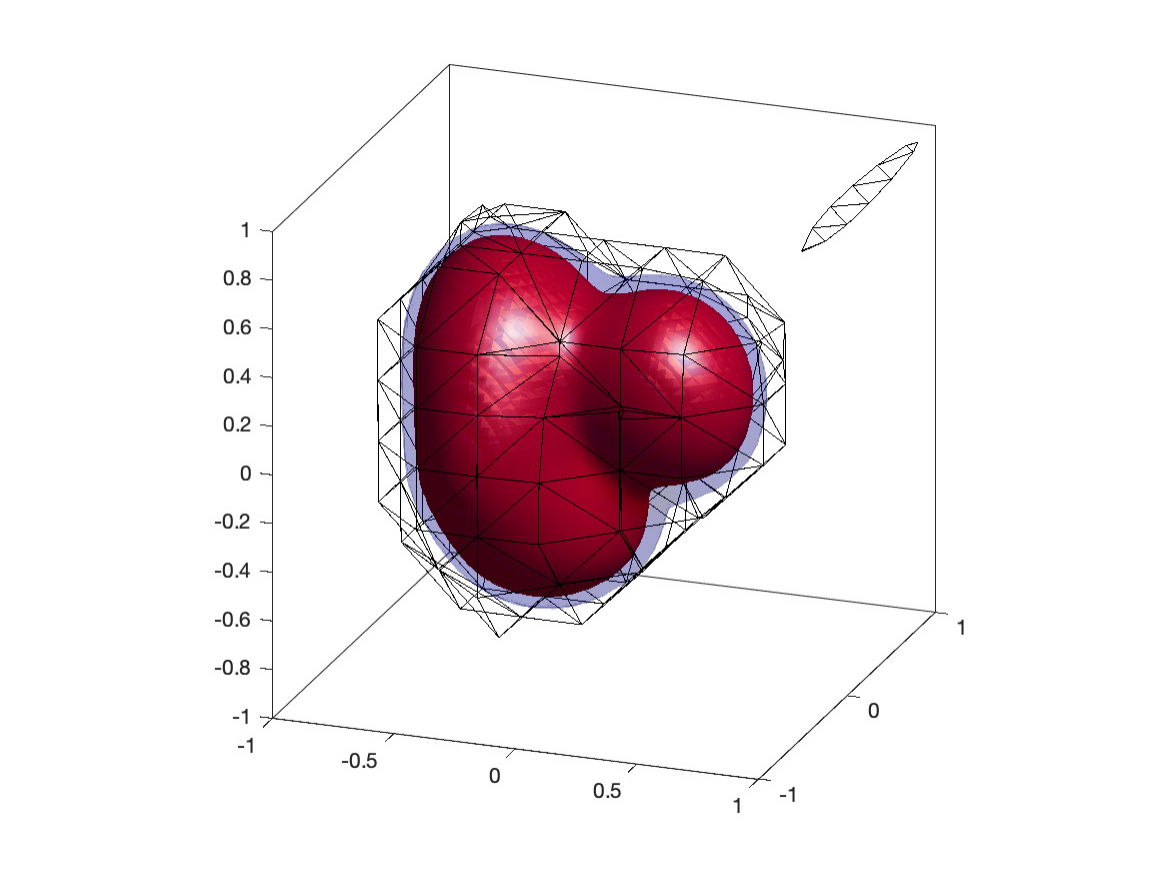}&
	\includegraphics[width=1.3in]{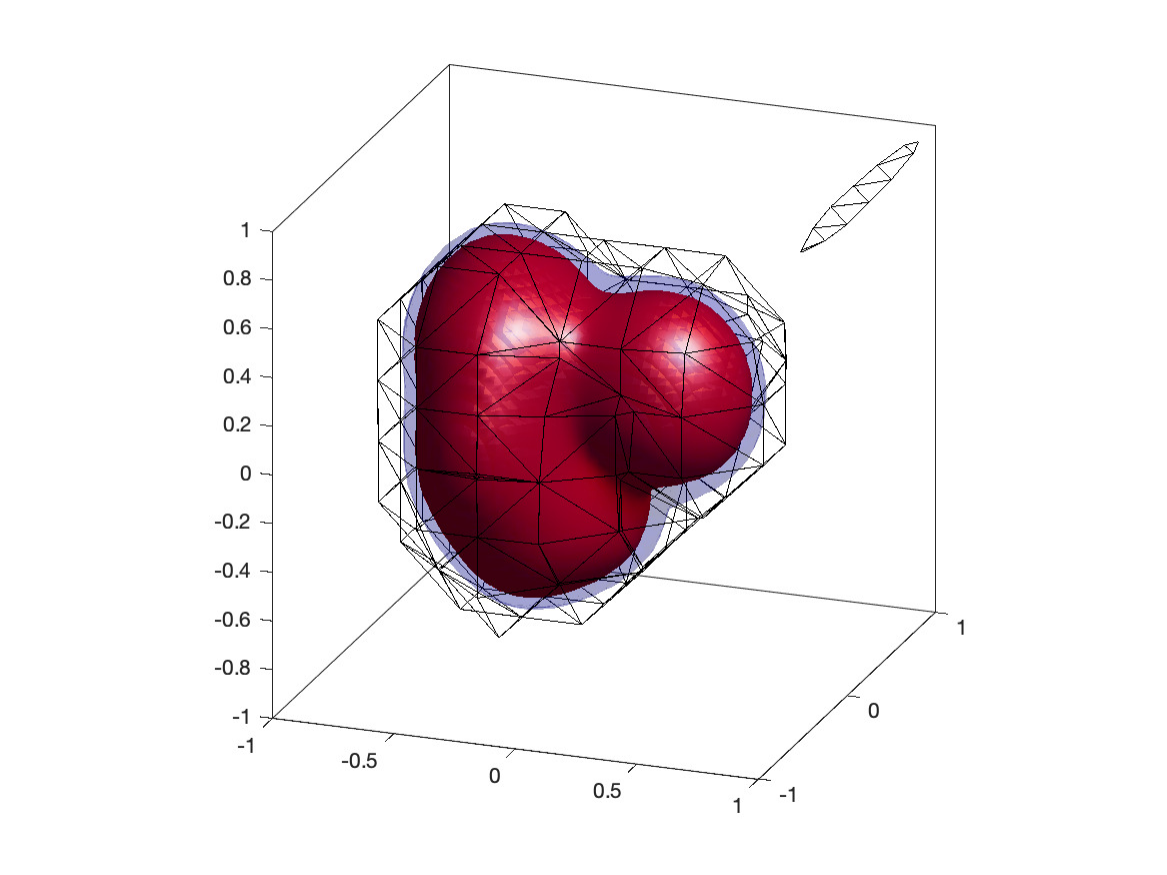}&
	\includegraphics[width=1.3in]{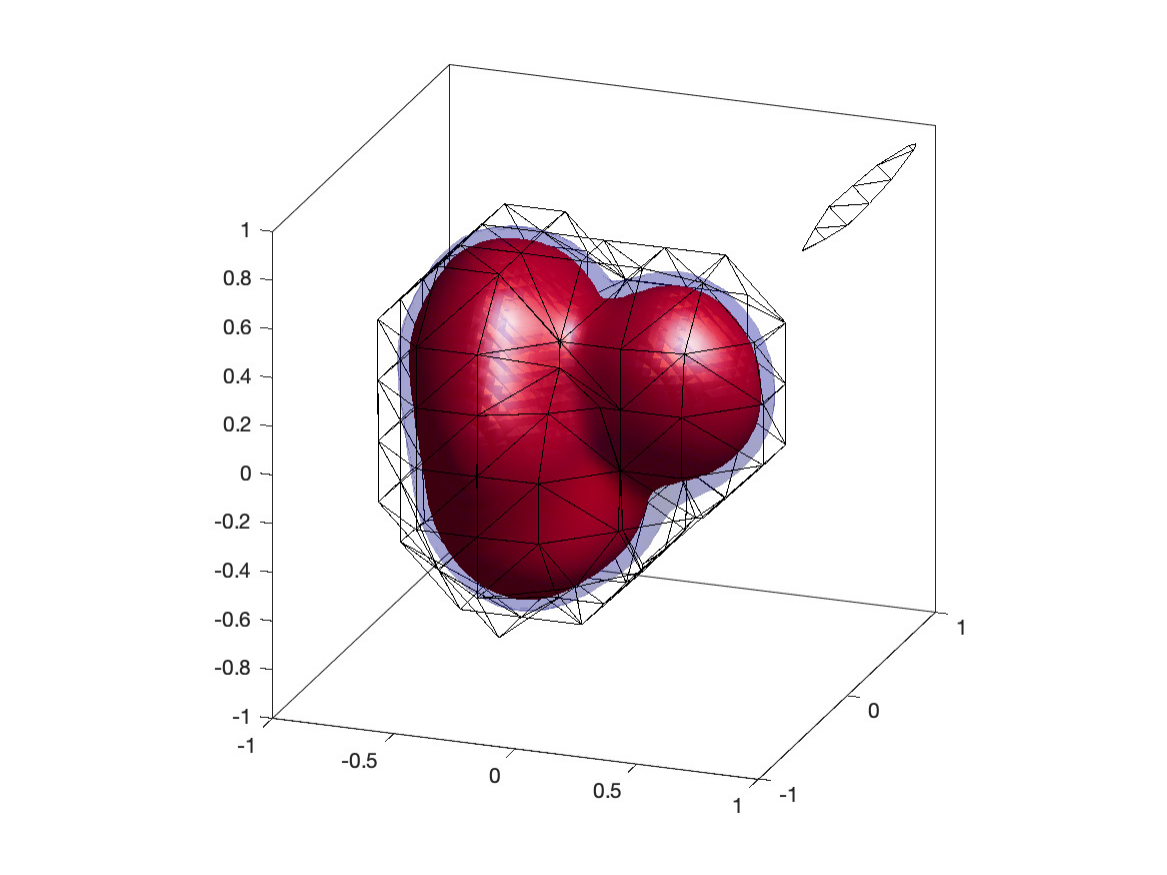}\\
	\includegraphics[width=1.3in]{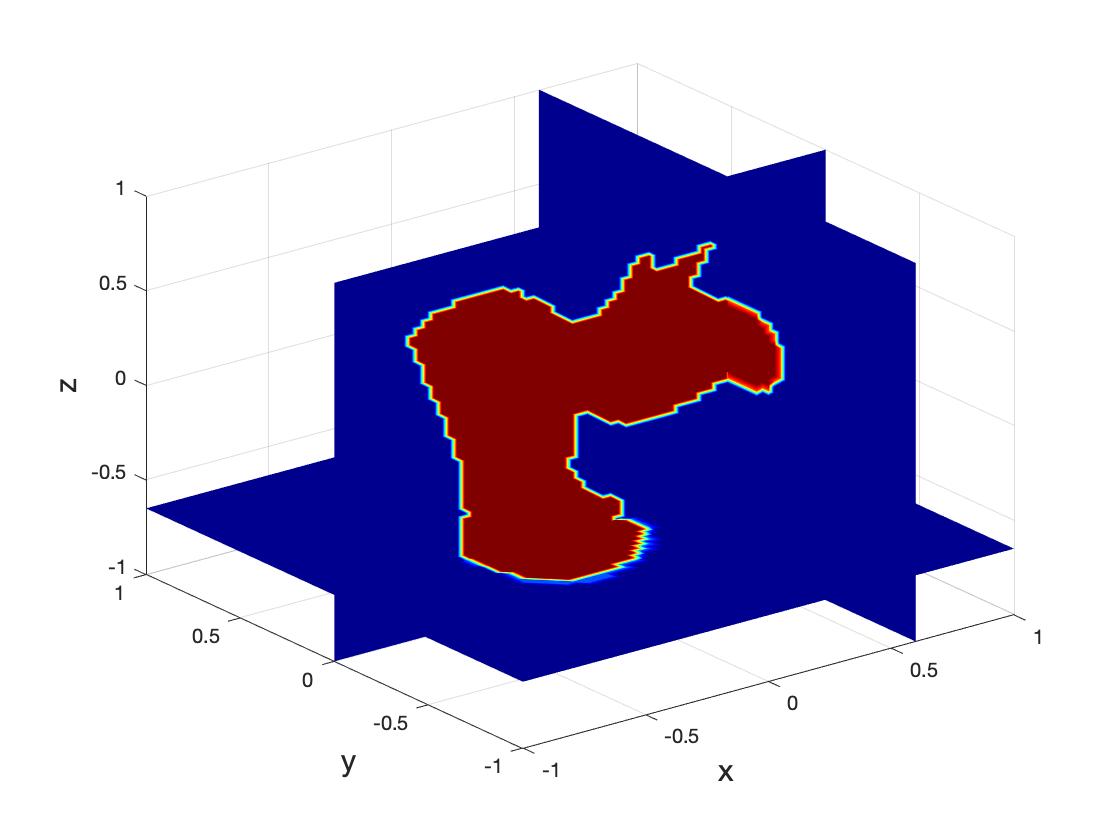}&
	\includegraphics[width=1.3in]{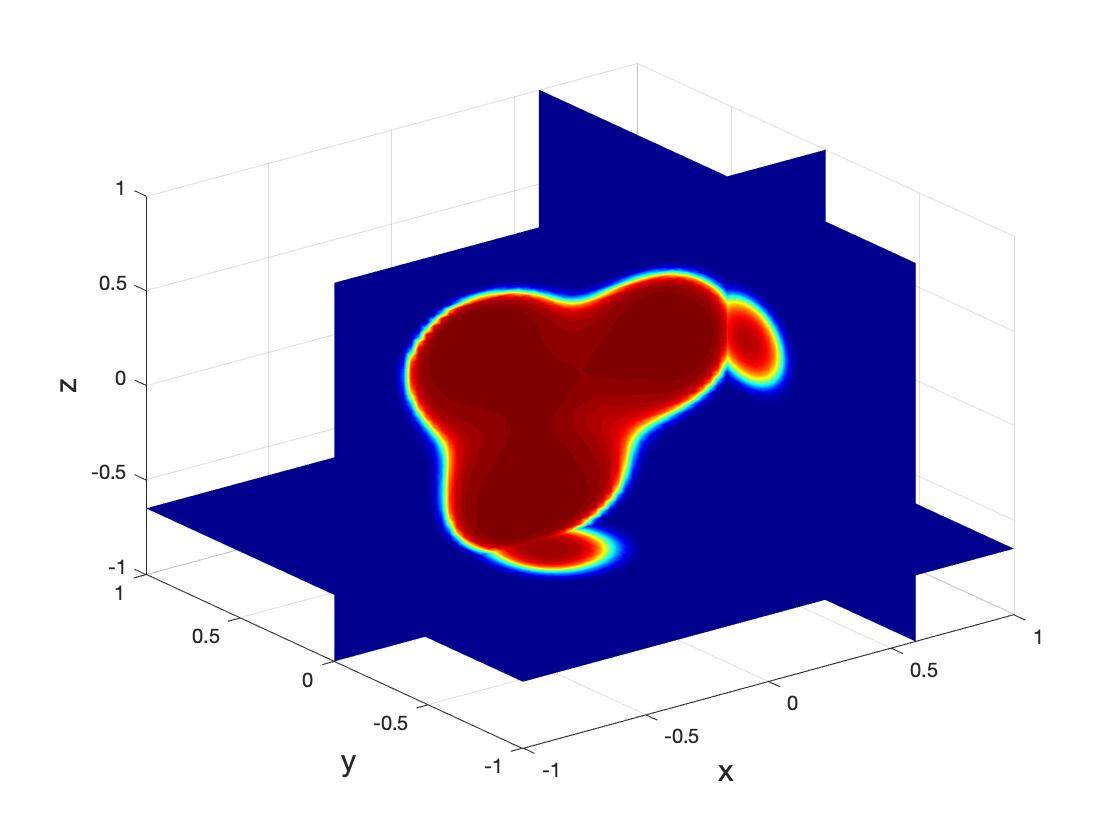}&
	\includegraphics[width=1.3in]{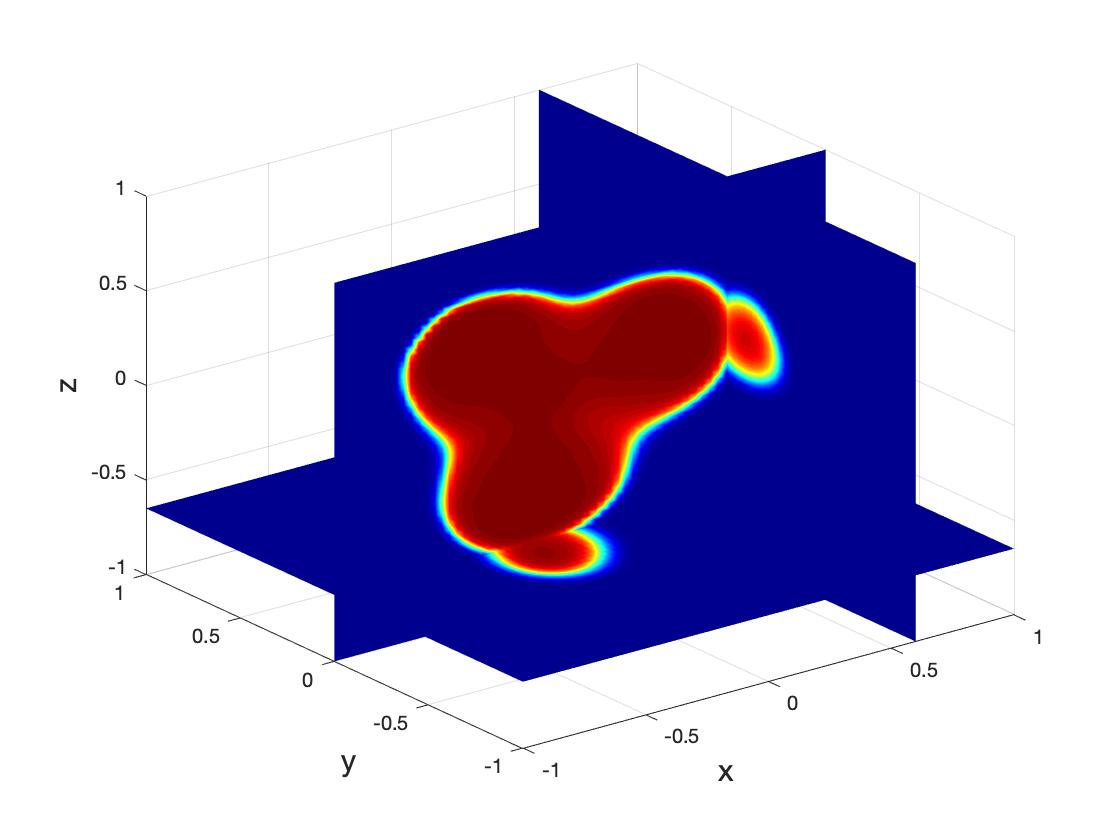}&
	\includegraphics[width=1.3in]{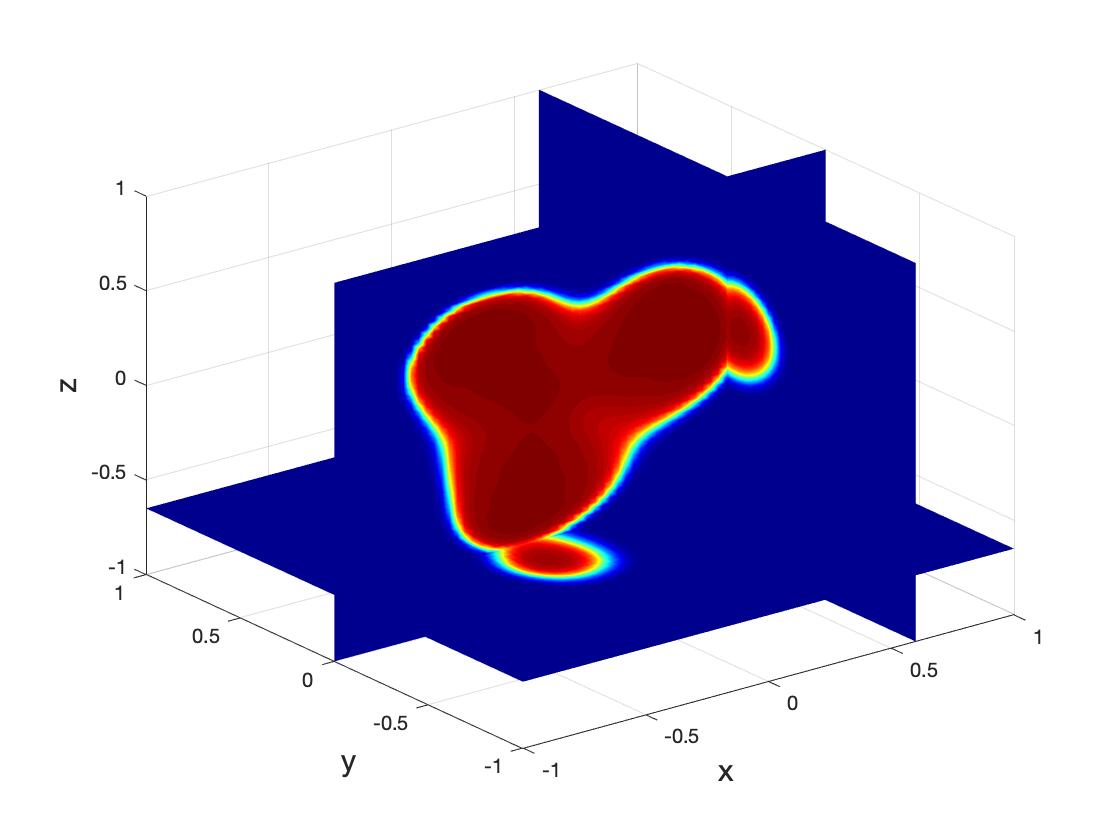}\\
%	\includegraphics[width=0.8in]{test-cir-20pair-noise-5perc-relu-mse-sgd-04-real}&
%	\includegraphics[width=0.8in]{test-cir-20pair-noise-10perc-relu-mse-sgd-04-real}\\
%	%%
%	\includegraphics[width=0.8in]{true-cir-07}&
%	\includegraphics[width=0.8in]{test-cir-1pair-noise-0perc-relu-mse-sgd-07}&
%	\includegraphics[width=0.8in]{test-cir-10pair-noise-0perc-relu-mse-sgd-07}&
%	\includegraphics[width=0.8in]{test-cir-20pair-noise-0perc-relu-mse-sgd-07}&
%	\includegraphics[width=0.8in]{test-cir-20pair-noise-5perc-relu-mse-sgd-07-real}&
%	\includegraphics[width=0.8in]{test-cir-20pair-noise-10perc-relu-mse-sgd-07-real}\\
%	%%
%	\includegraphics[width=0.8in]{true-cir-16}&
%	\includegraphics[width=0.8in]{test-cir-1pair-noise-0perc-relu-mse-sgd-16}&
%	\includegraphics[width=0.8in]{test-cir-10pair-noise-0perc-relu-mse-sgd-16}&
%	\includegraphics[width=0.8in]{test-cir-20pair-noise-0perc-relu-mse-sgd-16}&
%	\includegraphics[width=0.8in]{test-cir-20pair-noise-5perc-relu-mse-sgd-16-real}&
%	\includegraphics[width=0.8in]{test-cir-20pair-noise-10perc-relu-mse-sgd-16-real}\\
	%%
\end{tabular}
\caption{Reconstruction for a duck-shaped inclusion with different noise level.} 
\label{3D_duck}
\end{figure}

%\begin{figure}[htbp]
%\begin{tabular}{ >{\centering\arraybackslash}m{1.3in}>{\centering\arraybackslash}m{1.3in} >{\centering\arraybackslash}m{1.3in}  >{\centering\arraybackslash}m{1.3in}   }
%	\centering
%	True coefficients &
%          $\delta=0$ &
%          $\delta=5\%$ &
%          $\delta=10\%$ \\
%	\includegraphics[width=1.3in]{3D_case1_true}&
%	\includegraphics[width=1.3in]{3D_case1_N0_mu01100}&
%	\includegraphics[width=1.3in]{3D_case1_N5_mu01100}&
%	\includegraphics[width=1.3in]{3D_case1_N10_mu01100}\\
%         \includegraphics[width=1.3in]{I_true_1SPL_3d_2ellipsoid_mu_01_100_present_01_100_201_01}&
%	\includegraphics[width=1.3in]{predict-1SPL-2ellipsoid-noise-0perc-mu-01-100-present-5000-0038-01-02-100}&
%	\includegraphics[width=1.3in]{predict-1SPL-2ellipsoid-noise-5perc-mu-01-100-present-5000-0038-01-02-100}&
%	\includegraphics[width=1.3in]{predict-1SPL-2ellipsoid-noise-10perc-mu-01-100-present-5000-0038-01-02-100}\\
%	%%
%\end{tabular}
%\caption{Reconstruction for different $\mu$: $\mu=(0.1,100)$} 
%\label{3D_ell_mu}
%\end{figure}

%x
%
%x
%
%x
%
%x
%
%
%x
%
%x
%
%x
%
%x
%
%x
%
%x
%
%x
%
%x
%
%x
%
%x
%
%x
%
%x
%
%x
%
%x
%
%
%x
%
%x
%
%x
%
%x
%
%x
%
%x
%
%x
%
%x
%
%x
%
%x

\section{Conclusions:}
\label{sec:conclusions}
Inspired by the DSM \cite{chow2015direct},  this paper proposes a novel deep direct sampling method for the DOT problem, where the neural network architecture is designed to approximate the index functional. Hybridizing the DSM and CNN has several advantages. First, the index functional approximated by the CNN is capable of systematically incorporating multiple Cauchy data pairs, which can improve the reconstruction quality and the algorithmic robustness against the noise. Second, once the neural network is well trained, the data-driven index functional can be executed efficiently, which inherits the main benefit of the conventional DSM. Third, the DDSM can successfully handle challenging cases such as limited data and 3D reconstruction problems, which shows great potential in its practical application. Various numerical experiments justify these findings. Hence, we believe the proposed DDSM provides an efficient technique and a new promising direction for solving the DOT problem. \\

\textbf{Acknowledgements:} The first author and the second author were funded by the startup funding from ShanghaiTech University ( No. 2020F0203-000-16). The third author was funded by NSF DMS-2012465.

%\bibliographystyle{abbrv}
%
%\bibliography{EITbib}

\end{document}